\documentclass[a4paper,9pt,notitlepage,twoside]{amsart}

\usepackage[british]{babel}
\usepackage{lmodern}
\usepackage[T1]{fontenc}
\usepackage{textcomp}

\usepackage[ps2pdf,verbose]{hyperref}

\newtheorem{theorem}{Theorem}
\newtheorem{lemma}[theorem]{Lemma}

\newtheorem{remark}[theorem]{Remark}
\newtheorem*{remark*}{Remark}

\theoremstyle{definition}

\newcommand*\nz[1][]{\ensuremath{\mathbb{N}^{#1}}}

\newcommand*\rz[1][]{\ensuremath{\mathbb{R}^{#1}}}

\newcommand*\lrabs[1]{\left|{#1}\right|}
\newcommand*\abs[1]{\lvert{#1}\rvert}
\newcommand*\bigabs[1]{\bigl\lvert{#1}\bigr\rvert}
\newcommand*\Bigabs[1]{\Bigl\lvert{#1}\Bigr\rvert}
\newcommand*\biggabs[1]{\biggl\lvert{#1}\biggr\rvert}

\newcommand*\norm[1]{\lVert{#1}\rVert}
\newcommand*\bignorm[1]{\bigl\lVert{#1}\bigr\rVert}
\newcommand*\widefrac[2]{\frac{\phantom{\lVert}#1\phantom{\lVert}}{\phantom{\lVert}#2\phantom{\lVert}}}
\newcommand*\argplaceholder{\ensuremath{\,\cdot\,}}
\newcommand*\stack[2]{\genfrac{}{}{0pt}{3}{#1}{#2}}
\newcommand*\myunderbrace[3]{\underbrace{{#1}}_{\phantom{{#3}}{#2}{#3}}}

\DeclareMathOperator{\diag}{diag}
\DeclareMathOperator{\sign}{sign}

\let\div\divergence
\DeclareMathOperator{\curl}{curl}

\newcommand*\ball{\ensuremath B}

\newcommand*\nball[1][n]{\ball^{#1}}

\newcommand*\scprod[1]{\langle{#1}\rangle}
\newcommand*\bigscprod[1]{\bigl\langle{#1}\bigr\rangle}
\DeclareMathOperator{\id}{id}

\newlength{\tmpwidth}\newlength{\tmpheight}\newlength{\tmpdepth}
\def\tmplengthsarg#1#2{\mbox{$#1 #2$}}
\def\settmplengths#1#2{%
  \settowidth{\tmpwidth}{\tmplengthsarg{#1}{#2}}%
  \settoheight{\tmpheight}{\tmplengthsarg{#1}{#2}}%
  \settodepth{\tmpdepth}{\tmplengthsarg{#1}{#2}}%
}
\def\settmplengthsdisplaystyle#1{\settmplengths{\displaystyle}{#1}}
\def\settmplengthstextstyle#1{\settmplengths{\textstyle}{#1}}
\def\settmplengthsscriptstyle#1{\settmplengths{\scriptstyle}{#1}}
\def\settmplengthsscriptscriptstyle#1{\settmplengths{\scriptscriptstyle}{#1}}

\def\makepostsup#1#2{{#1}^{#2}\hspace{-\tmpwidth}\phantom{#1}}
\newcommand{\postsup}[2]{%
  \mathchoice%
    {\settmplengthsdisplaystyle{#1}%
      \makepostsup{#1}{#2}}%
    {\settmplengthstextstyle{#1}%
      \makepostsup{#1}{#2}}%
    {\settmplengthsscriptstyle{#1}%
      \makepostsup{#1}{#2}}%
    {\settmplengthsscriptscriptstyle{#1}%
      \makepostsup{#1}{#2}}%
}
\def\makepostsub#1#2{{#1}_{#2}\hspace{-\tmpwidth}\phantom{#1}}
\newcommand{\postsub}[2]{%
  \mathchoice%
    {\settmplengthsdisplaystyle{#1}%
      \makepostsub{#1}{#2}}%
    {\settmplengthstextstyle{#1}%
      \makepostsub{#1}{#2}}%
    {\settmplengthsscriptstyle{#1}%
      \makepostsub{#1}{#2}}%
    {\settmplengthsscriptscriptstyle{#1}%
      \makepostsub{#1}{#2}}%
}
\def\makepresup#1#2{{}\hspace{-\tmpwidth}\phantom{#1}^{#2}\mkern-3.0mu}
\newcommand{\presup}[2]{%
  \mathchoice%
    {\settmplengthsdisplaystyle{#1}%
      \makepresup{#1}{#2}}%
    {\settmplengthstextstyle{#1}%
      \makepresup{#1}{#2}}%
    {\settmplengthsscriptstyle{#1}%
      \makepresup{#1}{#2}}%
    {\settmplengthsscriptscriptstyle{#1}%
      \makepresup{#1}{#2}}%
  #1%
}

\DeclareFontFamily{U}{mathx}{\hyphenchar\font45}
\DeclareFontShape{U}{mathx}{m}{n}{
      <5> <6> <7> <8> <9> <10>
      <10.95> <12> <14.4> <17.28> <20.74> <24.88>
      mathx10
      }{}
\DeclareSymbolFont{mathx}{U}{mathx}{m}{n}
\DeclareFontSubstitution{U}{mathx}{m}{n}
\DeclareMathAccent{\widecheck}{0}{mathx}{"71}

\newcommand{\ulIndex}[3]{#1^#2_{\phantom{#2}#3}}
\newcommand{\ul}[1]{\ensuremath\ulIndex#1}

\newcommand{\cullIndex}[4]{#1^{#2}_{#3#4}}
\newcommand{\cull}[1]{\ensuremath\cullIndex#1}

\makeatletter
\newcommand{\biggg}{\bBigg@{3.5}} 
\newcommand{\Biggg}{\bBigg@{4}}
\newcommand{\bigggg}{\bBigg@{4.5}}
\newcommand{\Bigggg}{\bBigg@{5}}
\makeatother
\newcommand{\bigggl}{\mathopen\biggg}

\newcommand{\bigggr}{\mathclose\biggg}

\begin{document}

\def\mytitle{{On the formation of shocks of electromagnetic plane waves in non--linear crystals}}
\def\myshorttitle{\mytitle}

\hypersetup{pdftitle={D. Christodoulou and D. R. Perez, On the formation of shocks of electromagnetic plane waves in non-linear crystals}} 

\thanks{This work was funded by the European Research Council (ERC Advanced Grant 246574, Partial Differential Equations of Classical Physics)}

\expandafter\title[\myshorttitle]\mytitle

\author[D.~Christodoulou]{Demetrios Christodoulou}
\author[D.~R.~Perez]{Daniel Raoul Perez}
\address[Demetrios Christodoulou]{Department Mathematik\\ETH Z\"urich\\R\"amistrasse 101\\CH-8092 Zurich\\Switzerland}
\email{demetrios.christodoulou@math.ethz.ch}
\address[Daniel Raoul Perez]{Department Mathematik\\ETH Z\"urich\\R\"amistrasse 101\\CH-8092 Zurich\\Switzerland}
\email{daniel.perez@math.ethz.ch}


\begin{abstract}
An influential result of F.~John states that no genuinely non--linear strictly hyperbolic quasi--linear first order system of partial differential equations in two variables has a global $C^2$--solution for small enough initial data. Inspired by recent work of D.~Christodoulou, we revisit John's original proof and extract a more precise description of the behaviour of solutions at the time of shock. We show that John's singular first order quantity, when expressed in characteristic coordinates, remains bounded until the final time, which is then characterised by an inverse density of characteristics tending to zero in one point. Moreover, we study the derivatives of second order, showing again their boundedness when expressed in appropriate coordinates. We also recover John's upper bound for the time of shock formation and complement it with a lower bound. Finally, we apply these results to electromagnetic plane waves in a crystal with no magnetic properties and cubic electric non--linearity in the energy density, assuming no dispersion.
\end{abstract}

\maketitle

\section{Introduction} \label{sec:introduction}

Electromagnetic plane waves of a single polarisation can form shocks in non--linear media \cite[\S111]{landaulifshitz84}. If both polarisations are present, one may appeal to \cite{john74} in the anisotropic case. A solution to the general three--dimensional problem of electromagnetic waves in non--linear media, even in the homogeneous and isotropic setting, is still out of reach. A model case, however, was recently dealt with by Miao and Yu \cite{miaoyu14}.

In this paper, we focus on electromagnetic plane waves in non--linear crystals. These are solutions of certain quasi--linear hyperbolic systems of first order partial differential equations, which fall in the framework of Fritz John's treatment \cite{john74}. However, the description in \cite{john74} is somewhat incomplete, since only the non--existence of smooth global solutions is shown. The purpose of the present work is to overcome this limitation. Its key feature is a more precise description of the behaviour of smooth solutions at their final time of existence. It turns out that their singular behaviour, as mentioned by John, is a consequence of the singular relationship between the physical space--time coordinates and coordinates adapted to the characteristics. Indeed, all relevant quantities remain bounded when expressed in appropriate coordinates adapted to the characteristics, and a shock is characterised by an inverse density of characteristics decreasing to zero somewhere. This is then similar to the situation in \cite{christodoulou07}, where the formation of shocks in relativistic compressible fluids is considered (but see also the more recent and self--contained exposition of the non--relativistic case \cite{christodouloumiao14}). Those shocks are characterised by the foliation into outgoing characteristic null--hypersurfaces becoming infinitely dense somewhere. At the same time, the solution can be smoothly extended to the time of shock when expressed in acoustical coordinates. This description served as an inspiration for the work at hand.

Our results are not limited to the four--by--four system obtained when considering both polarisations of electromagnetic plane waves in non--linear crystals, but they apply to the general framework set up in \cite{john74}, i.e., genuinely non--linear strictly hyperbolic quasi--linear first order systems of partial differential equations in one space variable and one time variable. One calls a system strictly hyperbolic when the characteristic speeds are distinct, whereas such a system is said to be genuinely non--linear, if certain directional derivatives of the characteristic speeds do not vanish. As in \cite{john74}, a key ingredient is the assumption of smooth, compactly supported and \emph{small} non--zero initial data, leading to the long--time behaviour of solutions as essentially simple waves, i.e., maps with one--dimensional range. This is expressed by the fact that the problem reduces to a system of non--linear ordinary differential equations along each characteristic, obtained by an appropriate decomposition of the spatial derivatives of the solution vector. Calling \emph{characteristic strip} the collection of characteristics of a certain speed emanating from the support of the initial data, John's result states that solutions have to remain $C^1$--bounded outside those strips. On the other hand, John shows a blow--up in finite time of a component of the first spatial derivative of a solution inside at least one of those strips (although solutions themselves remain bounded). His argument can be modified slightly to get a deeper insight into the failure of regularity. By expressing the solution in characteristic coordinates rather than the Euclidean ones, the same component of the first spatial derivative remains bounded in every strip. This is due to the fact that, in those coordinates, the spatial derivative corresponds to the partial derivative with respect to Euclidean space multiplied by the inverse density of characteristics of the strip at hand. This inverse density acts as a regularising factor. The final time is then characterised by this quantity going to zero somewhere, which, therefore, implies John's result.

The first part of our work can, in fact, be extracted directly from John's proof. John introduces the inverse density of characteristics, as well as the regularised component of the first order spatial derivative of the solution. However, he uses both those quantities only in an integral sense, thereby hiding their actual behaviour at the final time. This is why we choose to re--expose completely his proof. We also use slightly different assumptions on the initial data. While he requires continuous second derivative, we only need Lipschitz--continuity of the first derivative. Even if this is not an important modification, we feel that it is more natural, as the initial system is only of first order.

In a second step, we extend the above result by considering the derivatives of second order of solutions. It is not hard to see that, again, outside every characteristic strip, all quantities of interest remain bounded. Inside such a strip, however, the situation becomes more complicated, as the quantities relating to the strip at hand depend too strongly on quantities related to other strips. The problem is that we cannot bound certain quantities when integrating along the appropriate characteristics, while inside a characteristic strip of a different speed. This is overcome by expressing the relevant quantities in bi--characteristic coordinates, where time as a parameter for a characteristic is replaced by the spatial characteristic coordinate of a transversally intersecting characteristic of a different speed.

We end this paper with an application of the theory to electromagnetic plane waves inside crystals with no magnetic properties and cubic electric non--linearity in the energy density, neglecting effects of dispersion. One should point out here, though, that physical experiments would always exhibit dispersion which will have the tendency to counteract shock formation. Accepting that a description of the full three--dimensional problem in crystals is still far away, solving the present problem seems like a good complement to a potential three--dimensional theory for isotropic media. Indeed, one would probably have to focus the beam in the latter case in order to counteract dispersion, as the non--linearity is only of fourth order in the Lagrangian.

The present work is organised as follows. In Section~\ref{sec:preliminaries}, we introduce the necessary notation in order to state our result. We explain characteristics and the coordinates related to them. We also expose some first order variations and formulate the technical assumption for hyperbolic systems to be genuinely non--linear. Section~\ref{sec:firstorder} is devoted to a re--exposition of John's proof. We begin by obtaining the evolution equations along characteristics of first order quantities. After presenting some a--priori assumptions on the initial data, we recover John's bounds on those quantities using his method. We end the section by analysing more in depth the behaviour near the final time of the inverse density of characteristics. In doing so, we establish upper and lower bounds on the final time (but see also \cite{hoermander97} for existence, uniqueness and lifespan of solutions). Moreover, we show a lower bound on the inverse density of characteristics in certain regions, including parts of its relative strip. In Section~\ref{sec:secondorder}, we prove boundedness of the second order quantities. Using characteristic coordinates for a particular characteristic strip, the respective quantities also depend on second order quantities related to other characteristic speeds. We deal with the latter by expressing them in bi--characteristic coordinates. It turns out that they form a complete regular system that can thus be bounded. The bounds on the second order quantities related to the strip at hand are then obtained by a Gronwall--type argument for systems. We end this paper with Section~\ref{sec:application}, in which we consider electromagnetic plane waves inside crystals with no magnetic properties and cubic electric non--linearity in the energy density, assuming absence of dispersion. After quickly recalling the relevant equations, we start by exposing under what conditions the theory can actually be applied. We then point out that the simplest two--polarisation case actually decouples and can be dealt with using Riemann invariants. In contrast, as we subsequently show, no Riemann invariants exist for generic cubic non--linear energy densities. As an actual experimental setup would require to deal with a material boundary, we end this section with giving a sketch of that situation.

\section{Preliminaries} \label{sec:preliminaries}

\subsection{The Cauchy problem} \label{subsec:cauchyproblem}

Let $N\in\nz\setminus\{0\}$ and $U\subset\rz[N]$ be open with $0\in U$. Let $K\in\nz\setminus\{0\}$ and $a\in C^{\max\{2,K\}}(U;\rz[N\times N])$ be such that the eigenvalues $\lambda_i$ ($i\in\{1,\ldots,N\}$) of $a$ fulfil $\lambda_N<\ldots<\lambda_1$ in $U$, i.e., we assume that the first order quasi--linear system 
\begin{equation} \label{eq:main_pde_plain}
  \frac{\partial}{\partial t}u+a(u)\frac{\partial}{\partial x}u=0
\end{equation}
is strictly hyperbolic around the trivial solution $u\equiv0$. Take $\delta>0$ small enough so that the open $N$--ball $\nball[N]_{2\delta}(0)$ of radius $2\delta$ around $0\in\rz[N]$ is completely contained in $U$, i.e., 
\begin{equation} \label{eq:2deltaball_in_U}
  \nball[N]_{2\delta}(0)\subset U .
\end{equation}
For each $i\in\{1,\ldots,N\}$, define 
\begin{align} \label{eq:supinflambdai}
  \overline\lambda_i &= \sup_{v\in\nball[N]_{2\delta}(0)} \lambda_i(v) , & \underline\lambda_i &= \inf_{v\in\nball[N]_{2\delta}(0)} \lambda_i(v) ,
\end{align}
and
\begin{equation} \label{eq:sigma}
  \sigma = \min_{\stack{k,l}{k<l}}(\underline\lambda_k - \overline\lambda_l) .
\end{equation}
We will assume $\delta$ to be small enough so that
\begin{equation} \label{eq:sigma_pos}
  \sigma > 0 ,
\end{equation}
which is possible, since $a$ is continuous. This means that the system \eqref{eq:main_pde_plain} is \emph{uniformly} strictly hyperbolic in $\nball[N]_{2\delta}(0)$.

For a general $N$--dimensional vector space $\mathcal V$ endowed with a basis $\{E_i\}_{i\in\{1,\ldots,N\}}$, we denote by $\{\postsup{E}{\star}^i\}_{i\in\{1,\ldots,N\}}$ the basis of the dual space $\mathcal V^\star$ of $\mathcal V$ which is dual to $\{E_i\}_{i\in\{1,\ldots,N\}}$, i.e., we have 
\begin{equation*}
  \postsup{E}{\star}^i E_j = \ul{\delta ij} , \qquad \forall i,j\in\{1,\ldots,N\} .
\end{equation*}
We can then expand any vector $V\in\mathcal V$ and any covector $\eta\in\mathcal V^\star$ in the bases $\{E_i\}_{i\in\{1,\ldots,N\}}$ and $\{\postsup{E}{\star}^i\}_{i\in\{1,\ldots,N\}}$, respectively, as follows: 
\begin{equation*}
  V = \sum_i \bigl( \postsup{E}{\star}^i V \bigr) E_i , \qquad \eta = \sum_i \bigl( \eta E_i \bigr) \postsup{E}{\star}^i .
\end{equation*}
We define the scalar product $\scprod{\argplaceholder,\argplaceholder}_{E}$ on $\mathcal V$ by 
\begin{equation*}
  \scprod{V,W}_{E} = \sum_k \bigl(\postsup{E}{\star}^kV\bigr)\bigl(\postsup{E}{\star}^kW\bigr) , \qquad V,W\in\mathcal V .
\end{equation*}
Then 
\begin{equation*}
  \scprod{E_i,E_j}_{E} = \sum_k \ul{\delta ki}\ul{\delta kj} = \delta_{ij} , \qquad i,j\in\{1,\ldots,N\} , 
\end{equation*}
so that $\{E_i\}_{i\in\{1,\ldots,N\}}$ is an orthonormal basis of $\mathcal V$ with respect to $\scprod{\argplaceholder,\argplaceholder}_{E}$. Let $\iota_E$ denote the canonical isomorphism of $\mathcal V^\star$ onto $\mathcal V$ induced by $\scprod{\argplaceholder,\argplaceholder}_{E}$. It is defined as follows. Given a covector $\eta\in\mathcal V^\star$, the element $\iota_E(\eta)\in\mathcal V$ is the unique vector which fulfils 
\begin{equation*}
  \bigscprod{\iota_E(\eta),V}_E = \eta V ,\qquad \forall V\in\mathcal V .
\end{equation*}
It follows that $\iota_E$ maps $\postsup{E}{\star}^i$ to $E_i$, $i\in\{1,\ldots,N\}$. Also, the isomorphism $\iota_E$ induces a scalar product $\scprod{\argplaceholder,\argplaceholder}_{E^\star}$ on $\mathcal V^\star$, defined by 
\begin{equation*}
  \scprod{\eta,\xi}_{E^\star} = \bigscprod{\iota_E(\eta),\iota_E(\xi)}_{E} , \qquad \eta,\xi\in\mathcal V^\star ,
\end{equation*}
relative to which the basis $\{\postsup{E}{\star}^i\}_{i\in\{1,\ldots,N\}}$ is orthonormal: 
\begin{equation*}
  \scprod{\postsup{E}{\star}^i,\postsup{E}{\star}^j}_{E^\star} = \bigscprod{\iota_E({\postsup{E}{\star}^i}),\iota_E({\postsup{E}{\star}^j})}_{E} = \scprod{E_i,E_j}_{E} = \delta^{ij} , \qquad i,j\in\{1,\ldots,N\} .
\end{equation*}

Returning to the system \eqref{eq:main_pde_plain}, let $\postsub{e}{0}_i=e_i(0)$ denote an eigenvector of $a(0)$ corresponding to the eigenvalue $\lambda_i(0)$, $i\in\{1,\ldots,N\}$. Since all the eigenvalues of $a$ are distinct, $\{\postsub{e}{0}_i\}_{i\in\{1,\ldots,N\}}$ forms a basis of $\rz[N]$. Denoting by $\{\postsup{\postsub{e}{0}}{\star}^i\}_{i\in\{1,\ldots,N\}}$ its dual basis, we define the scalar product $\scprod{\argplaceholder,\argplaceholder}_0$ by
\begin{equation} \label{eq:ei0_scprod}
  \scprod{v,w}_0 = \sum_k \bigl(\postsup{\postsub{e}{0}}{\star}^kv\bigr)\bigl(\postsup{\postsub{e}{0}}{\star}^kw\bigr) , \qquad v,w\in\rz[N] .
\end{equation}
With $\iota_0$ the isomorphism of $\postsup{\rz[N]}{\star}$ onto $\rz[N]$ induced by $\scprod{\argplaceholder,\argplaceholder}_0$, the inner product $\scprod{\argplaceholder,\argplaceholder}_{0^\star}$ induced by $\iota_0$ on $\postsup{\rz[N]}{\star}$ is given by 
\begin{equation*}
  \scprod{\eta,\xi}_{0^\star} = \bigscprod{\iota_0(\eta),\iota_0(\xi)}_0 , \qquad \eta,\xi\in\postsup{\rz[N]}{\star} .
\end{equation*}
We then obtain the induced norms $\norm{\argplaceholder}_0$ and $\norm{\argplaceholder}_{0^\star}$ on $\rz[N]$ and $\postsup{\rz[N]}{\star}$, respectively, from 
\begin{equation} \label{eqs:ei0_norms}
  \norm{v}_0 = \sqrt{\scprod{v,v}_0} , \quad v\in\rz[N] , \qquad\text{ and }\qquad \norm{\eta}_{0^\star} = \sqrt{\scprod{\eta,\eta}_{0^\star}} , \quad \eta\in\postsup{\rz[N]}{\star} .
\end{equation}

Now, for $i\in\{1,\ldots,N\}$ and $u\in\nball[n]_{2\delta}(0)$, let $e_i=e_i(u)$ be the eigenvector of $a(u)$ corresponding to the eigenvalue $\lambda_i(u)$ and fulfilling the normalisation condition 
\begin{subequations} \label{eqs:eiestari_normalisation_and_duality_conditions}
  \begin{equation} \label{eq:ei_normalisation_condition}
    \norm{e_i}_0 = 1 , \qquad i\in\{1,\ldots,N\} .
  \end{equation}
  By the strict hyperbolicity of \eqref{eq:main_pde_plain}, the $e_i$ form a basis of $\rz[N]$. Let $\{\postsup{e}{\star}^i\}_{i\in\{1,\ldots,N\}}$ denote the basis dual to $\{e_i\}_{i\in\{1,\ldots,N\}}$, so that 
  \begin{equation} \label{eq:eiestari_duality_condition}
    \postsup{e}{\star}^i e_j = \ul{\delta ij} , \qquad i,j \in \{1,\ldots,N\} .
  \end{equation}
\end{subequations}
Observe that $e_i$, and thus $\postsup{e}{\star}^i$, are determined uniquely up to sign by \eqref{eq:ei_normalisation_condition}. Moreover, viewing $a$ as a linear map acting on the dual $\postsup{\rz[N]}{\star}$ of $\rz[N]$, we have 
\begin{equation} \label{eq:estaria}
  \postsup{e}{\star}^i a = \lambda_i\postsup{e}{\star}^i , \qquad i\in\{1,\ldots,N\} .
\end{equation}
Notice, as well, that with the matrix $a$, also $e_i$, $\postsup{e}{\star}^i$ and $\lambda_i$ are $C^{\max\{2,K\}}$.

\begin{remark*}
  Defining the $(N\times N)$--matrices $G$ and $G_0$ by $G_{ij}=e_i\cdot e_j$ and $\postsub{G}{0}_{ij}=\scprod{e_i,e_j}_0$, $i,j\in\{1,\ldots,N\}$, respectively, $\postsup{e}{\star}^i$, $i\in\{1,\ldots,N\}$ acts on vectors $v\in\rz[N]$ by 
  \begin{equation*}
    \postsup{e}{\star}^i v = \sum_j (G^{-1})^{ij} e_j\cdot v = \sum_j (\postsub{G}{0}^{-1})^{ij}\scprod{e_j,v}_0 .
  \end{equation*}
  Here, ``$\cdot$'' denotes the Euclidean inner product in $\rz[N]$. %
  The normalisation condition \eqref{eq:ei_normalisation_condition} states that $\postsub{G}{0}_{ii}=1$, $i\in\{1,\ldots,N\}$. This is different from \cite{john74}, where John sets $(G^{-1})^{ii}=1$, $i\in\{1,\ldots,N\}$.
\end{remark*}

\begin{remark} \label{rem:geometric_interpretation}
  For a better geometric understanding, consider the following. Think of $u$ as a mapping $u:\mathcal M\to\mathcal N$, where $\mathcal M=\rz\times[0,T]$ is (a strip in) the $(x,t)$--plane and $\mathcal N$ is an $N$--dimensional manifold. Then, given $(x_0,t_0)\in\mathcal M$ and writing $y_0=u(x_0,t_0)\in\mathcal N$, the vectors $\frac{\partial}{\partial x}u(x_0,t_0)$ and $\frac{\partial}{\partial t}u(x_0,t_0)$ in $T_{y_0}\mathcal N$ are tangent to the images under $u$ of the coordinate lines $\{t=t_0\}$ and $\{x=x_0\}$, respectively, in the $(x,t)$--plane. Moreover, $a(y_0)$ is a map of $T_{y_0}\mathcal N$ into itself, and the equation \eqref{eq:main_pde_plain} at $(x_0,t_0)$ gives a relation between the tangent vectors $\frac{\partial}{\partial x}u(x_0,t_0)$ and $\frac{\partial}{\partial t}u(x_0,t_0)$:
  \begin{equation*}
    \Bigl( \frac{\partial}{\partial t} u(x_0,t_0)\Bigr) + a\bigl(u(x_0,t_0)\bigr)\Bigl(\frac{\partial}{\partial x}u(x_0,t_0)\Bigr) = 0 .
  \end{equation*}
  A smooth vectorfield $v$ on $\mathcal N$ is an eigen--vectorfield of $a$, if $v(y)$ is an eigenvector of $a(y)$ at each $y\in\mathcal N$. So there is a smooth function $\lambda$ on $\mathcal N$ such that 
  \begin{equation*}
    av = \lambda v .
  \end{equation*}
  Note that the condition that the eigenvalues $\lambda$ of $a$ be real and distinct implies that the manifold $\mathcal N$ is parallelisable.

  At any arbitrary $y\in\mathcal N$, we can also consider $a(y)$ as a linear map from $T_y^\star\mathcal N$ into itself by:
  \begin{equation*}
    \zeta \mapsto \zeta a(y) , \qquad \forall \zeta\in T_y^\star\mathcal N ,
  \end{equation*}
  where 
  \begin{equation*}
    \bigl( \zeta a(y) \bigr)w = \zeta\bigl( a(y)w \bigr) , \qquad \forall w\in T_y\mathcal N .
  \end{equation*}
  A smooth $1$--form $\xi$ on $\mathcal N$ is an eigen--$1$--form of $a$, if $\xi(y)$ is an eigen--covector of $a(y)$ at each $y\in\mathcal N$. So there is a smooth function $\lambda'$ on $\mathcal N$ such that 
  \begin{equation*}
    \xi a = \lambda' \xi .
  \end{equation*}
  Assuming all the eigenvalues $\lambda_i$, $i\in\{1,\ldots,N\}$, of $a$ are distinct, and given a basis $\{e_i\}_{i\in\{1,\ldots,N\}}$ of eigen--vectorfields, the dual basis $\{\postsup{e}{\star}^i\}_{i\in\{1,\ldots,N\}}$ of $1$--forms fulfils 
  \begin{equation*}
    \postsup{e}{\star}^i(y)e_j(y) = \ul{\delta ij} , \qquad \forall i,j\in\{1,\ldots,N\},\ \forall y\in\mathcal N .
  \end{equation*}
  It follows that $\postsup{e}{\star}^i$ is an eigen--$1$--form of $a$ with the same eigenvalue function $\lambda_i$ as $e_i$.
\end{remark}

Now consider a non--zero ``seed'' function $f_0\in C_c^{K,1}\bigl(\rz;\nball[N]_\delta(0)\bigr)$ with support in the interval $I_0=[-1,1]$. Here, $C_c^{K,1}$ is the space of compactly supported $K$ times differentiable functions with Lipschitz continuous $K^{\text{th}}$ derivative.  For $\vartheta>0$, let 
\begin{equation} \label{eq:f}
  f = \vartheta f_0 .
\end{equation}
In what follows, we want to study solutions $u\in C^K\bigl(\rz\times[0,t_\star);\nball[N]_\delta(0)\bigr)$ of the Cauchy problem 
\begin{equation} \label{eq:main_cp}
  \left\{ \begin{aligned}
    \frac{\partial}{\partial t} u(x,t) + a\bigl(u(x,t)\bigr)\frac{\partial}{\partial x} u(x,t) &= 0 ,&\qquad (x,t)&\in\rz\times[0,t_\star) , \\[2.5pt]
    u(z,0) &= f(z) ,&\qquad z&\in\rz ,
  \end{aligned} \right.
\end{equation}
on their maximal time--slab of existence $\rz\times[0,t_\star)$. Our focus will lie on their behaviour at the final time $t_\star$.

\subsection{\texorpdfstring{Characteristics (cf.~\cite{john74})}{Characteristics}} \label{subsec:characteristics}

Consider a solution $u\in C^{K}\bigl(\rz\times[0,T];\nball[N]_\delta(0)\bigr)$ of \eqref{eq:main_cp} for some $T>0$. For $z\in\rz$ and $i\in\{1,\ldots,N\}$, let $\mathcal C_i(z)$ denote the $i^{\text{th}}$ characteristic starting at $z$, i.e., the set of points 
\begin{equation} \label{eq:Ci}
  \mathcal C_i(z) = \Bigl\{ \bigl(X_i(z,t),t\bigr)\in\rz\times[0,T] \Bigm| t\in[0,T] \Bigr\} ,
\end{equation}
where $X_i(z,t)$ is a solution of
\begin{equation} \label{eq:Xi_ivp}
  \left\{ \begin{aligned}
      \frac{\partial}{\partial t} X_i(z,t) &= \lambda_i\Bigl(u\bigl(X_i(z,t),t\bigr)\Bigr) , \\[0.5pt] 
      X_i(z,0) &= z .
  \end{aligned} \right.
\end{equation}
Note that, since $\lambda_i,u\in C^{K}$, the theory of ordinary differential equations tells us that $X_i$ is unique and itself $C^{K}$. Therefore, 
\begin{equation} \label{eq:Xiz1z2_ineq}
  X_i(z_1,t) < X_i(z_2,t) , \qquad\text{ for all $t\in[0,T]$, whenever $z_1<z_2$, }
\end{equation}
so that 
\begin{equation} \label{eq:Ri_width}
  X_i(-1,t) \leq X_i(z,t) \leq X_i(1,t) , \qquad\text{ for all $z\in I_0=[-1,1]$, $t\in[0,T]$. }
\end{equation}
Notice that $u$ is supported in the region delimited by $t=0$, $t=T$, and the extreme characteristics $\mathcal C_N(-1)$ and $\mathcal C_1(1)$, which are actually straight lines:
\begin{equation} \label{eq:extreme_chars}
  X_N(-1,t) = -1 + \lambda_N(0)t , \qquad X_1(1,t) = 1 + \lambda_1(0)t .
\end{equation}

We now show that our assumption \eqref{eq:sigma_pos} ensures that the sets $\mathcal R_i=\mathcal C_i(I_0)=\bigcup_{z\in I_0}\mathcal C_i(z)$, which we shall call \emph{characteristic strips}, are ``separated'' after a short time. Indeed, from \eqref{eq:Xi_ivp}, \eqref{eq:supinflambdai}, we have, for all $i\in\{0,\ldots,N\}$, $z\in I_0$, $t\in[0,T]$, 
\begin{equation} \label{eq:Xi_extremes}
  z + \underline\lambda_i t \leq X_i(z,t) \leq z + \overline\lambda_i t .
\end{equation}
But then, using \eqref{eq:sigma}, we have for $k<l$,
\begin{equation*}
  X_k(-1,t) - X_l(1,t) \geq \bigl( -1 + \underline\lambda_k t \bigr) - \bigl( 1 + \overline\lambda_l t \bigr) \geq -2 + ( \underline\lambda_k-\overline\lambda_l)t \geq -2 + \sigma t ,
\end{equation*}
which, by \eqref{eq:sigma_pos}, is strictly positive as soon as $t$ exceeds 
\begin{equation} \label{eq:t0}
  t_0 = \frac2\sigma .
\end{equation}
As a result, 
\begin{equation} \label{eq:RkRl_separated}
  \bigl(\mathcal R_k\bigcap \mathcal R_l\bigr)\bigcap\bigl\{(x,t)\bigm|t\in[t_0,T]\bigr\} = \emptyset , \qquad \forall k\neq l .
\end{equation}

Next, we observe that \eqref{eq:Xiz1z2_ineq} implies that we can introduce a new set of coordinates on $\rz\times[0,T]$.

\subsection{Characteristic coordinates} \label{subsec:char_coords}

Let $i\in\{1,\ldots,N\}$ be given. Since, by \eqref{eq:Xiz1z2_ineq}, $X_i(z_1,t)<X_i(z_2,t)$, whenever $z_1<z_2$, and since $X_i\in C^{K}$, we can introduce the following coordinates $(z_i,s_i)$ on $\rz\times[0,T]$: Given $(x,t)\in\rz\times[0,T]$, there is a unique $(z_i,s_i)\in\rz\times[0,T]$ such that
\begin{equation*}
  (x,t) = \bigl( X_i(z_i,s_i), s_i \bigr) .
\end{equation*}

Introducing the function (cf.~\cite{john74})
\begin{equation} \label{eq:rhoi}
  \rho_i(z_i,s_i) = \frac{\partial}{\partial z_i} X_i(z_i,s_i) 
\end{equation}
on $\rz\times[0,T]$, and recalling from \eqref{eq:Xi_ivp} that 
\begin{equation*}
  \lambda_i(z_i,s_i) = \frac{\partial}{\partial s_i} X_i(z_i,s_i) ,
\end{equation*}
we obtain
\begin{equation} \label{eqs:dccoords}
  \left\{ \begin{aligned}
    \frac{\partial}{\partial z_i} &= \rho_i\frac{\partial}{\partial x} , \\
    \frac{\partial}{\partial s_i} &= \lambda_i\frac{\partial}{\partial x} + \frac{\partial}{\partial t} ,
  \end{aligned} \right.
  \qquad\text{ and }\qquad
  \left\{ \begin{aligned}
    dx &= \rho_i dz_i + \lambda_i ds_i , \\
    dt &= ds_i .
  \end{aligned} \right.
\end{equation}

Note that $\rho_i$ represents the inverse density of the $i^{\text{th}}$ characteristics, whereas $\lambda_i$ represents their speed. It will turn out that $\rho_i$ acts as a regularising factor for $\frac{\partial}{\partial x}u$ inside the characteristic strip $\mathcal R_i$. The formation of ``shock'' will then be characterised by $\rho_i$ tending to zero.

\subsection{Bi--characteristic coordinates} \label{subsec:bichar_coords}

Observe that, by our assumption \eqref{eq:sigma_pos}, the intersection between an $i^{\text{th}}$ and a $j^{\text{th}}$ characteristic ($i\neq j$) has to be transversal. Fixing then $i,j\in\{1,\ldots,N\}$, $i\neq j$, we can eliminate $s_i$ in favour of $z_j$ as a coordinate, $z_j$ acting as a parameter along the characteristic $\mathcal C_i(z_i)$ (and likewise for $s_j$). The coordinates thus obtained on $\rz\times[0,T]$ are described as follows: Given $(x,t)\in\rz\times[0,T]$, there is a unique $(y_i,y_j)\in\rz[2]$ such that 
\begin{equation*}
  (x,t) = \Bigl( X_i\bigl( y_i, t'(y_i,y_j) \bigr), t'(y_i,y_j) \Bigr) = \Bigl( X_j\bigl( y_j, t'(y_i,y_j) \bigr), t'(y_i,y_j) \Bigr) ,
\end{equation*}
for some $C^K$--function $t'$ of $y_i$ and $y_j$. We compute 
\begin{align*}
  \frac{\partial}{\partial y_i} &= \Bigl( \rho_i + \lambda_i\frac{\partial}{\partial y_i}t' \Bigr)\frac{\partial}{\partial x} + \Bigl(\frac{\partial}{\partial y_i}t'\Bigr)\frac{\partial}{\partial t} = \lambda_j\Bigl(\frac{\partial}{\partial y_i}t'\Bigr)\frac{\partial}{\partial x} + \Bigl(\frac{\partial}{\partial y_i}t'\Bigr)\frac{\partial}{\partial t} , \\
  \frac{\partial}{\partial y_j} &= \lambda_i\Bigl(\frac{\partial}{\partial y_j}t'\Bigr)\frac{\partial}{\partial x} + \Bigl(\frac{\partial}{\partial y_j}t'\Bigr)\frac{\partial}{\partial t} = \Bigl(\rho_j + \lambda_j\frac{\partial}{\partial y_j}t'\Bigr)\frac{\partial}{\partial x} + \Bigl(\frac{\partial}{\partial y_j}t'\Bigr)\frac{\partial}{\partial t} .
\end{align*}
It follows that 
\begin{align*}
  \frac{\partial}{\partial y_i}t' &= \frac{\rho_i}{\lambda_j-\lambda_i} , & \frac{\partial}{\partial y_j}t' &= \frac{\rho_j}{\lambda_i-\lambda_j} ,
\end{align*}
where the denominators are non--zero by \eqref{eq:sigma_pos}. Hence, using also \eqref{eqs:dccoords}, 
\begin{subequations} \label{eqs:dbccoords}
  \begin{equation} \label{eqs:dbccoords_partial}
    \left\{ \begin{aligned}
      \frac{\partial}{\partial y_i} &= \frac{\rho_i\lambda_j}{\lambda_j-\lambda_i}\frac{\partial}{\partial x} + \frac{\rho_i}{\lambda_j-\lambda_i}\frac{\partial}{\partial t} = \frac{\rho_i}{\lambda_j-\lambda_i}\frac{\partial}{\partial s_j} = \frac{\partial}{\partial z_i} + \frac{\rho_i}{\lambda_j-\lambda_i}\frac{\partial}{\partial s_i} , \\[2.5pt]
      \frac{\partial}{\partial y_j} &= \frac{\rho_j\lambda_i}{\lambda_i-\lambda_j}\frac{\partial}{\partial x} + \frac{\rho_j}{\lambda_i-\lambda_j}\frac{\partial}{\partial t} = \frac{\rho_j}{\lambda_i-\lambda_j}\frac{\partial}{\partial s_i} = \frac{\partial}{\partial z_j} + \frac{\rho_j}{\lambda_i-\lambda_j}\frac{\partial}{\partial s_j} ,
    \end{aligned} \right.
  \end{equation}
  and 
  \begin{equation} \label{eqs:dbccoords_d}
    \left\{ \begin{aligned}
      dx &= \frac{\rho_i\lambda_j}{\lambda_j-\lambda_i} dy_i + \frac{\rho_j\lambda_i}{\lambda_i-\lambda_j}dy_j , \\
      dt = ds_i = ds_j &= \frac{\rho_i}{\lambda_j-\lambda_i} dy_i + \frac{\rho_j}{\lambda_i-\lambda_j} dy_j , \\
      dz_i &= dy_i , \\
      dz_j &= dy_j .
    \end{aligned} \right.
  \end{equation}
\end{subequations}

Before stating our results, we introduce the technical assumption ultimately responsible for the formation of shocks.

\subsection{\texorpdfstring{The first variation of $a$ and genuine non--linearity (cf.~\cite{john74})}{The first variation of a and genuine non--linearity}} \label{subsec:genuine_nl}

So far, we only required our Cauchy problem \eqref{eq:main_cp} to fulfil \eqref{eq:sigma_pos}, i.e., we assumed the matrix function $a$ of $u$ to be strictly  hyperbolic around the trivial solution in such a way that the different characteristic speeds are uniformly bounded away from each other. In order to ensure the formation of shock, we require, in addition, that the system \eqref{eq:main_pde_plain} be \emph{genuinely non--linear} around the trivial solution, i.e., for each $i$, we assume that the change of $\lambda_i(0)$ in the direction $e_i(0)$ is non--zero.

Before giving the precise definition at the end of this subsection, we introduce the following coefficients to describe the first variation of $a$: 
\begin{equation} \label{eq:cjkl}
  \cull{cjkl} = \cull{cjkl}(u) = \postsup{e}{\star}^j(u) \bigr( D_{e_l(u)} a(u) \bigr) e_k(u) , \qquad j,k,l\in\{1,\ldots,N\} ,
\end{equation}
where 
\begin{equation} \label{eq:Delphi}
  D_{e_l(u)} \varphi(u) = d\varphi(u)e_l(u) = \lim_{\epsilon\to0} \frac1\epsilon \Bigl( \varphi\bigl(u+\epsilon e_l(u)\bigr)-\varphi(u) \Bigr)
\end{equation}
denotes the directional derivative of a (scalar--, vector-- or matrix--valued) function $\varphi$ of $u$ in the direction $e_l$. Note that, since $\bigl\{e_i(u)\bigr\}_{i\in\{1,\ldots,N\}}$ forms a basis of the tangent space $T_u\rz[N]$ at $u$, the coefficients defined above are sufficient to completely describe the first variation of $a$.

Now, since $e_k$ is an eigenvector of $a$, so that $( a - \lambda_k\id )e_k = 0$, we have 
\begin{equation*}
  ( a - \lambda_k\id )( D_{e_l}e_k ) = -( D_{e_l}a - D_{e_l}\lambda_k\id )e_k .
\end{equation*}
Thus, using \eqref{eq:estaria}, \eqref{eq:eiestari_duality_condition} and \eqref{eq:cjkl}, we get 
\begin{multline*}
  (\lambda_m-\lambda_k)\postsup{e}{\star}^m( D_{e_l}e_k ) = \postsup{e}{\star}^m( a - \lambda_k\id )( D_{e_l}e_k ) \\
    = -\postsup{e}{\star}^m( D_{e_l}a - D_{e_l}\lambda_k\id )e_k = -\cull{cmkl} + \ul{\delta mk}(D_{e_l}\lambda_k) .
\end{multline*}
Setting $m=k$, this yields 
\begin{equation} \label{eq:Dellambdak}
  D_{e_l}\lambda_k = \cull{ckkl} ,
\end{equation}
while for $m\neq k$ we get, 
\begin{equation*}
  \postsup{e}{\star}^m(D_{e_l}e_k) = \frac1{\lambda_k-\lambda_m}\cull{cmkl} .
\end{equation*}
Thus, expanding $D_{e_l}e_k$ in the basis $\{e_1,\ldots,e_N\}$, we have 
\begin{equation*}
  D_{e_l}e_k = \sum_m \bigl(\postsup{e}{\star}^m(D_{e_l}e_k)\bigr)e_m = \bigl(\postsup{e}{\star}^k(D_{e_l}e_k)\bigr)e_k + \sum_{\stack{m}{m\neq k}}\frac1{\lambda_k-\lambda_m}\cull{cmkl}e_m .
\end{equation*}
But, since $\norm{e_k}_0=1$, whence $\bigscprod{(D_{e_l}e_k),e_k}_0=\frac12 D_{e_l}\bigl(\norm{e_k}_0^2\bigr) = 0$, we obtain 
\begin{equation*}
  \postsup{e}{\star}^k(D_{e_l}e_k) = - \sum_{\stack{m}{m\neq k}} \frac1{\lambda_k-\lambda_m}\cull{cmkl}\scprod{e_m,e_k}_0 ,
\end{equation*}
from which we infer 
\begin{equation} \label{eq:Delek}
  D_{e_l}e_k = \sum_{\stack{m}{m\neq k}}\frac1{\lambda_m-\lambda_k}\cull{cmkl}\bigl( \scprod{e_m,e_k}_0e_k - e_m \bigr) .
\end{equation}

From the duality condition \eqref{eq:eiestari_duality_condition}, we get 
\begin{equation*}
  (D_{e_l}\postsup{e}{\star}^j)e_k = -\postsup{e}{\star}^j(D_{e_l} e_k) ,
\end{equation*}
and compute 
\begin{equation*}
  (D_{e_l}\postsup{e}{\star}^j)e_k = -\sum_{\stack{m}{m\neq k}}\frac1{\lambda_m-\lambda_k}\cull{cmkl}\bigl( \scprod{e_m,e_k}_0\ul{\delta jk} - \ul{\delta jm} \bigr) .
\end{equation*}
Thus, expanding $D_{e_l}\postsup{e}{\star}^j$ in the basis $\{\postsup{e}{\star}^1,\ldots,\postsup{e}{\star}^N\}$, it follows that 
\begin{align}
\nonumber
  D_{e_l}\postsup{e}{\star}^j &= \sum_{k}\bigl((D_{e_l}\postsup{e}{\star}^j)e_k\bigr)\postsup{e}{\star}^k = \sum_{\stack{k,m}{m\neq k}}\frac1{\lambda_k-\lambda_m}\cull{cmkl}\bigl( \scprod{e_m,e_k}_0\ul{\delta jk} - \ul{\delta jm} \bigr)\postsup{e}{\star}^k \\
\nonumber
    &= \sum_{\stack{m}{m\neq j}}\frac1{\lambda_j-\lambda_m}\cull{cmjl}\scprod{e_m,e_j}_0\postsup{e}{\star}^j - \sum_{\stack{k}{k\neq j}}\frac1{\lambda_k-\lambda_j}\cull{cjkl}\postsup{e}{\star}^k \\
 \label{eq:Delestarj}
    &= \sum_{\stack{k}{k\neq j}}\frac1{\lambda_j-\lambda_k}\bigl( \cull{ckjl}\scprod{e_k,e_j}_0\postsup{e}{\star}^j + \cull{cjkl}\postsup{e}{\star}^k \bigr) .
\end{align}

We are now ready to give the definition of \emph{genuine non--linearity} for the system \eqref{eq:main_pde_plain}: we simply require $\cull{ciii}(0)=D_{e_i}\lambda_i(0)\neq0$ (cf.~\eqref{eq:cjkl} and \eqref{eq:Dellambdak}).

By an appropriate choice of sign for each $e_i$ (and hence for $\postsup{e}{\star}^i$), we can without loss of generality assume that 
\begin{equation} \label{eq:gen_nl_cond}
  \cull{ciii}(0) < 0 , \qquad \forall i\in\{1,\ldots,N\} .
\end{equation}
Choosing then $\delta>0$ small enough, we can ensure by continuity that 
\begin{equation} \label{eq:gen_nl_cond_unif}
  \cull{ciii}(v) < 0 , \qquad \forall v\in\nball[N]_{2\delta}(0),\ \forall i\in\{1,\ldots,N\} .
\end{equation}

\subsection{Statement of results} \label{subsec:statement_of_results}

Before we can state our results, we need to introduce the following first and second order quantities related to a solution $u\in C^K\bigl(\rz\times[0,t_\star);\nball[N]_\delta(0)\bigr)$ of the Cauchy--problem \eqref{eq:main_cp} subject to the additional assumption \eqref{eq:gen_nl_cond} on its maximal time--slab $\rz\times[0,t_\star)$ of existence.

Fix $i\in\{1,\ldots,N\}$. For $k\in\{1,\ldots,N\}$, let 
\begin{equation} \label{eq:mt_wk}
  w^k(x,t) = \postsup{e}{\star}^k\bigl(u(x,t)\bigr) \frac{\partial}{\partial x} u(x,t) ,
\end{equation}
and, for the given $i$, let 
\begin{equation} \label{eq:mt_vi}
  v^i(z_i,s_i) = \postsup{e}{\star}^i\bigl(u(z_i,s_i)\bigr) \frac{\partial}{\partial z_i} u(z_i,s_i) = \rho_i(z_i,s_i)w^i(z_i,s_i) .
\end{equation}
Here, and in the following, we abuse notation and write
\begin{equation*}
  \text{$u(z_i,s_i)$ for $u\bigl(X_i(z_i,s_i),s_i\bigr)$, etc\ldots}
\end{equation*}
Further, define 
\begin{equation} \label{eqs:mt_muinui}
  \mu_i(z_i,s_i) = \frac{\partial}{\partial z_i}\rho_i(z_i,s_i) \qquad\text{ and }\qquad \nu^i(z_i,s_i) = \frac{\partial}{\partial z_i}v^i(z_i,s_i) .
\end{equation}
Finally, fix $j\in\{1,\ldots,N\}$, $j\neq i$, and define 
\begin{equation} \label{eqs:mt_tauijomegaij}
  \presup{\tau}{(i)}_j(y_i,y_j) = \frac{\partial}{\partial y_j}\rho_j(y_i,y_j) \qquad\text{ and }\qquad \presup{\omega}{(i)}^j(y_i,y_j) = \frac{\partial}{\partial y_j}v^j(y_i,y_j) .
\end{equation}

\begin{theorem} \label{thm:main_thm}
  Let $N,K\in\nz\setminus\{0\}$, let $U\subset\rz[N]$ be an open neighbourhood of $0$, and let $a\in C^{\max\{2,K\}}(U;\rz[N\times N])$ be such that the system 
\begin{equation*}
\tag{\ref{eq:main_pde_plain}}
  \frac{\partial}{\partial t}u+a(u)\frac{\partial}{\partial x}u=0   
\end{equation*}
is strictly hyperbolic and genuinely non--linear. Let $\delta>0$ be sufficiently small, so that conditions 
\begin{equation*} 
\tag{\ref{eq:2deltaball_in_U}}
  \nball[N]_{2\delta}(0)\subset U ,
\end{equation*}
\begin{equation*}
\tag{\ref{eq:sigma_pos}}
  \sigma > 0 
\end{equation*}
and
\begin{equation*}
\tag{\ref{eq:gen_nl_cond_unif}}
  \cull{ciii}(v) < 0 , \qquad \forall v\in\nball[N]_{2\delta}(0), i\in\{1,\ldots,N\} .
\end{equation*}
hold. Let $f_0\in C_c^{K,1}\bigl(\rz;\nball[N]_\delta(0)\bigr)$ have support in $I_0=[-1,1]$. Finally, let $\varepsilon\in\bigl(0,\frac1{100}\bigr)$ be a parameter.

Then there is a constant $\vartheta_0>0$, depending solely on $a$, $\delta$, $\varepsilon$ and the $C^{K,1}$--norm of $f_0$, such that the following statements hold for any solution $u\in C^K\bigl(\rz\times[0,t_\star);\nball[N]_\delta(0)\bigr)$ of 
\begin{equation*}
\tag{\ref{eq:main_cp}}
  \left\{ \begin{aligned}
    \frac{\partial}{\partial t} u(x,t) + a\bigl(u(x,t)\bigr)\frac{\partial}{\partial x} u(x,t) &= 0 ,&\qquad (x,t)&\in\rz\times[0,t_\star) , \\[2.5pt]
    u(z,0) &= f(z) ,&\qquad z&\in\rz ,
  \end{aligned} \right.
\end{equation*}
with initial data $f=\vartheta f_0$ and given on its maximal time--slab of existence $\rz\times[0,t_\star)$, whenever $0<\vartheta<\vartheta_0$.
  \begin{enumerate} \renewcommand{\theenumi}{\roman{enumi}}
    \item \label{thm:main_thm_firstorder} (cf.~\cite{john74}) For every $i\in\{1,\ldots,N\}$, we have: $\rho_i$ and $v^i$ are bounded everywhere. Moreover, $\rho_i(z_i,t)>0$ for all $z_i\in\rz$, $t<t_\star$. Also, $w_i$ is bounded outside the strip $\mathcal R_i$.
    \item \label{thm:main_thm_tstar} Let $W_0^+ = \sup\limits_{\stack{k\in\{1,\ldots,N\}}{z\in\rz}}w^k(z,0) . $ Then
       \begin{equation*}
         \min_i\frac1{(1+\varepsilon)^3\abs{\cull{ciii}(0)}W_0^+} \leq t_\star\leq \max_i\frac1{(1-\varepsilon)^4\abs{\cull{ciii}(0)}W_0^+} .
       \end{equation*}
      Moreover, there is an index $i^+\in\{1,\ldots,N\}$ and an initial location $z^+\in\rz$ such that $\lim_{t\nearrow t_\star}\rho_{i^+}(z^+,t) = 0$.
    \item \label{thm:main_thm_secondorder} If $K\geq2$, we have for every $i\in\{1,\ldots,N\}$: $\mu_i$ and $\nu^i$ remain bounded inside $\mathcal R_i$. Moreover, for every $j\in\{1,\ldots,N\}\setminus\{i\}$, $\presup{\tau}{(i)}_j$ and $\presup{\omega}{(i)}^j$ remain bounded inside $\mathcal R_i$. Also, $\frac{\partial}{\partial x} w^i$ is bounded outside all the strips, i.e., in $\bigl(\rz\times[0,t_\star)\bigr)\setminus\bigcup_{k\in\{1,\ldots,N\}}\mathcal R_k$.
  \end{enumerate}
\end{theorem}

\subsection{Some notation} \label{subsec:notation}

In the following and similarly to \cite{john74}, we shall write $A=\mathcal O(B)$ for two quantities $A$ and $B$ depending on the initial data $f=\vartheta f_0$, whenever there exists a constant $C$, depending only on the matrix function $v\mapsto a(v)$, on the size $\delta>0$ of the $N$--ball on which our system \eqref{eq:main_pde_plain} is uniformly strictly hyperbolic, in the sense that \eqref{eq:sigma_pos} holds, as well as genuinely non--linear, in the sense that \eqref{eq:gen_nl_cond_unif} holds, and on the $C^{K,1}$--norm of the seed $f_0$, such that 
\begin{equation*}
  \bignorm{A(f)} \leq C\bignorm{B(f)} ,
\end{equation*}
for all $f$ with $\vartheta>0$ small enough. Here, the meaning of $\norm{\argplaceholder}$ depends on the quantity under consideration. It is the usual absolute value when applied to scalars. If the quantities are vectors, we use the Euclidean norm $\abs{\argplaceholder}$, the norm $\norm{\argplaceholder}_0$ or the norm $\norm{\argplaceholder}_{0^\star}$ (cf.~\eqref{eqs:ei0_norms}), in dependence of whether the quantity lives in $\nball[N]_{\delta}(0)\subset\rz[N]$, in the tangent space $T_u\rz[N]$ or in the co--tangent space $T_u^\star\rz[N]$, respectively.

Also, for the rest of the paper, let $\varepsilon\in\bigl(0,\frac1{100}\bigr)$ be fixed. We will use it as a parameter to tweak our estimates so as not to get too big deviations, the idea being that $\varepsilon$ is very close to zero (assuming $\varepsilon<1/100$ is enough, but we can obtain better approximations if we choose it even smaller).

\subsection{A reminder on the Gronwall lemma} \label{subsec:Gronwall}

For the sake of completeness, we recall the precise statement of the Gronwall lemma in the form in which we are going to apply it. We do this to avoid stronger regularity conditions or additional assumptions on the signs sometimes encountered in the literature. Note also that the lemma remains valid with reversed inequalities.
\begin{lemma} \label{lem:G}
  Let $T>0$ and let $\eta$ be an absolutely continuous, real--valued function on $[0,T)$. Also, let $\psi$ and $\varphi$ be two locally integrable, real--valued functions on $[0,T)$. Assume that 
  \begin{equation} \label{eq:G_diffineq}
    \frac{d}{dt}\eta(t) \leq \psi(t)\eta(t) + \varphi(t) 
  \end{equation}
  is satisfied for almost every $t\in[0,T)$. Then the following inequality holds for \emph{all} $t\in[0,T)$
  \begin{equation} \label{eq:G_ineq}
    \eta(t) \leq e^{\int_0^t\psi(r)dr}\biggl( \eta(0) + \int_0^t e^{-\int_0^s\psi(r)dr}\varphi(s)ds \biggr) .
  \end{equation}
\end{lemma}
\begin{proof}
  By the fundamental theorem of calculus, we have $\frac{d}{ds}\int_0^s\psi(r)dr=\psi(s)$ for almost every $s\in[0,T)$, since $\psi$ is locally integrable. Therefore, for almost every $s\in[0,T)$, 
  \begin{equation*}
    \frac{d}{ds}\biggl( \eta(s)e^{-\int_0^s\psi(r)dr} \biggr) = e^{-\int_0^s\psi(r)dr}\Bigl( \frac{d}{ds}\eta(s) + \psi(s)\eta(s) \Bigr) \stackrel{\eqref{eq:G_diffineq}}\leq e^{-\int_0^s\psi(r)dr}\varphi(s) .
  \end{equation*}
  Integrating this from $0$ to $t\in[0,T)$, we obtain for \emph{every} $t\in[0,T)$ that 
  \begin{equation*}
    \eta(t)e^{-\int_0^t\psi(r)dr} - \eta(0) \leq \int_0^t e^{-\int_0^s\psi(r)dr}\varphi(s)ds ,
  \end{equation*}
  from which \eqref{eq:G_ineq} is immediate.
\end{proof}

We deduce the following result for a special case of two coupled integral inequalities.
\begin{lemma} \label{lem:sG}
  Let $T>0$ and let $\eta_0$, $\eta_1$ be two non--negative, locally integrable functions on $[0,T)$. Assume further that there exist non--negative constants $A_0$, $A_1$, $B_0$, $B_1$, as well as two real--valued, locally integrable functions $\varphi_0$, $\varphi_1$ on $[0,T)$, such that 
  \begin{equation} \label{eqs:sG_intineqs}
    \eta_i(t) \leq A_i\int_0^t\eta_i(s)ds + B_i\int_0^t\eta_{1-i}(s)ds + \varphi_i(t) ,\qquad i\in\{0,1\} ,
  \end{equation}
  is satisfied for almost every $t\in[0,T)$. Then there are constants $c_i^j$, $d$, $f^j$, $g_i$ ($i,j\in\{0,1\}$), depending only on $A_0$, $A_1$, $B_0$, $B_1$ and $T$, such that the following inequalities hold for every $t\in[0,T)$ and every $i\in\{0,1\}$
  \begin{equation} \label{eqs:sG_ineqs}
    \eta_i(t) \leq \sum_{j=0}^1 c_i^j\Biggl( d \int_0^t e^{d(t-s)} \int_0^s e^{f^j(s-r)} \varphi_j(r) dr\,ds + \int_0^t e^{g_i(t-s)} \varphi_j(s) ds \Biggr) + \varphi_i(t) .
  \end{equation}
  In particular, if $\varphi_0$ and $\varphi_1$ are bounded on $[0,T)$, then so are $\eta_0$ and $\eta_1$.
\end{lemma}
\begin{proof}
  For $t\in[0,T)$ and $i\in\{0,1\}$, let 
  \begin{equation*}
    \widetilde\eta_i(t) = \int_0^t\eta_i(s)ds .
  \end{equation*}
  Then $\widetilde\eta_i$ is absolutely continuous, hence almost everywhere differentiable, as well as locally integrable. Moreover, $\widetilde\eta_i$ is non--negative, non--decreasing, and we have $\widetilde\eta_i(0)=0$. From \eqref{eqs:sG_intineqs}, we see that for almost every $t\in[0,T)$, 
  \begin{equation*}
    \frac{d}{dt}\widetilde\eta_i(t) \leq A_i\widetilde\eta_i(t) + B_i\widetilde\eta_{1-i}(t) + \varphi_i(t) .
  \end{equation*}
  Then we obtain from Lemma~\ref{lem:G} that, for every $t\in[0,T)$ and $i\in\{0,1\}$, 
  \begin{equation*}
    \widetilde\eta_i(t) \leq \int_0^t e^{A_i(t-s)}\bigl(B_i\widetilde\eta_{1-i}(s)+\varphi_i(s)\bigr) ds \leq e^{A_iT}B_i\int_0^t\widetilde\eta_{1-i}(s)ds + \int_0^te^{A_i(t-s)}\varphi_i(s)ds .
  \end{equation*}
  Letting 
  \begin{equation*}
    \widehat B_i = e^{A_iT}B_i ,\qquad \widehat\varphi_i(t) = \int_0^te^{A_i(t-s)}\varphi_i(s)ds ,
  \end{equation*}
  we can write this as 
  \begin{equation} \label{eqs:sG_etatilde_firstestimates}
    \widetilde\eta_i(t) \leq \widehat B_i\int_0^t\widetilde\eta_{1-i}(s)ds + \widehat\varphi_i(t) ,\qquad i\in\{0,1\} .
  \end{equation}
  By monotonicity of $\widetilde\eta_{1-i}$, 
  \begin{equation*}
    \int_0^t\widetilde\eta_{1-i}(s)ds \leq T\widetilde\eta_{1-i}(t) ,
  \end{equation*}
  whence 
  \begin{equation*}
    \widetilde\eta_i(t) \leq \widehat B_iT\widetilde\eta_{1-i}(t) + \widehat\varphi_i(t) .
  \end{equation*}
  Inserting \eqref{eqs:sG_etatilde_firstestimates} for $\widetilde\eta_{1-i}(t)$ into this, we obtain the following inequalities that hold for every $t\in[0,T)$ and $i\in\{0,1\}$
  \begin{equation*}
    \widetilde\eta_i(t) \leq \widehat B_i\widehat B_{1-i} T\int_0^t\widetilde\eta_i(s)ds + \widehat\varphi_i(t) + \widehat B_iT\widehat\varphi_{1-i}(t) .
  \end{equation*}
  Setting 
  \begin{equation*}
    \widetilde A = \widehat B_0\widehat B_1 T ,\qquad \widetilde\varphi_i(t) = \widehat\varphi_i(t) + \widehat B_iT\widehat\varphi_{1-i}(t) ,
  \end{equation*}
  this reads
  \begin{equation} \label{eqs:sG_etatilde_estimates}
    \widetilde\eta_i(t) \leq \widetilde A \int_0^t \widetilde\eta_i(s)ds + \widetilde\varphi_i(t) ,\qquad i\in\{0,1\} .
  \end{equation}
  Defining for $i\in\{0,1\}$ the absolutely continuous, non--negative, non--decreasing functions 
  \begin{equation*}
    H_i(t) = \int_0^t\widetilde\eta_i(s)ds ,
  \end{equation*}
  we obtain for almost every $t\in[0,T)$ 
  \begin{equation*}
    \frac{d}{dt}H_i(t) \leq \widetilde A H_i(t) + \widetilde\varphi_i(t) .
  \end{equation*}
  Since $H_i(0)=0$, we obtain from Lemma~\ref{lem:G} for every $t\in[0,T)$ and $i\in\{0,1\}$ that 
  \begin{equation*}
    H_i(t) \leq \int_0^t e^{\widetilde A(t-s)}\widetilde\varphi_i(s)ds .
  \end{equation*}
  Since, by \eqref{eqs:sG_intineqs}, \eqref{eqs:sG_etatilde_estimates}, we have for $i\in\{0,1\}$, 
  \begin{equation*}
    \eta_i(t) \leq A_i\bigl( \widetilde A H_i(t) + \widetilde\varphi_i(t) \bigr) + B_i\bigl( \widetilde A H_{1-i}(t) + \widetilde\varphi_{1-i}(t) \bigr) + \varphi_i(t) ,
  \end{equation*}
  we obtain the desired result with the constants 
  \begin{equation*}
    \begin{split}
      c_i^i &= A_i+B_0B_1Te^{A_{1-i}T} , \\
      c_i^{1-i} &= B_i+A_iB_iTe^{A_iT} , \\
      d &= e^{(A_0+A_1)T}B_0B_1T , \\
      f^i &= A_i , \\
      g_i &= A_i ,
    \end{split}
  \end{equation*}
  where $i\in\{0,1\}$.
\end{proof}

\section{John's bounds on the first order quantities and the inverse density of characteristics} \label{sec:firstorder}

\subsection{\texorpdfstring{The evolution equations for the first order quantities~(cf.~\cite{john74})}{The evolution equations for the first order quantities}} \label{subsec:1storder_evoleqs}

In this subsection, we compute the evolution equations of some first order quantities along an $i^{\text{th}}$ characteristic.

So fix $i\in\{1,\ldots,N\}$. For $k\in\{1,\ldots,N\}$, let (see \eqref{eq:mt_wk}, \eqref{eq:mt_vi})
\begin{align}
\label{eq:wk}
  w^k &= \postsup{e}{\star}^k \frac{\partial}{\partial x}u , \\
\intertext{and, for the given $i$, let}
\label{eq:vi}
  v^i &= \postsup{e}{\star}^i\frac{\partial}{\partial z_i}u = \rho_iw^i .
\end{align}
Then, by \eqref{eq:wk}, 
\begin{align}
  \label{eq:dxu}
  \frac{\partial}{\partial x}u &= \sum_k \Bigl(\postsup{e}{\star}^k\frac{\partial}{\partial x}u\Bigr)e_k = \sum_k w^k e_k , \\
\intertext{and so, by \eqref{eqs:dccoords}, \eqref{eq:vi}, }
  \label{eq:dziu}
  \frac{\partial}{\partial z_i}u &= \rho_i\sum_kw^k e_k = v^i e_i + \rho_i\sum_{\stack k{k\neq i}}w^k e_k . 
\end{align}

From \eqref{eq:main_pde_plain} and \eqref{eqs:dccoords} we have
\begin{equation} \label{eq:dsiu_intrinsic}
  \frac{\partial}{\partial s_i} u = \lambda_i\frac{\partial}{\partial x}u + \frac{\partial}{\partial t}u = \bigl( \lambda_i\id -a \bigr)\frac{\partial}{\partial x} u ,
\end{equation}
whence, by \eqref{eq:dxu}, 
\begin{equation} \label{eq:dsiu}
  \frac{\partial}{\partial s_i}u = \sum_k \bigl( \lambda_i-\lambda_k \bigr) w^k e_k .
\end{equation}

It then follows from \eqref{eq:cjkl} that 
\begin{align}
\label{eq:estarjdsiaek}
  \postsup{e}{\star}^j\Bigl(\frac{\partial}{\partial s_i}a\Bigr)e_k &= \sum_l \postsup{e}{\star}^j(D_{e_l}a)\Bigl(\postsup{e}{\star}^l\frac{\partial}{\partial s_i}u\Bigr) e_k = \sum_l (\lambda_i-\lambda_l)\cull{cjkl}w^l , \\
\intertext{while, using \eqref{eq:dxu},}
\label{eq:estarjdxaek}
  \postsup{e}{\star}^j\Bigl(\frac{\partial}{\partial x} a\Bigr)e_k &= \sum_l \postsup{e}{\star}^j(D_{e_l}a)\Bigl(\postsup{e}{\star}^l\frac{\partial}{\partial x}u\Bigr) e_k = \sum_l \cull{cjkl}w^l .
\end{align}

Similarly, from \eqref{eq:Dellambdak}, we get, using \eqref{eq:dsiu}, respectively \eqref{eq:dxu}, \eqref{eq:dziu},
\begin{align}
\label{eq:dsilambdaj}
  \frac{\partial}{\partial s_i}\lambda_j &= \sum_k (\lambda_i-\lambda_k)\cull{cjjk}w^k , \\
\label{eq:dxlambdaj}
  \frac{\partial}{\partial x}\lambda_j &= \sum_k \cull{cjjk}w^k , \\
\label{eq:dzilambdaj}
  \frac{\partial}{\partial z_i}\lambda_j &= \rho_i\sum_k \cull{cjjk}w^k = \cull{cjji}v^i + \rho_i\sum_{\stack k{k\neq i}}\cull{cjjk}w^k.
\end{align}

In the same way, we recover with \eqref{eq:dsiu} from \eqref{eq:Delestarj} that 
\begin{equation}
\label{eq:dsiestarj}
  \frac{\partial}{\partial s_i}\postsup{e}{\star}^j = \sum_{\stack{k,l}{k\neq j}} \frac{\lambda_i-\lambda_l}{\lambda_j-\lambda_k}w^l\bigl(\cull{ckjl}\scprod{e_k,e_j}_0\postsup{e}{\star}^j+\cull{cjkl}\postsup{e}{\star}^k\bigr) . \\
\end{equation}

Finally, we get, using \eqref{eq:rhoi}, \eqref{eq:Xi_ivp}, \eqref{eq:dzilambdaj}, 
\begin{equation} \label{eq:dsirhoi}
  \frac{\partial}{\partial s_i}\rho_i = \frac{\partial}{\partial s_i}\Bigl(\frac{\partial}{\partial z_i}X_i\Bigr) = \frac{\partial}{\partial z_i}\Bigl(\frac{\partial}{\partial s_i}X_i\Bigr) = \frac{\partial}{\partial z_i}\lambda_i = \rho_i\sum_k \cull{ciik}w^k = \cull{ciii}v^i + \Bigl( \sum_{\stack m{m\neq i}}\cull{ciim}w^m \Bigr)\rho_i .
\end{equation}

These are actually not all the first order quantities. As we shall promptly see, $\frac{\partial}{\partial s_i}w^i$ and $\frac{\partial}{\partial s_i}v^i$ are as well of order one in the derivatives of $u$.

We assume in the following that $K\geq2$, although equations \eqref{eq:dsiwi}, \eqref{eq:dsiwi_fully_expanded}, \eqref{eq:dsivi} below remain valid for $K=1$. We start by considering $\frac{\partial}{\partial s_i}w^i$. We have from \eqref{eq:wk}
\begin{equation*}
  \frac{\partial}{\partial s_i}w^i = \Bigl(\frac{\partial}{\partial s_i}\postsup{e}{\star}^i\Bigr)\frac{\partial}{\partial x} u + \postsup{e}{\star}^i\biggl( \frac{\partial}{\partial s_i}\Bigl(\frac{\partial}{\partial x} u\Bigr) \biggr) .
\end{equation*}
But, from \eqref{eqs:dccoords}, and using \eqref{eq:main_pde_plain}, we have 
\begin{equation*}
  \frac{\partial}{\partial s_i}\Bigl(\frac{\partial}{\partial x} u\Bigr) = \lambda_i\frac{\partial^2}{\partial x^2} u + \frac{\partial^2}{\partial x\partial t} u = \lambda_i\frac{\partial^2}{\partial x^2} u - \frac{\partial}{\partial x}\Bigl( a\frac{\partial}{\partial x} u \Bigr) = \bigl( \lambda_i\id - a \bigr) \frac{\partial^2}{\partial x^2} u - \Bigl(\frac{\partial}{\partial x} a\Bigr)\frac{\partial}{\partial x} u ,
\end{equation*}
whence 
\begin{equation*}
  \postsup{e}{\star}^i\biggl(\frac{\partial}{\partial s_i}\Bigl(\frac{\partial}{\partial x}u\Bigr)\biggr) = - \postsup{e}{\star}^i\biggl(\Bigl(\frac{\partial}{\partial x}a\Bigr)\frac{\partial}{\partial x}u\biggr) .
\end{equation*}
Therefore, by \eqref{eq:dsiestarj}, as well as \eqref{eq:cjkl} and \eqref{eq:dxu},
\begin{align*}
  \frac{\partial}{\partial s_i}w^i &= \sum_{\stack{k,l}{k\neq i}} \frac{\lambda_i-\lambda_l}{\lambda_i-\lambda_k}w^l\bigl(\cull{ckil}\scprod{e_k,e_i}_0w^i+\cull{cikl}w^k\bigr) - \sum_{k,l}\cull{cikl}w^kw^l \\
  &= -\cull{ciii}(w^i)^2 - \biggl(\sum_{\stack l{l\neq i}} \Bigl( \cull{ciil} + \cull{cili} - \sum_{\stack k{k\neq i}} \frac{\lambda_i-\lambda_l}{\lambda_i-\lambda_k}\cull{ckil}\scprod{e_k,e_i}_0 \Bigr) w^l\biggr)w^i \\
  &\phantom{=\ } {} + \sum_{\stack{k,l}{k\neq i,l\neq i}} \Bigl( \myunderbrace{\frac{\lambda_i-\lambda_l}{\lambda_i-\lambda_k} - 1}={\makebox[0pt][l]{\tiny $\frac{\lambda_k-\lambda_l}{\lambda_i-\lambda_k}$}} \Bigr) \cull{cikl} w^kw^l
.
\end{align*}
Thus, defining the coefficients 
\begin{equation} \label{eq:gamma_sym}
  \cull{\gamma ilm}=\cull{\gamma ilm}(u)=\cull{\gamma iml}
\end{equation} as in \cite{john74} (notice, however, the difference in the last term in \eqref{eq:gammaiim} stemming from the fact that we used the normalisation condition 
$\norm{e_i}_0=1$, whereas John uses $\abs{\postsup{e}{\star}^i}=1$), 
\begin{subequations} \label{eqs:gamma}
\begin{align}
\label{eq:gammaiii}
  \cull{\gamma iii} &= -\cull{ciii} , \\
\label{eq:gammaiim}
  \cull{\gamma iim} &= \frac12 \Biggl( - \cull{cimi} - \cull{ciim} + \sum_{\stack l{l\neq i}} \frac{\lambda_i-\lambda_m}{\lambda_i-\lambda_l}\cull{clim}\scprod{e_l,e_i}_0 \Biggr) , && m\neq i , \\
\label{eq:gammailm}
  \cull{\gamma ilm} &= \frac12 \biggl( \frac{\lambda_l-\lambda_m}{\lambda_i-\lambda_l}\cull{cilm} + \frac{\lambda_m-\lambda_l}{\lambda_i-\lambda_m}\cull{ciml} \biggr) , && l\neq i, m\neq i ,
\end{align}
\end{subequations}
we can write 
\begin{equation} \label{eq:dsiwi}
  \frac{\partial}{\partial s_i}w^i = \sum_{l,m} \cull{\gamma ilm} w^lw^m .
\end{equation}
Noticing that, by \eqref{eq:gammailm}, 
\begin{equation} \label{eq:gammaill_is_zero}
  \cull{\gamma ill}=0 , \qquad\text{ if $l\neq i$,}
\end{equation}
and using \eqref{eq:gamma_sym}, \eqref{eq:gammaiii}, we can write this as
\begin{align}
\nonumber 
  \frac{\partial}{\partial s_i}w^i &= -\cull{ciii}(w^i)^2 + \sum_{\stack{l,m}{l\neq m}} \cull{\gamma ilm}w^lw^m \\
\label{eq:dsiwi_fully_expanded}
  &= -\cull{ciii}(w^i)^2 + \Bigl(\sum_{\stack m{m\neq i}}2\cull{\gamma iim}w^m\Bigr)w^i + \sum_{\stack{l,m}{\stack{l\neq i,m\neq i}{l\neq m}}} \cull{\gamma ilm}w^lw^m .
\end{align}
We take note that the right--hand side is indeed of first order in derivatives of $u$, as announced.

It is now easy to compute $\frac{\partial}{\partial s_i}v^i$. We obtain from \eqref{eq:vi}, \eqref{eq:dsirhoi} and \eqref{eq:dsiwi_fully_expanded} 
\begin{align*}
  \frac{\partial}{\partial s_i}v^i &= \frac{\partial}{\partial s_i}\bigl(\rho_iw^i\bigr) = \Bigl(\frac{\partial}{\partial s_i}\rho_i\Bigr)w^i + \rho_i\Bigl(\frac{\partial}{\partial s_i}w^i\Bigr) \\
    &= \cull{ciii}v^iw^i + \Bigl(\sum_{\stack m{m\neq i}}\cull{ciim}w^m\Bigr)\rho_iw^i \\
    &\phantom{=\ } {} - \rho_i\bigl(\cull{ciii}(w^i)^2\bigr) + \rho_i\Bigl(\sum_{\stack m{m\neq i}} 2\cull{\gamma iim}w^m\Bigr)w^i + \rho_i\Bigl(\sum_{\stack{l,m}{\stack{l\neq i,m\neq i}{l\neq m}}} \cull{\gamma ilm}w^lw^m\Bigr) ,
\end{align*}
which, by \eqref{eq:vi}, immediately simplifies to 
\begin{equation} \label{eq:dsivi}
  \frac{\partial}{\partial s_i}v^i = \Bigl(\sum_{\stack m{m\neq i}}\bigl(2\cull{\gamma iim}+\cull{ciim}\bigr)w^m\Bigr)v^i + \Bigl(\sum_{\stack{l,m}{\stack{l\neq i,m\neq i}{l\neq m}}} \cull{\gamma ilm}w^lw^m \Bigr) \rho_i .
\end{equation}
Again, we take note that the right--hand side is of first order.

\subsection{Some a--priori assumptions on the initial data} \label{subsec:initdataassumptions}

Seeking $C^{K}$--solutions $u$ of the Cauchy problem \eqref{eq:main_cp} with initial condition (cf.~\eqref{eq:f})
\begin{equation} \label{eq:main_ic}
  u(\argplaceholder,0) = f = \vartheta f_0 ,
\end{equation}
where $f_0\in C^{K,1}_c\bigl(\rz;\nball[N]_\delta(0)\bigr)$ has support in $I_0=[-1,1]$, we first have to restrict ourselves to $\vartheta\in(0,1)$ so that $f$ has range in $\nball[N]_\delta(0)$. Then, since we assumed that $a\in C^{\max\{2,K\}}$, the quantities $a$, $\lambda_i$, $\postsup{e}{\star}^i$, $e_i$, $\cull{cikl}$ are defined and at least $C^1$ composed with $f$. We have (cf.~\cite{john74})
\begin{equation} \label{eq:fdf_are_Otheta}
  f = \mathcal O(\vartheta) , \qquad f' = \frac{d}{dz}f = \mathcal O(\vartheta) ,
\end{equation}
and thus 
\begin{equation} \label{eq:lambdaiestariei_of_f_differ_by_Otheta}
  \lambda_i(f) - \lambda_i(0) = \mathcal O(\vartheta) , \quad \postsup{e}{\star}^i(f) - \postsup{e}{\star}^i(0) = \mathcal O(\vartheta) , \quad e_i(f) - e_i(0) = \mathcal O(\vartheta) ,
\end{equation}
while 
\begin{equation} \label{eq:lambdaiestariei_are_Oone}
  \lambda_i = \mathcal O(1) , \qquad \postsup{e}{\star}^i = \mathcal O(1) , \qquad e_i = \mathcal O(1) .
\end{equation}
We wish to stress here that we compute lengths of tangent (co--)vectors \emph{not} with respect to the Euclidean norm, but with respect to the appropriate norm induced by the scalar product $\scprod{\argplaceholder,\argplaceholder}_0$ with respect to which the basis $\{e_i(0)\}_{i\in\{1,\ldots,N\}}$ is orthonormal (cf.~\eqref{eq:ei0_scprod}). %
We further obtain that 
\begin{equation*}
  \cull{ciii}(f) - \cull{ciii}(0) = \mathcal O(f) = \mathcal O(\vartheta) ,
\end{equation*}
so that, by \eqref{eq:gen_nl_cond_unif}, it suffices to choose $\vartheta$ small enough so that
\begin{equation} \label{eq:ciii_stays_close}
  (1+\varepsilon)\cull{ciii}(0) < \cull{ciii}(f) < (1-\varepsilon)\cull{ciii}(0) < 0 .
\end{equation}
Finally, by \eqref{eq:wk}, \eqref{eq:fdf_are_Otheta} and \eqref{eq:lambdaiestariei_are_Oone}, we have for every $z\in\rz$, 
\begin{equation} \label{eq:order_wi_initially}
  w^i(z,0) = \postsup{e}{\star}^i f'(z) = \mathcal O(f') = \mathcal O(\vartheta) .
\end{equation}

Let us introduce now the following quantities:
\begin{align}
\label{eq:W0W0p}
  W_0 &= \max_i\sup_z \lrabs{w^i(z,0)} 
         , & W_0^+ &= \max_i\sup_z w^i(z,0) %
         , \\
\label{eq:W00W00p}
  W_{0,0} &= \max_i\sup_z \lrabs{ \postsup{e}{\star}^i(0)f_0'(z) } 
             , & W_{0,0}^+ &= \max_i\sup_z \bigl( \postsup{e}{\star}^i(0)f_0'(z) \bigr) %
             .
\end{align}
We trivially have $0\leq W_0^+\leq W_0$ and $0\leq W_{0,0}^+\leq W_{0,0}$. Moreover, \eqref{eq:order_wi_initially} implies that 
\begin{equation} \label{eq:W0_is_Otheta}
  W_0 = \mathcal O(\vartheta) .
\end{equation}
Also, we immediately infer from \eqref{eq:lambdaiestariei_are_Oone} (using in the first line also \eqref{eq:dxu}) that 
\begin{align}
\label{eq:df_is_OW0}
  f' &= \sum_i w^i(\argplaceholder,0) e_i(f) = \mathcal O(W_0) \\
\intertext{and}
\label{eq:df0_is_OW00}
  f_0' &= \sum_i \bigl(\postsup{e}{\star}^i(0)f_0'\bigr) e_i(0) = \mathcal O(W_{0,0}) .
\end{align}

Now, since $f_0$ is $C^{K}$ and has support in $I_0=[-1,1]$, we must have, for all $i\in\{1,\ldots,N\}$, 
\begin{equation} \label{eq:intestaridf0_is_zero}
  0 = \postsup{e}{\star}^i(0)\biggl( \underbrace{\int_{-1}^1 f_0'(z)dz}_{\phantom{0}=0} \biggr) = \int_{-1}^1 \bigl( \postsup{e}{\star}^i(0)f_0'(z) \bigr) dz .
\end{equation}
Therefore, if $f_0$ is non--constant (i.e., if $f_0$ is non--zero, since, by continuity, any other constant would contradict $f_0$ being supported in $I_0$), there must be an index $i\in\{1,\ldots,N\}$ and a point in $I_0$ where $\postsup{e}{\star}^i(0)f_0'$ is strictly positive. We conclude 
\begin{equation} \label{eq:W00p_pos}
  0 < W_{0,0}^+ \leq W_{0,0} .
\end{equation}
We now establish an explicit lower bound for $W_{0,0}^+$. By \eqref{eq:W00W00p}, there is an $i_0$ and a $z_0$ such that $\lrabs{\postsup{e}{\star}^{i_0}(0)f_0'(z_0)} = W_{0,0}$. If $\postsup{e}{\star}^{i_0}(0)f_0'(z_0)\geq0$, we obtain from \eqref{eq:W00W00p} that 
\begin{equation} \label{eq:W00p_lwr_bd_poscase}
  W_{0,0}^+ \geq \postsup{e}{\star}^{i_0}(0)f_0'(z_0) = W_{0,0} .
\end{equation}
So let us assume that $\postsup{e}{\star}^{i_0}(0)f_0'(z_0)<0$, and set 
\begin{equation} \label{eq:epsilon0}
  \epsilon_0 = \frac{W_{0,0}}L ,
\end{equation}
where 
\begin{equation} \label{eq:L}
  L = \sup_{\stack{z',z''\in I_0}{z'\neq z''}} \frac{\bignorm{f_0'(z')-f_0'(z'')}_0}{\abs{z'-z''}} 
\end{equation}
is the Lipschitz constant of $f_0'$. We then have, following \cite{john74}, 
\begin{align*}
  W_{0,0} &\ \ =\ \ \ -\postsup{e}{\star}^{i_0}(0)f_0'(z_0) = -\frac1{2\epsilon_0} \int_{\lrabs{z-z_0}<\epsilon_0} \postsup{e}{\star}^{i_0}(0)f_0'(z_0) dz \\
  &\ \ =\ \ \ \frac1{2\epsilon_0} \int_{\lrabs{z-z_0}<\epsilon_0} \postsup{e}{\star}^{i_0}(0)\bigl(f_0'(z)-f_0'(z_0)\bigr) dz + \frac1{2\epsilon_0}\int_{\lrabs{z-z_0}\geq\epsilon_0} \postsup{e}{\star}^{i_0}(0)f_0'(z) dz \\
  &\ \ \phantom{=\ }\ \ {} - \frac1{2\epsilon_0} \underbrace{\int_{I_0}\postsup{e}{\star}^{i_0}(0)f_0'(z) dz}_{\phantom0\stackrel{\eqref{eq:intestaridf0_is_zero}}=0} \\
  &\ \ \stackrel{\makebox[0pt][c]{\scriptsize \eqref{eq:L},\eqref{eq:W00W00p}}}{\leq\ }\ \ \frac{2L}{2\epsilon_0}\underbrace{\bignorm{\postsup{e}{\star}^{i_0}(0)}_0}_{\phantom1=1}\int_0^{\epsilon_0} y dy + \frac1{2\epsilon_0}\int_{I_0} W_{0,0}^+ dz \stackrel{\eqref{eq:epsilon0}}\leq \frac12W_{0,0} + \frac L{W_{0,0}} W_{0,0}^+ .
\end{align*}
We conclude that 
\begin{equation} \label{eq:W00p_lwr_bd_negcase}
  W_{0,0}^+ \geq \frac{W_{0,0}^2}{2L} .
\end{equation}
Since $W_{0,0}\leq L$, for $f_0'(\pm 1)=0$, the right--hand side of \eqref{eq:W00p_lwr_bd_negcase} is smaller than the right--hand side of \eqref{eq:W00p_lwr_bd_poscase}, so that \eqref{eq:W00p_lwr_bd_negcase} is valid for either sign of $\postsup{e}{\star}^{i_0}(0)f_0'(z_0)$.

We next show that \eqref{eq:W00p_lwr_bd_negcase} entails that we can estimate $W_0$ in terms of $W_0^+$. By \eqref{eq:lambdaiestariei_of_f_differ_by_Otheta} and \eqref{eq:df0_is_OW00}, we have for all $i$ 
\begin{equation} \label{eq:estarifdf0_close_to_estarizerodf0}
\bigl( \postsup{e}{\star}^i(f)-\postsup{e}{\star}^i(0) \bigr)f_0' = \bigl( \postsup{e}{\star}^i(\vartheta f_0) - \postsup{e}{\star}^i(0) \bigr)f_0' = \mathcal O(\vartheta W_{0,0}) .
\end{equation}
Consequently, choosing $\vartheta$ small enough, we can ensure that 
\begin{equation*}
  \Bigl| \bigl(\postsup{e}{\star}^i(\vartheta f_0) - \postsup{e}{\star}^i(0)\bigr)f_0' \Bigr| \leq \frac12 W_{0,0}^+ , \qquad \forall i\in\{0,\ldots,N\} ,
\end{equation*}
so that, if we pick $i^+\in\{1,\ldots,N\}$ and $z^+\in I_0$ such that $\postsup{e}{\star}^{i^+}(0)f_0'(z^+) = W_{0,0}^+$ (cf.~\eqref{eq:W00W00p}), we get 
\begin{equation*}
  \postsup{e}{\star}^{i^+}\bigl(\vartheta f_0(z^+)\bigr)f_0'(z^+) = W_{0,0}^+ + \Bigl(\postsup{e}{\star}^{i^+}\bigl(\vartheta f_0(z^+)\bigr) - \postsup{e}{\star}^{i^+}(0)\Bigr)f_0'(z^+) \geq W_{0,0}^+ - \frac12 W_{0,0}^+ = \frac12 W_{0,0}^+ .
\end{equation*}
But then, since we have $W_0^+\geq \postsup{e}{\star}^{i^+}(f)f'$ on all $I_0$ (cf.~\eqref{eq:W0W0p}), we infer that 
\begin{equation} \label{eq:W0p_pos}
  W_0^+ \geq \postsup{e}{\star}^{i^+}\bigl(\vartheta f_0(z^+)\bigr)\bigl(\vartheta f_0'(z^+)\bigr) \geq \frac\vartheta2 W_{0,0}^+ > 0 .
\end{equation}
On the other hand, if $\overline i$ and $\overline z$ are such that $\lrabs{\postsup{e}{\star}^{\overline i}\bigl(f(\overline z)\bigr)f'(\overline z)}=W_0$, we recover from \eqref{eq:main_ic}, \eqref{eq:W00W00p}, \eqref{eq:estarifdf0_close_to_estarizerodf0}, that 
\begin{align*}
  W_0 &= \lrabs{ \postsup{e}{\star}^{\overline i}\bigl(\vartheta f_0(\overline z)\bigr)\bigl(\vartheta f_0'(\overline z)\bigr) } \leq \vartheta\lrabs{ \postsup{e}{\star}^{\overline i}(0)f_0'(\overline z) } + \vartheta\lrabs{ \Bigl(\postsup{e}{\star}^{\overline i}\bigl(\vartheta f_0(\overline z)\bigr)-\postsup{e}{\star}^{\overline i}(0)\Bigr)f_0'(\overline z) } \\
    &= \mathcal O(\vartheta W_{0,0} + \vartheta^2W_{0,0}) = \mathcal O( \vartheta W_{0,0} ).
\end{align*}
We conclude, using \eqref{eq:W0p_pos}, \eqref{eq:W00p_lwr_bd_negcase}, that 
\begin{equation*}
  \frac{W_0}{W_0^+} = \mathcal O\biggl( \frac{\vartheta W_{0,0}}{\vartheta W_{0,0}^+} \biggr) = \mathcal O(1) ,
\end{equation*}
i.e., 
\begin{equation} \label{eq:W0_is_OW0p}
  W_0 = \mathcal O(W_0^+) .
\end{equation}

\subsection{John's Theorem --- the bounds} \label{subsec:johns_thm}

In this subsection, we recover the bounds appearing in the proof of the main theorem of F.~John (\cite[Section~2]{john74}) following closely his argumentation. We assume that, for some $T>0$, $u\in C^{K}\bigl(\rz\times[0,T];\nball[N]_\delta(0)\bigr)$ is a solution of \eqref{eq:main_cp}, where $a$ and $\delta$ are as in Theorem~\ref{thm:main_thm}, and $\vartheta$ is small enough so that the estimates of Subsection~\ref{subsec:initdataassumptions} hold.

We start by considering 
\begin{equation} \label{eq:Wt}
  W(t) = \max_i\sup_{\stack{(x',t')}{0\leq t'\leq t}} \lrabs{ w^i(x',t') }
\end{equation}
on $[0,T]$, where $w^i(x,t)=w^i\bigl(u(x,t)\bigr)$ is defined by \eqref{eq:wk}. Notice that $W$ is non--decreasing in $t$ (it is actually also continuous, since the $w^i$'s are $C^{K-1}$ and have bounded support in $x$), and that $W(0)=W_0$. From \eqref{eq:dsiwi}, we have for all $i\in\{1,\ldots,N\}$ and every $x\in\rz$ that 
\begin{equation} \label{eq:dsiwiabs_initialestimate}
  \frac{\partial}{\partial s_i}\lrabs{w^i(x,t)} \leq \lrabs{\frac{\partial}{\partial s_i}w^i(x,t)} = \lrabs{ \sum_{l,m} \cull{\gamma ilm}w^lw^m } \leq \Gamma W(t)^2 , \qquad \text{ for almost every } t\in[0,T] ,
\end{equation}
where 
\begin{equation*}
  \Gamma = \sup_{v\in\nball[N]_\delta(0)}\sum_{i,l,m}\lrabs{\cull{\gamma ilm}(v)} .
\end{equation*}
We then apply the following standard Gronwall--type argument (cf., e.g., \cite[\S~XI.13]{mpf91}) to obtain that 
\begin{equation} \label{eq:wiabs_initialestimate}
  \lrabs{w^i(x,t)} \leq Y(t)
\end{equation}
for all $t\in[0,T]$ for which the solution $Y(t)$ of the initial value problem 
\begin{equation} \label{eq:Y_ivp}
  \left\{ \begin{aligned}
    \frac{d}{dt}Y &= \Gamma Y^2 , \\
    Y(0) &= W_0 ,
  \end{aligned}\right.
\end{equation}
is defined. Indeed, let $Z$ be a solution of the initial value problem 
\begin{equation*}
  \left\{ \begin{aligned}
    \frac{d}{dt}Z &= \Gamma \bigl( \max\{W,Z\} \bigr)^2 , \\
    Z(0) &= W_0 ,
  \end{aligned} \right.
\end{equation*}
and assume that there is a time $t_1$ for which $W(t_1)>Z(t_1)$. Then there is an $i_1\in\{1,\ldots,N\}$ and a $z_1\in\rz$, such that 
\begin{equation} \label{eq:Wt1gtZt1}
  \lrabs{w^{i_1}\bigl(X_{i_1}(z_1,t_1),t_1\bigr)} = W(t_1) > Z(t_1) .
\end{equation}
Let $\underline t=\inf\bigl\{ t \bigm| W(s)>Z(s),\forall s\in[t,t_1] \bigr\}$. Then, necessarily, 
\begin{equation} \label{eq:Wt0isZt0}
W(\underline t)=Z(\underline t) .
\end{equation}
Moreover, on $[\underline t,t_1]$, we have $\max\{W,Z\}=W$, so that 
\begin{equation*}
  Z(t_1)-Z(\underline t) = \int_{\underline t}^{t_1} \Gamma W(s)^2 ds ,
\end{equation*}
or
\begin{equation} \label{eq:Zt0plusintisZt1}
  Z(\underline t) + \int_{\underline t}^{t_1} \Gamma W(s)^2 ds = Z(t_1) .
\end{equation}
On the other hand, by integrating \eqref{eq:dsiwiabs_initialestimate} for $i_1$ along the $i_1^{\text{th}}$ characteristic $\mathcal C_{i_1}(z_1)$ passing through $\bigl(X_{i_1}(z_1,t_1),t_1\bigr)$, we have 
\begin{equation*}
  \Bigabs{w^{i_1}\bigl(X_{i_1}(z_1,t_1),t_1\bigr)} - \Bigabs{w^{i_1}\bigl(X_{i_1}(z_1,\underline t),\underline t\bigr)} \leq \int_{\underline t}^{t_1} \Gamma W(s)^2 ds .
\end{equation*}
This means that, using \eqref{eq:Wt}, \eqref{eq:Wt0isZt0} and \eqref{eq:Zt0plusintisZt1}, we get 
\begin{align*}
  \lrabs{w^{i_1}\bigl(X_{i_1}(z_1,t_1),t_1\bigr)} &\leq \lrabs{w^{i_1}\bigl(X_{i_1}(z_1,\underline t),\underline t\bigr)} + \int_{\underline t}^{t_1} \Gamma W(s)^2 ds \\
    &\leq W(\underline t) + \int_{\underline t}^{t_1} \Gamma W(s)^2 ds = Z(\underline t) + \int_{\underline t}^{t_1} \Gamma W(s)^2 ds = Z(t_1).
\end{align*}
But this contradicts \eqref{eq:Wt1gtZt1}, and we must have $W(t)\leq Z(t)$ for all $t\in[0,T]$ for which $Z$ is defined. But then $Z$ is actually a solution of \eqref{eq:Y_ivp}, which establishes \eqref{eq:wiabs_initialestimate}.

Now, \eqref{eq:Y_ivp} has the solution 
\begin{equation} \label{eq:Yt}
  Y(t) = \frac{W_0}{1-\Gamma W_0t} , \qquad t\in\left[0,\frac1{\Gamma W_0}\right) .
\end{equation}
Recalling from \eqref{eq:t0} that the time $t_0$ after which the characteristic strips $\mathcal R_i$ are separated is $\mathcal O(1)$, so that, by \eqref{eq:W0_is_Otheta}, 
\begin{equation*}
  \Gamma W_0 t_0 = \mathcal O(W_0) = \mathcal O(\vartheta) ,
\end{equation*}
we see that \eqref{eq:wiabs_initialestimate} implies 
\begin{equation} \label{eq:wiabs_t0estimate}
  \lrabs{w^i(x,t)} \leq (1+\varepsilon) W_0 , \qquad \forall x\in\rz ,\; \forall t\in[0,t_0] ,
\end{equation}
provided that $\vartheta$ is small enough. In other words, using \eqref{eq:W0_is_Otheta},
\begin{equation} \label{eq:Wt_t0estimate}
  W(t) = \mathcal O(W_0) = \mathcal O(\vartheta) , \qquad \forall t\in[0,t_0] .
\end{equation}

Let us now introduce the following time--dependent bounds defined on $[0,T]$:
\begin{subequations} \label{eqs:VSJU}
\begin{align}
\label{eq:Vt}
    V(t) &= \max_i\sup_{\stack{(x',t')\not\in \mathcal R_i}{0\leq t'\leq t}} \lrabs{ w^i(x',t') } , \\
\label{eq:St}
    S(t) &= \max_i\sup_{\stack{(z_i',s_i')}{\stack{z_i'\in I_0}{0\leq s_i'\leq t}}} \rho_i(z_i',s_i') , \\
\label{eq:Jt}
    J(t) &= \max_i\sup_{\stack{(z_i',s_i')}{\stack{z_i'\in I_0}{0\leq s_i'\leq t}}} \lrabs{ v^i(z_i',s_i') } , \\
\label{eq:Ut}
    U(t) &= \sup_{\stack{(x',t')}{0\leq t'\leq t}} \lrabs{ u(x',t') } .
\end{align}
\end{subequations}
Like $W(t)$ (see \eqref{eq:Wt}), these quantities are non--decreasing in $t$ (and, actually, also continuous). By \eqref{eq:rhoi}, \eqref{eq:Xi_ivp}, we see that, initially, 
\begin{equation} \label{eq:rhoi_initially}
  \rho_i(z_i,0)=1 ,
\end{equation}
and thus, by \eqref{eq:vi}, that 
\begin{equation} \label{eq:vi_initially}
  v^i(z_i,0) = w^i(z_i,0) ,
\end{equation}
so that we obtain, using that $u(\argplaceholder,0)=f(\argplaceholder)\in\nball[N]_\delta(0)$ is supported in $I_0$,
\begin{subequations} \label{eqs:VSJU_initially}
\begin{align}
\label{eq:V0}
  V(0) &= 0 , \\
\label{eq:S0}
  S(0) &= 1 , \\
\label{eq:J0}
  J(0) &= W_0 , \\
\label{eq:U0}
  U(0) &\leq \delta , \qquad\text{ and, by \eqref{eq:fdf_are_Otheta}, }\qquad U(0) = \mathcal O(\vartheta) . 
\end{align}
\end{subequations}
Next, we want to estimate these quantities on $[0,t_0]$. For $V$, we make use of \eqref{eq:dsiwiabs_initialestimate} and the fact that, for any $(x',t')\not\in \mathcal R_i$ the corresponding characteristic coordinates $(z_i',s_i')=\bigl(z_i'(x',t'),s_i'(x',t')\bigr)$ fulfil $z_i'\not\in I_0$ so that $w^i(z_i',0)=0$. Then we have, using \eqref{eq:t0}, \eqref{eq:Wt_t0estimate}, 
\begin{equation} \label{eq:Vt_t0estimate}
  V(t) = \mathcal O(W_0^2) , \qquad \forall t\in[0,t_0] .
\end{equation}
We then consider $\rho_i$. By the second to last equality in \eqref{eq:dsirhoi}, we see that 
\begin{equation} \label{eq:dsirhoi_is_OWrhoi}
  \frac{\partial}{\partial s_i}\rho_i(z_i,s_i) = \mathcal O\bigl( W(s_i)\rho_i(z_i,s_i) \bigr) ,
\end{equation}
from which we immediately conclude that, for as long as $u$ is defined and $C^K$ (i.e., for $t\in[0,T]$ according to our assumption), so that $W$ is bounded, we have, using the standard Gronwall--inequality (Lemma~\ref{lem:G}),  
\begin{equation*}
  \rho_i(z_i,t) \geq \underbrace{\rho_i(z_i,0)}_{\phantom{1}=1}\exp\Bigl(-\mathcal O\bigl(tW(t)\bigr)\Bigr) ,
\end{equation*}
and hence 
\begin{equation} \label{eq:rhoi_pos}
  \rho_i(\argplaceholder,t) > 0 , \qquad \forall i\in\{1,\ldots,N\} ,\; \forall t\in[0,T] .
\end{equation}
In particular, we have for $t\in[0,t_0]$, using \eqref{eq:t0}, \eqref{eq:Wt_t0estimate}, 
\begin{equation*}
  \rho_i(z_i,t) \geq \exp\bigl(-\mathcal O(\vartheta)\bigr) ,
\end{equation*}
so that we get for small enough $\vartheta$, 
\begin{equation} \label{eq:rhoi_t0estimate}
  \rho_i(\argplaceholder,t) \geq 1-\varepsilon > 0 , \qquad \forall t\in[0,t_0] .
\end{equation}
Similarly, \eqref{eq:dsirhoi_is_OWrhoi} implies that $\rho_i(z_i,t)\leq\exp\Bigl(\mathcal O\bigl(tW(t)\bigr)\Bigr)$, so that, by \eqref{eq:Wt_t0estimate}, \eqref{eq:St}, 
\begin{equation} \label{eq:St_bettert0estimate}
  S(t) \leq (1+\varepsilon) , \qquad \forall t\in[0,t_0] ,
\end{equation}
for sufficiently small $\vartheta$. In other words, 
\begin{equation} \label{eq:St_t0estimate}
  S(t) = \mathcal O(1) , \qquad \forall t\in[0,t_0] .
\end{equation}
Regarding $v^i$, we obtain from \eqref{eq:dsivi}, using \eqref{eq:vi}, \eqref{eq:Wt}, \eqref{eq:St}, that
\begin{equation*}
  \frac{\partial}{\partial s_i}v^i(z_i,s_i) = \mathcal O( \rho_iW(t)^2 ) = \mathcal O\bigl(S(t)W(t)^2\bigr) ,
\end{equation*}
which, by \eqref{eq:t0}, \eqref{eq:St_t0estimate}, \eqref{eq:Wt_t0estimate}, \eqref{eq:J0}, \eqref{eq:W0_is_Otheta}, implies after integrating and taking supremums that 
\begin{equation} \label{eq:Jt_t0estimate}
  J(t) = J(0) + \mathcal O( W_0^2 ) = W_0 + \mathcal O( \vartheta W_0 ) = \mathcal O(W_0) , \qquad \forall t\in[0,t_0] .
\end{equation}
Finally, since 
\begin{equation*}
  u(x,t) = \int_{X_N(-1,t)}^{x} \frac{\partial}{\partial x}u(x',t) dx' ,
\end{equation*}
we obtain from \eqref{eq:dxu}, 
\begin{equation} \label{eq:u_is_Osumintwi}
  \bigabs{u(x,t)} = \mathcal O \Biggl( \sum_i \int_{X_N(-1,t)}^{X_1(1,t)} \lrabs{w^i(x',t)} dx' \Biggr) .
\end{equation}
Then, by \eqref{eq:Ut}, \eqref{eq:extreme_chars}, \eqref{eq:Wt}, \eqref{eq:t0}, 
\begin{equation} \label{eq:Ut_t0estimate}
  U(t) = \mathcal O \biggl( \Bigl(2+\bigl(\lambda_1(0)-\lambda_N(0)\bigr)t\Bigr)W(t) \biggr) = \mathcal O( W_0 ) = \mathcal O(\vartheta) , \qquad \forall t\in[0,t_0] ,
\end{equation}
so that 
\begin{equation} \label{eq:Ut0_is_le_delta}
  U(t) \leq \delta , \qquad \forall t\in[0,t_0] ,
\end{equation}
for sufficiently small $\vartheta$.

Using \eqref{eq:W0_is_Otheta}, we conclude from \eqref{eq:Vt_t0estimate}, \eqref{eq:rhoi_t0estimate} and \eqref{eq:St_bettert0estimate}, and \eqref{eq:Jt_t0estimate}, respectively, that $V$, $S$, and $J$ stay close to their initial values \eqref{eqs:VSJU_initially}, as long as $t\in[0,t_0]$ and $\vartheta$ is suitably small. Also, under those assumptions, the solution $u$ remains in the $N$--ball $\nball[N]_{\delta}(0)$ where it lay initially.

We next consider times $t\in[t_0,T]$. Let $(x,t)\in \mathcal R_i$ for some $i$ be given. Then the last equality in \eqref{eq:dsirhoi} gives along the $i^{\text{th}}$ characteristic $\mathcal C_i(z_i)$ through $(x,t)$ that 
\begin{equation} \label{eq:dsirhoi_estimate}
  \frac{\partial}{\partial s_i}\rho_i(z_i,s_i) = \mathcal O\bigl( J(s_i) + V(s_i)S(s_i) \bigr) , \qquad \forall s_i\in[t_0,T] ,
\end{equation}
where we have used that, by \eqref{eq:RkRl_separated}, $\bigl(\mathcal R_m\bigcap \mathcal R_i\bigr)\bigcap\bigl\{(x',t')\bigm|t'\in[t_0,T]\bigr\}=\emptyset$, $\forall m\neq i$. Integrating, and using \eqref{eq:St_t0estimate}, we obtain 
\begin{equation*}
  \rho_i(z_i,t) = \mathcal O\bigl( 1 + tJ(t) + tV(t)S(t) \bigr) ,
\end{equation*}
so that, by varying $i$ and $x$ such that $(x,t)\in \mathcal R_i\bigcap\bigl\{(x',t')\bigm|t'\in[t_0,T]\bigr\}$, we have 
\begin{equation} \label{eq:St_testimate}
  S(t) = \mathcal O\bigl( 1 + tJ(t) + tV(t)S(t) \bigr) , \qquad \forall t\in[0,T] ,
\end{equation}
which is compatible with \eqref{eq:St_t0estimate}. Next, we argue in the same way for $v^i$. Fixing again $(x,t)\in \mathcal R_i$ for some $i$ and using \eqref{eq:dsivi}, we get 
\begin{equation} \label{eq:dsivi_estimate}
  \frac{\partial}{\partial s_i}v^i(z_i,s_i) = \mathcal O\bigl( V(s_i)J(s_i) + V(s_i)^2S(s_i) \bigr) , \qquad \forall s_i\in[t_0,t] .
\end{equation}
Integrating, using \eqref{eq:Jt_t0estimate}, and varying $i$ and $x$, we arrive at 
\begin{equation} \label{eq:Jt_testimate}
  J(t) = \mathcal O\bigl( W_0 + tV(t)J(t) + tV(t)^2S(t) \bigr) , \qquad \forall t\in[0,T] .
\end{equation}
From \eqref{eq:St_testimate} and \eqref{eq:Jt_testimate}, we see that it will be fundamental to estimate $V(t)$. So fix $i\in\{1,\ldots,N\}$ and assume that $(x,t)\not\in \mathcal R_i$ with $t\in[t_0,T]$. \eqref{eq:dsiwi} then gives along the $i^{\text{th}}$ characteristic $\mathcal C_i(z_i)$ passing through $(x,t)$, 
\begin{equation}
\label{eq:dsiwi_estimate}
  \frac{\partial}{\partial s_i} w^i\bigl( X_i(z_i,s_i),s_i \bigr) = \sum_{l,m} (\cull{\gamma ilm}w^lw^m)\bigl(X_i(z_i,s_i),s_i\bigr) .
\end{equation}
For the given $z_i$, let 
\begin{equation} \label{eq:pik}
  \Pi^{z_i}_k(t) = \Bigl\{ t'\in[0,t] \Bigm| \bigl(X_i(z_i,t'),t'\bigr) \in \mathcal R_k \Bigl\} , \qquad k\in\{1,\ldots,N\} .
\end{equation}
Clearly, $\Pi^{z_i}_i(t)=\emptyset$, since $z_i\not\in I_0$. Also, by definition \eqref{eq:t0} of $t_0$, so that \eqref{eq:RkRl_separated} holds, $\bigl(\Pi^{z_i}_l(t)\bigcap\Pi^{z_i}_m(t)\bigr)\bigcap [t_0,T] =\emptyset$. Therefore, using $z_i\not\in I_0$, whence $w^i(z_i,0)=0$, as well as \eqref{eq:gammaill_is_zero}, \eqref{eq:Vt}, \eqref{eq:Vt_t0estimate}, we obtain after integration along $\mathcal C_i(z_i)$ 
\begin{equation} \label{eq:wi_estimate}
  \lrabs{w^i(z_i,t)} = \mathcal O \Biggl( W_0^2 + tV(t)^2 + V(t)\sum_{\stack k{k\neq i}}\int_{\Pi^{z_i}_k(t)}\lrabs{w^k\bigl(X_i(z_i,t'),t'\bigr)}dt' \Biggr) .
\end{equation}
Observe that 
\begin{equation*} \tag{\ref{eq:gammaill_is_zero}}
  \cull{\gamma ill}=0 , \qquad\text{ if $l\neq i$,}
\end{equation*}
was crucial for this estimate. Now, for each $t'\in\Pi^{z_i}_k(t)$, we have $\bigl(X_i(z_i,t'),t'\bigr)\in \mathcal R_k$. Parametrising $\mathcal C_i(z_i)\bigcap\mathcal R_k$ by $y_k\in I^{z_i}_k(t)$ for some $I^{z_i}_k(t)\subset I_0$ (recall Subsection~\ref{subsec:bichar_coords}), we get from \eqref{eqs:dbccoords_d}, using that $y_i=z_i$ is constant along $\mathcal C_i(z_i)$, 
\begin{align}
\label{eq:intpikwk_is_OJ}
  \int_{\Pi^{z_i}_k(t)}\lrabs{ w^k\bigl(X_i(z_i,t'),t'\bigr) } dt' &= \int_{I^{z_i}_k(t)} \frac{\rho_k\bigl(y_k,t'(y_i,y_k)\bigr)}{\lrabs{\lambda_i-\lambda_k}}\Bigabs{ w^k\bigl(y_k,t'(y_i,y_k)\bigr) } dy_k \\
\nonumber
    &= \int_{I^{z_i}_k(t)} \frac{\bigabs{ v^k\bigl(y_k,t'(y_i,y_k)\bigr) }}{\lrabs{\lambda_i-\lambda_k}} dy_k = \mathcal O\bigl( J(t) \bigr) .
\end{align}
Consequently, we get from \eqref{eq:wi_estimate}, \eqref{eq:Vt} after varying $i$ and $x$ such that $(x,t)\not\in \mathcal R_i$, 
\begin{equation} \label{eq:Vt_testimate}
  V(t) = \mathcal O\bigl( W_0^2 + tV(t)^2 + V(t)J(t) \bigr) , \qquad t\in[0,T] .
\end{equation}

Finally, going back to \eqref{eq:u_is_Osumintwi} and sub--dividing the $x$--interval $\bigl[X_N(-1,t),X_1(1,t)\bigr]$ for $t\in[t_0,T]$ into regions where $(x,t)\in \mathcal R_i$ for some $i$ and regions where $(x,t)\not\in \bigcup_{k} \mathcal R_k$, we obtain with \eqref{eq:extreme_chars} 
\begin{equation*}
  \bigabs{u(x,t)} = \mathcal O\Bigl( \bigl( X_1(1,t)-X_N(-1,t) \bigr)V(t) + J(t) \Bigr) = \mathcal O\biggl( \Bigl( 2+\bigl(\lambda_1(0)-\lambda_N(0)\bigr) t \Bigr) V(t) + J(t) \biggr) ,
\end{equation*}
where we have used that, by \eqref{eqs:dccoords}, changing variables from $(x',t)$ to $(z_i',t)$ yields for constant $t$ 
\begin{equation*}
  \int_{X_i(-1,t)}^{X_i(1,t)} \lrabs{ w^i(x',t) } dx' = \int_{-1}^1 \lrabs{ w^i\bigl(X_i(z_i',t),t\bigr) } \rho_i(z_i',t) dz_i' = \int_{I_0} \lrabs{v^i(z_i',t)} dz_i' = \mathcal O\bigl( J(t) \bigr) .
\end{equation*}
Consequently, we obtain taking supremums 
\begin{equation} \label{eq:Ut_testimate}
  U(t) = \mathcal O\bigl( V(t) + tV(t) + J(t) \bigr) , \qquad t\in[0,T] ,
\end{equation}
where we have used that, for $t\in[0,t_0]$, \eqref{eq:Ut_t0estimate}, \eqref{eq:Wt}, \eqref{eq:vi}, \eqref{eq:rhoi_t0estimate} and \eqref{eq:Jt} allow us to estimate $U(t)$ by $J(t)$.

Summarising, we have from \eqref{eq:Vt_testimate}, \eqref{eq:St_testimate}, \eqref{eq:Jt_testimate} and \eqref{eq:Ut_testimate}, omitting the explicit $t$--de\-pend\-ence, 
\begin{subequations} \label{eqs:VSJU_testimates}
\begin{align}
\label{eq:V_testimate}
  V &= \mathcal O( W_0^2 + tV\cdot V + VJ ) , \\
\label{eq:S_testimate}
  S &= \mathcal O( 1 + tJ + tV\cdot S ) , \\
\label{eq:J_testimate}
  J &= \mathcal O( W_0 + tV\cdot J + tV\cdot VS ) , \\
\label{eq:U_testimate}
  U &= \mathcal O( V + tV + J ) ,
\end{align}
\end{subequations}
which are valid for all $C^K$--solutions $u$ of \eqref{eq:main_cp} for which $u\in\nball[N]_\delta(0)$ for all $t\in[0,T]$, provided $\vartheta$ is small enough.

We now perform a bootstrap argument. We let $t$ increase from zero so as to keep the following inequalities satisfied:
\begin{subequations} \label{eqs:bootstrap_assumptions}
\begin{gather}
\label{eq:t_bootstrap_assumption}
  tW_0^+ \leq \max_i \frac1{(1-\varepsilon)^4\abs{\cull{ciii}(0)}} , \\
\label{eqs:tV_J_V_U_bootstrap_assumptions}
  tV \leq \sqrt\vartheta , \qquad J \leq \sqrt\vartheta , \qquad V \leq \vartheta , \qquad U \leq \sqrt\vartheta .
\end{gather}
\end{subequations}
When $t=0$, the estimates \eqref{eqs:tV_J_V_U_bootstrap_assumptions} hold by \eqref{eqs:VSJU_initially}, \eqref{eq:W0_is_Otheta}, for all sufficiently small $\vartheta$. Notice also that the right--hand side of \eqref{eq:t_bootstrap_assumption} is a constant (depending only on $a$ and $\varepsilon$) which is finite, since we assumed the system to be genuinely non--linear. We then get, for sufficiently small $\vartheta$,
\begin{align*}
  V &= \mathcal O( W_0^2 + \sqrt\vartheta V + \sqrt\vartheta V ) &\Longrightarrow\qquad\qquad V &= \mathcal O( W_0^2 ) , \\
  S &= \mathcal O( 1 + tJ + \sqrt\vartheta S ) &\Longrightarrow\qquad\qquad S &= \mathcal O( 1 + tJ ) , \\
  J &= \mathcal O( W_0 + \sqrt\vartheta J + \sqrt\vartheta VS ) &\Longrightarrow\qquad\qquad J &= \mathcal O( W_0 + \sqrt\vartheta VS ) ,
\end{align*}
so that 
\begin{align*}
  tV &= \mathcal O( tW_0^2 ) , \\ 
  VS &= \mathcal O( V + tV\cdot J ) = \mathcal O( V + \sqrt\vartheta J ) .
\end{align*}
Then, using \eqref{eq:t_bootstrap_assumption}, \eqref{eq:W0_is_OW0p}, \eqref{eq:W0_is_Otheta}, we obtain 
\begin{align*}
  V &= \mathcal O(W_0^2) = \mathcal O(\vartheta^2) , \\
  tV &= \mathcal O( tW_0^+\cdot W_0 ) = \mathcal O( W_0 ) = \mathcal O( \vartheta ) , \\ 
  J &= \mathcal O( W_0 + \sqrt\vartheta V + \vartheta J ) \quad\Longrightarrow\quad J = \mathcal O( W_0 + \sqrt\vartheta W_0^2 ) = \mathcal O( W_0 ) = \mathcal O(\vartheta) .
\end{align*}
Finally, this gives 
\begin{equation*}
  U = \mathcal O( W_0^2 + W_0 + W_0 ) = \mathcal O( W_0 ) = \mathcal O( \vartheta ) ,
\end{equation*}
so that, indeed, $u$ stays in $\nball[N]_\delta(0)$, provided $\vartheta$ is small enough. So we have established that, as long as $t$ is so small as to keep the conditions \eqref{eqs:bootstrap_assumptions} fulfilled, we have the stronger estimates 
\begin{align}
\label{eqs:tV_J_V_U_bootstrap_results}
  tV &= \mathcal O(W_0)  & J &= \mathcal O(W_0)  & V &= \mathcal O(W_0^2)  & U &= \mathcal O(W_0) \\
\nonumber
     &= \mathcal O(\vartheta), &&= \mathcal O(\vartheta), &&= \mathcal O(\vartheta^2), &&= \mathcal O(\vartheta),
\end{align}
which, therefore, hold as long as \eqref{eq:t_bootstrap_assumption} does, the assumptions \eqref{eqs:tV_J_V_U_bootstrap_assumptions} being unnecessary. For later reference, we also list the following additional estimates that thus hold (using, as well, \eqref{eq:W0_is_OW0p} and \eqref{eq:W0_is_Otheta}):
\begin{equation} \label{eqs:other_bootstrap_results}
  \begin{split}
    S &= \mathcal O(1+tJ) = \mathcal O(1+tW_0) = \mathcal O(1+tW_0^+) = \mathcal O(1) , \\
    VS &= \mathcal O( V+tV\cdot J ) = \mathcal O( W_0^2 + \vartheta W_0 ) = \mathcal O( \vartheta W_0 ) = \mathcal O( \vartheta^2 ) .
  \end{split}
\end{equation}
\eqref{eqs:tV_J_V_U_bootstrap_results} and \eqref{eqs:other_bootstrap_results}, together with \eqref{eq:rhoi_pos}, establish statement \eqref{thm:main_thm_firstorder} of Theorem~\ref{thm:main_thm} (with the exception of the bounds on $\rho_i$ and $v^i$ outside $\mathcal R_i$, which are obtained in the next subsection).

We now focus our attention on the inverse density $\rho_i$ of the $i^{\text{th}}$ characteristics, thereby also establishing the second part of F.~John's theorem (\cite[Section~2]{john74}).

\subsection{The inverse density of characteristics} \label{subsec:rho}
In this subsection, we analyse more in detail the behaviour of the inverse density of characteristics $\rho_i$. We assume that $u\in C^{K}\bigl(\rz\times[0,t_\star);\nball[N]_\delta(0)\bigr)$ is a solution of the strictly hyperbolic and genuinely non--linear system \eqref{eq:main_pde_plain} with initial condition \eqref{eq:main_ic} satisfying the additional assumptions set out in Subsection~\ref{subsec:initdataassumptions}, where $t_\star$ is the maximal time of existence of $u$. Among other things, we will give an explicit estimate for the size of $t_\star$, as well as establish that $\rho_i$ has to vanish somewhere at that time. This then recovers the second part of F.~John's main theorem (\cite[Section~2]{john74}), since \eqref{eq:vi}, the bound on $J$ in \eqref{eqs:tV_J_V_U_bootstrap_results} and the lower bound on $v^i$ in \eqref{eq:vip_in_zp_pos_for_all_t} below imply that $w^i$ has to blow up as $t$ approaches $t_\star$.

Our starting point is the evolution equation~\eqref{eq:dsirhoi} of the inverse density $\rho_i$ along an $i^{\text{th}}$ characteristic, 
\begin{equation*} \tag{\ref{eq:dsirhoi}}
  \frac{\partial}{\partial s_i} \rho_i = \cull{ciii}v^i + \Bigl( \sum_{\stack m{m\neq i}} \cull{ciim}w^m \Bigr) \rho_i ,
\end{equation*}
which, given our sign convention \eqref{eq:gen_nl_cond_unif} ($\cull{ciii}< 0$), we rewrite as 
\begin{equation} \label{eq:dsirhoi_signed}
  \frac{\partial}{\partial s_i} \rho_i = -\abs{\cull{ciii}}v^i + \Bigl( \sum_{\stack m{m\neq i}} \cull{ciim}w^m \Bigr) \rho_i .
\end{equation}
We want to show, in order to obtain Theorem~\ref{thm:main_thm}~\eqref{thm:main_thm_tstar}, that there is an index $i^+\in\{1,\ldots,N\}$ and a point $z^+\in I_0=[-1,1]$ such that $\rho_{i^+}\bigl(z^+,t\bigr)$ tends to zero as $t\nearrow t_\star$, and that  
\begin{equation} \label{eq:tstar_range}
  t_\star \in \bigl[ \underline{T_0}, \overline{T_0} \bigr] ,
\end{equation}
where 
\begin{equation} \label{eq:T0}
  \underline{T_0} = \min_i \frac1{(1+\varepsilon)^3\abs{\cull{ciii}(0)}W_0^+} \qquad\text{ and }\qquad \overline{T_0} = \max_i \frac1{(1-\varepsilon)^4\abs{\cull{ciii}(0)}W_0^+} .
\end{equation}
Notice that $tW_0^+$ then fulfils \eqref{eq:t_bootstrap_assumption} for all $t\in[0,\overline{T_0}]$, so that the estimates \eqref{eqs:tV_J_V_U_bootstrap_results}, \eqref{eqs:other_bootstrap_results} hold for any $t\in[0,T]$, $T\leq \overline{T_0}$, for which we have a $C^{K}$ solution $u$ on $\rz\times[0,T]$.

First, let us define 
\begin{equation} \label{eq:alphat}
  \alpha(t) = \max_i \sup_{\stack{(z',s')}{\stack{z'\in I_0}{0\leq s'\leq t}}} \biggabs{ \sum_{\stack m{m\neq i}} \cull{ciim}\Bigl(u\bigl(X_i(z',s'),s'\bigr)\Bigr) w^m\bigl(X_i(z',s'),s'\bigr) } ,
\end{equation}
so that, by \eqref{eq:dsirhoi_signed}, 
\begin{equation*}
  -\abs{\cull{ciii}}v^i - \alpha \rho_i \leq \frac{\partial}{\partial s_i}\rho_i \leq -\abs{\cull{ciii}}v^i + \alpha \rho_i .
\end{equation*}
Then, if $v^i\geq0$, we have by \eqref{eq:ciii_stays_close} for small enough $\vartheta$, 
\begin{equation} \label{eq:dsirhoi_range_vinotneg}
  -(1+\varepsilon)\abs{\cull{ciii}(0)}v^i - \alpha \rho_i \leq \frac{\partial}{\partial s_i}\rho_i \leq -(1-\varepsilon)\abs{\cull{ciii}(0)}v^i + \alpha \rho_i .
\end{equation}
Now, by \eqref{eq:Wt_t0estimate}, \eqref{eq:Vt}, \eqref{eqs:tV_J_V_U_bootstrap_results}, \eqref{eq:W0_is_Otheta}, we have for all $t\in\bigl[0,\min\{t_\star,\overline{T_0}\}\bigr)$, 
\begin{equation*}
  \int_0^t \alpha(s) ds = \mathcal O(W_0+tV) = \mathcal O(\vartheta) ,
\end{equation*}
so that, for sufficiently small $\vartheta$, 
\begin{equation*}
  \exp \biggl( \int_0^t \alpha(s) ds \biggr) \leq (1+\varepsilon), \qquad\text{ implying }\qquad \exp \biggl( -\int_0^t \alpha(s) ds \biggr) \geq (1-\varepsilon) .
\end{equation*}
Consequently, by the standard Gronwall--inequality (Lemma~\ref{lem:G}), \eqref{eq:rhoi_initially} and \eqref{eq:dsirhoi_range_vinotneg} imply for sufficiently small $\vartheta$, and for as long as $v^i\geq 0$ along the $i^{\text{th}}$ characteristic starting at $z_i$, that 
\begin{multline} \label{eq:rhoi_range_vinotneg}
  (1-\varepsilon)\biggl( 1 - (1+\varepsilon)^2\abs{\cull{ciii}(0)}\int_0^t v^i(z_i,s) ds \biggr) \leq \rho_i(z_i,t) \\
    \leq (1+\varepsilon)\biggl( 1 - (1-\varepsilon)^2\abs{\cull{ciii}(0)}\int_0^t v^i(z_i,s) ds \biggr) .
\end{multline}
In particular, 
\begin{equation} \label{eq:rhoi_upr_bd_vinotneg}
\rho_i\leq(1+\varepsilon)
\end{equation}
as long as $v^i\geq 0$, provided $\vartheta$ is small enough.

Now recall \eqref{eq:dsivi}, 
\begin{equation*} \tag{\ref{eq:dsivi}}
  \frac{\partial}{\partial s_i}v^i = \Bigl( \sum_{\stack m{m\neq i}} \bigl( 2\cull{\gamma iim}+\cull{ciim} \bigr) w^m \Bigr) v^i + \Bigl( \sum_{\stack{l,m}{\stack{l\neq i,m\neq i}{l\neq m}}} w^lw^m \Bigr) \rho_i .
\end{equation*}
Letting 
\begin{equation} \label{eqs:beta12t}
  \begin{split}
    \beta_1(t) &= \max_i \sup_{\stack{(z',s')}{\stack{z'\in I_0}{0\leq s'\leq t}}} \biggabs{ \sum_{\stack m{m\neq i}} \bigl(2\cull{\gamma iim}+\cull{ciim}\bigr)\Bigl(u\bigl(X_i(z',s'),s'\bigr)\Bigr) w^m\bigl(X_i(z',s'),s'\bigr) } , \\
    \beta_2(t) &= \max_i \sup_{\stack{(z',s')}{\stack{z'\in I_0}{0\leq s'\leq t}}} \biggabs{ \sum_{\stack{l,m}{\stack{l\neq i,m\neq i}{l\neq m}}} \cull{\gamma ilm}\Bigl(u\bigl(X_i(z',s'),s'\bigr)\Bigr) w^l\bigl(X_i(z',s'),s'\bigr) w^m\bigl(X_i(z',s'),s'\bigr) } ,
  \end{split}
\end{equation}
we have, whenever $v^i\geq 0$, and using \eqref{eq:rhoi_upr_bd_vinotneg}, 
\begin{equation} \label{eq:dsivi_lwr_bd_vinotneg}
  \frac{\partial}{\partial s_i}v^i \geq - \beta_1 v^i - \beta_2 \rho_i \geq - \beta_1 v^i - (1+\varepsilon)\beta_2 .
\end{equation}
So, again by Gronwall and arguing as before, taking into account \eqref{eq:wiabs_t0estimate}, \eqref{eqs:tV_J_V_U_bootstrap_results}, \eqref{eq:W0_is_Otheta}, we get for sufficiently small $\vartheta$ that 
\begin{equation*}
  v^i(z_i,t) \geq (1-\varepsilon)\biggl( w^i(z_i,0) - (1+\varepsilon)^2\int_0^t \beta_2(s) ds \biggr) ,
\end{equation*}
provided that $v^i(z_i,s)\geq 0$, $\forall s\in[0,t]$. In particular, taking $i^+\in\{1,\ldots,N\}$ and $z^+\in I_0$ such that $v^{i^+}(z^+,0)=w^{i^+}(z^+,0)=W_0^+>0$ (recall \eqref{eq:gen_nl_cond}, \eqref{eq:vi_initially}, \eqref{eq:W0W0p}), we obtain 
\begin{equation*}
  v^{i^+}(z^+,t) \geq (1-\varepsilon)\biggl( W_0^+ - (1+\varepsilon)^2\int_0^t \beta_2(s) ds \biggr) .
\end{equation*}
Observing that, by \eqref{eq:wiabs_t0estimate}, \eqref{eqs:tV_J_V_U_bootstrap_results}, \eqref{eq:W0_is_Otheta}, \eqref{eq:W0_is_OW0p}, 
\begin{equation*}
  \int_0^t\beta_2(s)ds = \mathcal O( W_0^2 + tV^2 ) = \mathcal O( W_0^2 + \vartheta W_0^2 ) = \mathcal O( \vartheta W_0^+ ) ,
\end{equation*}
we obtain, for sufficiently small $\vartheta$, 
\begin{equation} \label{eq:vip_in_zp_pos_for_all_t}
  v^{i^+}(z^+,t) \geq (1-\varepsilon)\bigl( W_0^+ - \varepsilon W_0^+ \bigr) \geq (1-\varepsilon)^2 W_0^+ .
\end{equation}
In particular, $v^{i^+}(z^+,t) > 0$ for all $t\in\bigl[0,\min\{t_\star,\overline{T_0}\}\bigr)$, and we did not need to assume $v^{i^+}\geq 0$ along $\mathcal C_{i^+}(z^+)$. Also, going back to \eqref{eq:rhoi_range_vinotneg}, we have 
\begin{equation} \label{eq:rhoi_upr_bd_max}
  \rho_{i^+}(z^+,t) \leq (1+\varepsilon) \Bigl( 1 - (1-\varepsilon)^4\bigabs{\cull{c{i^+}{i^+}{i^+}}(0)} tW_0^+ \Bigr) , \qquad \forall t\in\bigl[ 0 , \min\{t_\star,\overline{T_0}\} \bigr) ,
\end{equation}
from which it is immediate that $\lim_{t\nearrow T_{i^+}} \rho_{i^+}(z^+,t) = 0$ for some 
\begin{equation} \label{eq:Tip_estimate}
  T_{i^+} \leq \frac1{(1-\varepsilon)^4\bigabs{\cull{c{i^+}{i^+}{i^+}}(0)}W_0^+} \leq \overline{T_0} .
\end{equation}
Since we already argued at the beginning of this subsection that $u$ cannot exist after the time $T_{i^+}$ when $\rho_{i^+}$ becomes zero, this establishes the upper bound for $t_\star$ in \eqref{eq:tstar_range}, which is similar to the one obtained by John.

To obtain the lower bound on $t_\star$, we first consider again \eqref{eq:dsivi}, from which we immediately obtain by \eqref{eq:vi}, \eqref{eq:Wt_t0estimate}, \eqref{eq:St_t0estimate}, \eqref{eq:t0}, \eqref{eqs:VSJU}, \eqref{eqs:tV_J_V_U_bootstrap_results}, \eqref{eqs:other_bootstrap_results}, \eqref{eq:W0_is_Otheta} 
\begin{equation*}
  v^i(z_i,t) - w^i(z_i,0) \leq \mathcal O( W_0^2 + tVJ + tV^2S ) = \mathcal O( \vartheta W_0 + \vartheta W_0 + \vartheta^2 W_0 ) = \mathcal O( \vartheta W_0 ) ,
\end{equation*}
i.e., 
\begin{equation} \label{eq:vi_upr_bd}
  v^i(z_i,t) \leq w^i(z_i,0) + \mathcal O( \vartheta W_0 ) .
\end{equation}
Similarly, and for later reference, we obtain 
\begin{equation} \label{eq:vi_lwr_bd}
  v^i(z_i,t) \geq w^i(z_i,0) - \mathcal O( \vartheta W_0 ) .
\end{equation}
Consequently, we can assume by \eqref{eq:W0W0p}, \eqref{eq:W0_is_OW0p}, that, for suitably small $\vartheta$, 
\begin{equation} \label{eq:vip_stays_close_to_W0p}
  v^i \leq (1+\varepsilon) W_0^+ , \qquad \forall i\in\{1,\ldots,N\} .
\end{equation}
Now, going back to \eqref{eq:dsirhoi_signed}, we have with \eqref{eq:alphat} 
\begin{equation*}
  \frac{\partial}{\partial s_i}\rho_i \geq -\abs{\cull{ciii}}(v^i)_+ - \alpha\rho_i ,
\end{equation*}for $v^i \leq (v^i)_+$, whence $-v^i \geq -(v^i)_+$, where $(v^i)_+$ denotes the positive part of $v^i$. Thus, by \eqref{eq:ciii_stays_close} and for $\vartheta$ small enough, 
\begin{equation} \label{eq:dsirhoi_lwr_bd_vip}
  \frac{\partial}{\partial s_i}\rho_i \geq -(1+\varepsilon)\abs{\cull{ciii}(0)}(v^i)_+ - \alpha\rho_i .
\end{equation}
Then, by Gronwall (Lemma~\ref{lem:G}), using \eqref{eq:Wt_t0estimate}, \eqref{eq:Vt}, \eqref{eqs:tV_J_V_U_bootstrap_results}, \eqref{eq:rhoi_initially}, we get for sufficiently small $\vartheta$ and all $t\in[0,t_\star)$ (we already showed that $t_\star\leq\overline{T_0}$) 
\begin{equation} \label{eq:rhoi_lwr_bd_vip}
  \rho_i(z_i,t) \geq (1-\varepsilon) \biggl( 1 - (1+\varepsilon)^2\abs{\cull{ciii}(0)}\int_0^t \bigl(v^i(z_i,s)\bigr)_+ ds \biggr) .
\end{equation}
Thus, \eqref{eq:vip_stays_close_to_W0p} implies that 
\begin{equation*}
  \rho_i(\argplaceholder,t) \geq (1-\varepsilon)\bigl( 1 - (1+\varepsilon)^3\abs{\cull{ciii}(0)}tW_0^+ \bigr) .
\end{equation*}
Therefore, 
\begin{equation*}
  \rho_i(\argplaceholder,t) > 0 \qquad\text{ as long as }\qquad t < \frac1{(1+\varepsilon)^3\abs{\cull{ciii}(0)}W_0^+} .
\end{equation*}
We conclude that the inverse density of characteristics in \emph{each} characteristic strip remains bounded away from zero as long as $t<\underline{T_0}$ (cf.~\eqref{eq:T0}). Thus, from \eqref{eq:dxu}, \eqref{eq:vi}, \eqref{eqs:tV_J_V_U_bootstrap_results}, we see that $u\in\nball[N]_{\delta}(0)$ and $\frac{\partial}{\partial x}u$ remains bounded, for all times in $\bigl[0,\min\{t_\star,\underline{T_0}\}\bigr)$. Thus, if $t_\star<\underline{T_0}$, we should be able to extend $u$ across $t=t_\star$, which would contradict the definition of $t_\star$ as the maximal time of existence of the solution $u$. Therefore, also the lower bound in \eqref{eq:tstar_range} must hold. This finishes the proof of Theorem~\ref{thm:main_thm}~\eqref{thm:main_thm_tstar}.

We next wish to show that $\rho_i$ remains bounded away from zero along every $j^\text{th}$ characteristic ending (that is, for $t$ approaching $t_\star$) in a point outside $\mathcal R_i$. Before doing so, however, we quickly analyse the behaviour of $\rho_i$ outside $\mathcal R_i$. From \eqref{eq:dsirhoi}, we have for every $z_i\not\in I_0$ 
\begin{equation*}
  -\widetilde\alpha(z_i,t)\rho_i(z_i,t) \leq \frac{\partial}{\partial s_i}\rho_i(z_i,t) \leq \widetilde\alpha(z_i,t)\rho_i(z_i,t) ,
\end{equation*}
where 
\begin{equation*}
  \widetilde\alpha(z_i,t) = \sup_{0\leq s'\leq t} \biggabs{ \sum_m \cull{ciim}\Bigl(u\bigl(X_i(z_i,s'),s'\bigr)\Bigr)w^m\bigl(X_i(z_i,s'),s'\bigr) } , 
\end{equation*}
which is non--decreasing in $t$. Then, by \eqref{eq:Wt}, \eqref{eq:Wt_t0estimate}, \eqref{eqs:VSJU}, \eqref{eq:pik}, \eqref{eq:intpikwk_is_OJ}, \eqref{eq:W0_is_Otheta}, \eqref{eqs:tV_J_V_U_bootstrap_results}, we get for all $z_i\not\in I_0$, $t\in[0,t_\star)$,
\begin{align*}
  \int_0^t\widetilde\alpha(z_i,s)ds &= \mathcal O \biggl( W_0 + tV + \sum_k\int_{\Pi^{z_i}_k(t)}\lrabs{w^k\bigl(X_i(z_i,t'),t'\bigr)}dt' \biggr) \\
  &= \mathcal O( W_0 + tV + J) = \mathcal O( \vartheta + \vartheta + \vartheta ) = \mathcal O(\vartheta) .
\end{align*}
Thus, using Gronwall with \eqref{eq:rhoi_initially}, we can ensure that, if $\vartheta$ is chosen suitably small, we have 
\begin{equation} \label{eq:rhoi_range_outside_Ri}
  1-\varepsilon \leq \rho_i(z_i,t) \leq 1+\varepsilon , \qquad \forall z_i\not\in I_0, \forall t\in[0,t_\star) .
\end{equation}
We observe that this then also implies that, using \eqref{eq:vi}, \eqref{eqs:tV_J_V_U_bootstrap_results},
\begin{equation} \label{eq:vi_bound_outside_Ri}
  v^i(z_i,t) = \mathcal O(W_0^2) , \qquad \forall z_i\not\in I_0, \forall t\in[0,t_\star) .
\end{equation}
This thus establishes the remaining part of Theorem~\ref{thm:main_thm}~\eqref{thm:main_thm_firstorder}.

Let now $(x,t)\in\bigl(\rz\times[0,t_\star)\bigr)\setminus \bigcup_{k}\mathcal R_k$ and consider the $j^{\text{th}}$ characteristic $\mathcal C_j(z_j)$, $z_j\not\in I_0$, passing through $(x,t)$. Assume further that $i\in\{1,\ldots,N\}$ is such that $\mathcal C_j(z_j)\bigcap \mathcal R_i\neq\emptyset$. Fixing $\kappa\in\Bigl(0,\frac{(1-3\varepsilon)}{(1+\varepsilon)}\Bigr)$ (for the sake of definiteness, let $\kappa=\frac{(1-\varepsilon)}3$), we face two possibilities. Either $\rho_i\geq\kappa$ all along $\mathcal C_j(z_j)\bigcap \mathcal R_i$, or there is a point on $\mathcal C_j(z_j)\bigcap \mathcal R_i$ where $\rho_i<\kappa$. Having nothing to prove in the first case, we focus our attention on the second case and show that $\rho_i$ cannot become too small along $\mathcal C_j(z_j)\bigcap \mathcal R_i$. Recalling Subsection~\ref{subsec:bichar_coords}, let us parametrise $\mathcal C_j(z_j)\bigcap \mathcal R_i$ by $y_i\in I_0$ ($(z_j,0)$ and $(x,t)$ both lie outside $\mathcal R_i$), so that 
\begin{equation*}
  \mathcal C_j(z_j)\bigcap\mathcal R_i = \biggl\{ \Bigl(X_i\bigl(y_i,t'(y_i,y_j)\bigr),t'(y_i,y_j)\Bigr) \biggm| y_j=z_j, y_i\in I_0 \biggr\} . 
\end{equation*}
By our assumption, there must be a $\widehat y_i\in I_0$ such that $\rho_i\bigl(\widehat y_i,t'(\widehat y_i,z_j)\bigr)<\kappa$, so let us study $\rho_i$ along the $i^{\text{th}}$ characteristic $\mathcal C_i(\widehat y_i)$ starting at $\widehat y_i$. Our first claim is that $\rho_i$ must be strictly decreasing whenever it is smaller than $\kappa$, at least when $\vartheta$ is sufficiently small. To see this, we first establish that $v^i$ needs to enjoy a positive initial lower bound in this case. For assume that, for any $z_i\in I_0$, 
\begin{equation} \label{eq:wi_initial_upr_bd_rhoi_larger_kappa}
  w^i(z_i,0) \leq \frac{(1-\varepsilon)^3\min_k\abs{\cull{ckkk}(0)}}{(1+\varepsilon)^2\max_l\abs{\cull{clll}(0)}}(1-2\varepsilon-\kappa)W_0^+ \stackrel{\eqref{eq:T0}}= \widefrac{\underline{T_0}}{\overline{T_0}}\cdot\frac{(1+\varepsilon)}{(1-\varepsilon)}(1-2\varepsilon-\kappa)W_0^+ .
\end{equation}
Then, by \eqref{eq:vi_upr_bd}, \eqref{eq:W0_is_OW0p}, 
\begin{equation*}
  v^i(z_i,s_i) \leq w^i(z_i,0)+\mathcal O(\vartheta W_0) \leq \widefrac{\underline{T_0}}{\overline{T_0}}\cdot\frac{(1+\varepsilon)}{(1-\varepsilon)}(1-\varepsilon-\kappa)W_0^+ , \qquad \forall s_i \in [0,t_\star) ,
\end{equation*}
if $\vartheta$ is chosen to be suitably small (but independently of $\kappa$). Since this bound holds also for the positive part $(v^i)_+$ of $v^i$, \eqref{eq:rhoi_lwr_bd_vip} then implies that 
\begin{equation*}
  \rho_i(z_i,s_i) \geq (1-\varepsilon)\biggl( 1 - (1+\varepsilon)^2\abs{\cull{ciii}(0)}\widefrac{\underline{T_0}}{\overline{T_0}}\cdot\frac{(1+\varepsilon)}{(1-\varepsilon)}(1-\varepsilon-\kappa)s_i W_0^+ \biggr),
\end{equation*}
which, since $s_i<t_\star\leq\overline{T_0}$, gives by \eqref{eq:T0}
\begin{align*}
  \rho_i(z_i,s_i) &\geq (1-\varepsilon)\biggl( 1- \frac{(1+\varepsilon)^3\abs{\cull{ciii}(0)}(1-\varepsilon-\kappa)W_0^+}{(1+\varepsilon)^3\bigl(\max_l\abs{\cull{clll}(0)}\bigr)W_0^+(1-\varepsilon)} \biggr) \\
  &\geq (1-\varepsilon)\biggl( \frac{1-\varepsilon - (1-\varepsilon-\kappa)}{1-\varepsilon} \biggr) = \kappa , & \forall s_i\in[0,t_\star) .
\end{align*}
We conclude that 
\begin{equation*}
  w^i(\widehat y_i,0) > \widefrac{\underline{T_0}}{\overline{T_0}}\cdot\frac{(1+\varepsilon)}{(1-\varepsilon)}(1-2\varepsilon-\kappa)W_0^+ ,
\end{equation*}
so that $\rho_i<\kappa$ is possible along $\mathcal C_i(\widehat y_i)$. But then, \eqref{eq:vi_lwr_bd} and \eqref{eq:W0_is_OW0p} imply that, for sufficiently (but independently of $\kappa$) small $\vartheta$,
\begin{equation*}
  v^i(\widehat y_i,s_i) \geq \widefrac{\underline{T_0}}{\overline{T_0}}\frac{(1+\varepsilon)}{(1-\varepsilon)}(1-3\varepsilon-\kappa)W_0^+ , \qquad \forall s_i\in [0,t_\star) .
\end{equation*}
We thus see that, by \eqref{eq:dsirhoi_signed} and using \eqref{eq:ciii_stays_close}, \eqref{eq:Vt}, we have 
\begin{equation*}
  - \frac{\partial}{\partial s_i}\rho_i(\widehat y_i, s_i) \geq (1+\varepsilon)\abs{\cull{ciii}(0)}\widefrac{\underline{T_0}}{\overline{T_0}}(1-3\varepsilon-\kappa)W_0^+ - \kappa \mathcal O(V) ,
\end{equation*}
whenever $\rho_i<\kappa$. By \eqref{eqs:tV_J_V_U_bootstrap_results}, $V=\mathcal O(W_0^2)$, hence, using \eqref{eq:W0_is_Otheta}, \eqref{eq:W0_is_OW0p}, we can estimate the second term on the right--hand side for suitably small $\vartheta$ (again, independently of $\kappa$) by a small fraction of the first term, say by 
\begin{equation*}
  -(1+\varepsilon)\abs{\cull{ciii}(0)}\widefrac{\underline{T_0}}{\overline{T_0}}\varepsilon\kappa W_0^+ ,
\end{equation*}
so that we obtain 
\begin{equation} \label{eq:mdsirhoi_lwr_bd_rhoi_smaller_kappa}
  - \frac{\partial}{\partial s_i}\rho_i(\widehat y_i,s_i) \geq (1+\varepsilon)\abs{\cull{ciii}(0)}\widefrac{\underline{T_0}}{\overline{T_0}}\bigl(1-3\varepsilon-(1+\varepsilon)\kappa\bigr)W_0^+ ,
\end{equation}
which is valid for all $s_i\in[0,t_\star)$ for which $\rho_i(\widehat y_i,s_i)<\kappa$. This means that, once on an $i^{\text{th}}$ characteristic $\rho_i$ drops below $\kappa$, it thereafter steadily decreases. It follows that 
\begin{equation*}
  \rho_i(\widehat y_i,s_i)<\kappa \quad\text{ and }\quad\text{ \eqref{eq:mdsirhoi_lwr_bd_rhoi_smaller_kappa} holds, }\qquad\text{ for all }  s_i\in\bigl[t'(\widehat y_i,y_j),t_\star\bigr) .
\end{equation*}
As a consequence, letting $\Delta T=t-t'(\widehat y_i,z_j)>0$, we have 
\begin{align}
\nonumber
  \rho_i(\widehat y_i, t-\Delta T) &\geq \underbrace{\rho_i(\widehat y_i, t)}_{\phantom0\geq0} + (1+\varepsilon)\abs{\cull{ciii}(0)}\widefrac{\underline{T_0}}{\overline{T_0}}\bigl(1-3\varepsilon-(1+\varepsilon)\kappa\bigr)W_0^+\Delta T \\
\label{eq:rhoi_lwr_bd_DeltaT}
  &\geq (1+\varepsilon)\abs{\cull{ciii}(0)}\widefrac{\underline{T_0}}{\overline{T_0}}\bigl(1-3\varepsilon-(1+\varepsilon)\kappa\bigr)W_0^+\Delta T .
\end{align}
Therefore, if we can find a lower bound on $\Delta T$, we will have established that $\rho_i$ remains bounded away from zero along $\mathcal C_j(z_j)\bigcap \mathcal R_i$. To do so, we will make use of the fact that the seed $f_0$ for the initial data $f$ is $C^{1,1}$ and compactly supported in $I_0=[-1,1]$, so that $f_0$ and $f_0'$ vanish at the boundary $\partial I_0=\{\pm1\}$ of $I_0$. For simplicity, we restrict ourselves to the case $i>j$, the case $i<j$ being shown in an entirely similar manner, requiring only a few obvious modifications. Under this additional assumption, we must have for our initially picked point $(x,t)$ that $x>X_i(1,t)$, i.e., $(x,t)$ lies on the right of $\mathcal R_i$. We then denote by 
\begin{equation} \label{eq:dir}
  d^r_i(z_i,s_i) = X_i(1,s_i)-X_i(z_i,s_i) ,\qquad z_i\in I_0 , s_i\in[0,t_\star) ,
\end{equation}
the $x$--distance at time $s_i$ of the $i^{\text{th}}$ characteristic $\mathcal C_i(z_i)$ emanating from $z_i\in I_0$ to the right boundary $X_i(1,t)$ of the $i^{\text{th}}$ strip $\mathcal R_i$. Noticing that \eqref{eq:lambdaiestariei_of_f_differ_by_Otheta} implies for sufficiently small $\vartheta$ that 
\begin{equation*}
  1-\varepsilon = (1-\varepsilon)\bignorm{\postsup{e}{\star}^i(0)}_0 \leq \bignorm{\postsup{e}{\star}^i(f)}_0 \leq (1+\varepsilon)\bignorm{\postsup{e}{\star}^i(0)}_0 = 1+\varepsilon ,
\end{equation*}
 we see from \eqref{eq:wk}, \eqref{eq:main_ic}, the fact that $f_0$ is supported in $I_0$, and \eqref{eq:L}, that 
\begin{align*}
  \lrabs{w^i(z_i,0)} &= \Bigabs{\postsup{e}{\star}^i\bigl(f(z_i)\bigr)f'(z_i)} \leq (1+\varepsilon)\bignorm{f'(z_i)}_0 \\
    &\leq (1+\varepsilon)\vartheta\bignorm{f_0'(z_i)}_0 = (1+\varepsilon)\vartheta\bignorm{f_0'(z_i)-\underbrace{f_0'(1)}_{\phantom0=0}}_0 \\
    &\leq (1+\varepsilon)\vartheta L\lrabs{z_i-1} = (1+\varepsilon)\vartheta L d^r_i(z_i,0) .
\end{align*}
As a consequence, since \eqref{eq:wi_initial_upr_bd_rhoi_larger_kappa} implies that $\rho_i\geq\kappa$, we have: 
\begin{equation*}
  \text{if }\quad d^r_i(z_i,0) \leq \widefrac{\underline{T_0}}{\overline{T_0}}\cdot\frac{(1-2\varepsilon-\kappa)}{(1-\varepsilon)}\frac{W_0^+}{\vartheta L} , \quad%
    \text{ then }\quad \rho_i(z_i,s_i) \geq \kappa , \quad \forall s_i\in[0,t_\star) .
\end{equation*}
We note that, in particular, we must have 
\begin{equation*}
  d^r_i(\widehat y_i,0) > \widefrac{\underline{T_0}}{\overline{T_0}}\cdot\frac{(1-2\varepsilon-\kappa)}{(1-\varepsilon)}\frac{W_0^+}{\vartheta L} .
\end{equation*}
Moreover, applying the mean value theorem to the definition \eqref{eq:rhoi} of $\rho_i$, we infer that: 
\begin{equation*}
  \text{if }\quad d^r_i(z_i,0) \leq \widefrac{\underline{T_0}}{\overline{T_0}}\cdot\frac{(1-2\varepsilon-\kappa)}{(1-\varepsilon)}\frac{W_0^+}{\vartheta L} , \quad%
    \text{ then }\quad d^r_i(z_i,s_i) \geq \kappa d^r_i(z_i,0) , \quad \forall s_i\in[0,t_\star) .
\end{equation*}
Since characteristics of the same family never cross (cf.~\eqref{eq:Xiz1z2_ineq}), we conclude that 
\begin{equation} \label{eq:dri_lwr_bd}
  d^r_i(\widehat y_i,s_i) \geq \widefrac{\underline{T_0}}{\overline{T_0}}\cdot\frac{(1-2\varepsilon-\kappa)}{(1-\varepsilon)}\frac{W_0^+}{\vartheta L}\kappa , \qquad \forall s_i\in[0,t_\star) .
\end{equation}
By \eqref{eq:Xi_ivp}, \eqref{eq:supinflambdai}, we have 
\begin{align*}
  x-X_i\bigl(\widehat y_i,t'(\widehat y_i,z_j)\bigr) &\leq \overline\lambda_j \Delta T \\
\intertext{and}
  X_i(\widehat y_i,t) - X_i\bigl(\widehat y_i,t'(\widehat y_i,z_j)\bigr) &\geq \underline\lambda_i \Delta T , \\
\intertext{so that (recall $i>j$, so that, by \eqref{eq:sigma_pos}, $\underline\lambda_i\leq\overline\lambda_i<\underline\lambda_j\leq\overline\lambda_j$)}
  \Delta T \geq \frac{x-X_i(\widehat y_i,t)}{\overline\lambda_j-\underline\lambda_i} \geq \frac{d^r_i(\widehat y_i,t)}{\overline\lambda_j-\underline\lambda_i} .
\end{align*}
Setting then 
\begin{equation} \label{eq:Sigma}
  \Sigma = \overline\lambda_1-\underline\lambda_N ,
\end{equation}
we recover from \eqref{eq:dri_lwr_bd} that 
\begin{equation*}
  \Delta T \geq \widefrac{\underline{T_0}}{\overline{T_0}}\cdot\frac{(1-2\varepsilon-\kappa)}{(1-\varepsilon)}\frac{W_0^+}{\vartheta\Sigma L}\kappa .
\end{equation*}
Inserting this into \eqref{eq:rhoi_lwr_bd_DeltaT}, we get 
\begin{align*}
  \rho_i(\widehat y_i,t-\Delta T) &\geq \frac{(1+\varepsilon)}{(1-\varepsilon)}\abs{\cull{ciii}(0)}\biggl(\widefrac{\underline{T_0}}{\overline{T_0}}W_0^+\biggr)^2\frac{\bigl((1-3\varepsilon)-(1+\varepsilon)\kappa\bigr)(1-2\varepsilon-\kappa)\kappa}{\vartheta\Sigma L} \\
    &\geq \frac{1}{27}(1+\varepsilon)(1-\varepsilon)^2\frac{\abs{\cull{ciii}(0)}}{\Sigma}\biggl(\widefrac{\underline{T_0}}{\overline{T_0}}\biggr)^2\frac{(W_0^+)^2}{\vartheta L} ,
\end{align*}
where the last inequality follows from our choice of $\kappa=\frac{(1-\varepsilon)}3$ (recall also that $\varepsilon$ was supposed to be very small --- here, $\varepsilon<1/10$ would have been enough). Since $\widehat y_i\in I_0$ was arbitrary (but such that $\rho_i\bigl(\widehat y_i,t'(\widehat y_i,z_j)\bigr)<\kappa$), we finally obtain using \eqref{eq:W0p_pos}, \eqref{eq:W00p_lwr_bd_negcase}, \eqref{eq:T0} 
\begin{equation} \label{eq:rhoi_lwr_bd_Cj_cap_Ri}
  \rho_i \geq \min\biggl\{ \frac{(1-\varepsilon)}{3} \;,\; \vartheta C_{i,\varepsilon} \biggr\} \qquad\text{ along $\mathcal C_j(z_j)\bigcap \mathcal R_i$, if $\vartheta$ is small enough} ,
\end{equation}
where 
\begin{equation*}
  C_{i,\varepsilon} = \frac{1}{432}\cdot\frac{(1-\varepsilon)^{10}}{(1+\varepsilon)^5}\cdot\frac{\abs{\cull{ciii}(0)}}{\Sigma}\biggl(\frac{\min_k\abs{\cull{ckkk}(0)}}{\max_l\abs{\cull{clll}(0)}}\biggr)^2\frac{(W_{0,0})^4}{L^3} 
\end{equation*}
is a constant depending only on $i$, $\varepsilon$, $a$, $\delta$ and the $C^{1,1}$--norm of $f_0$.

\section{The bounds on the second order quantities} \label{sec:secondorder}

\subsection{The evolution equations for the second order quantities} \label{subsec:2ndorder_evoleqs}

We use \eqref{eq:dsiwi_fully_expanded}, \eqref{eq:dsirhoi} and \eqref{eq:dsivi} to obtain the evolution equations along $i^{\text{th}}$ characteristics for $\frac{\partial}{\partial x}w^i$, $\mu_i=\frac{\partial}{\partial z_i}\rho_i$ and $\nu^i=\frac{\partial}{\partial z_i}v^i$, respectively, obtaining second order quantities on the right--hand side. However, for $\mu_i$ and $\nu^i$, these will depend on $\presup{\tau}{(i)}_j=\frac{\partial}{\partial y_j}\rho_j$ and $\presup{\omega}{(i)}^j=\frac{\partial}{\partial y_j}v^j$ ($j\in\{1,\ldots,N\}\setminus\{i\}$). Recall from Subsection~\ref{subsec:bichar_coords} that $y_i$ is held fixed when taking $\frac{\partial}{\partial y_j}$ in the definitions \eqref{eqs:mt_tauijomegaij} of $\presup{\tau}{(i)}_j$ and $\presup{\omega}{(i)}^j$, so that the latter are the derivatives along $i^{\text{th}}$ characteristics of $\rho_j$ and $v^j$, respectively, in bi--characteristic coordinates. As we will need to bound $\presup{\tau}{(i)}_j$ and $\presup{\omega}{(i)}^j$ inside $\mathcal R_i$, we compute their evolution equations in bi--characteristic coordinates along $j^{\text{th}}$ characteristics. Using those, we will be able to obtain bounds inside $\mathcal R_i$ that depend on bounds at the boundary of $\mathcal R_i$.

First, remark that, using \eqref{eqs:dccoords}, \eqref{eq:dxlambdaj}, we have 
\begin{equation} \label{eq:commutator_dsidx}
  \Bigl[ \frac{\partial}{\partial s_i},\frac{\partial}{\partial x} \Bigr] = \frac{\partial}{\partial s_i}\frac{\partial}{\partial x} - \frac{\partial}{\partial x}\frac{\partial}{\partial s_i} = - \Bigl(\frac{\partial}{\partial x}\lambda_i\Bigr)\frac{\partial}{\partial x} = - \Bigl( \sum_k\cull{ciik}w^k \Bigr)\frac{\partial}{\partial x} . 
\end{equation}
Therefore, \eqref{eq:dsiwi_fully_expanded} gives, taking into account \eqref{eq:dxu} and \eqref{eq:gamma_sym} (recall also the notation \eqref{eq:Delphi} from \ref{subsec:genuine_nl}), 
\begin{align}
  \frac{\partial}{\partial s_i}\Bigl(\frac{\partial}{\partial x} w^i\Bigr) &= - \Bigl(\frac{\partial}{\partial x}\lambda_i\Bigr)\Bigl(\frac{\partial}{\partial x} w^i\Bigr) + \frac{\partial}{\partial x}\Bigl( \frac{\partial}{\partial s_i}w^i \Bigr) \nonumber \displaybreak[0] \\
  &= - \cull{ciii}w^i\Bigl(\frac{\partial}{\partial x}w^i\Bigr) - \Bigl(\sum_{\stack m{m\neq i}}\cull{ciim}w^m\Bigr)\Bigl(\frac{\partial}{\partial x}w^i\Bigr) \nonumber \\
  &\phantom{=\ } {} + \frac{\partial}{\partial x}\biggl( -\cull{ciii}(w^i)^2 + \Bigl(\sum_{\stack m{m\neq i}}2\cull{\gamma iim}w^m\Bigr)w^i + \sum_{\stack{l,m}{\stack{l\neq i,m\neq i}{l\neq m}}} \cull{\gamma ilm}w^lw^m \biggr) \nonumber \displaybreak[0] \\
  &= - \cull{ciii}w^i\Bigl(\frac{\partial}{\partial x}w^i\Bigr) - \Bigl(\sum_{\stack m{m\neq i}}\cull{ciim}w^m\Bigr)\Bigl(\frac{\partial}{\partial x}w^i\Bigr) \nonumber \\
  &\phantom{=\ } {} - \bigl(D_{e_i}\cull{ciii}\bigr)(w^i)^3 - \Bigl(\sum_{\stack p{p\neq i}}\bigl(D_{e_p}\cull{ciii}\bigr)w^p\Bigr)(w^i)^2 - 2\cull{ciii}w^i\Bigl(\frac{\partial}{\partial x}w^i\Bigr) \nonumber \\
  &\phantom{=\ } {} + \Bigl( \sum_{\stack m{m\neq i}}\bigl(2D_{e_i}\cull{\gamma iim}\bigr)w^m\Bigr)(w^i)^2 + \Bigl( \sum_{\stack{m,p}{m\neq i,p\neq i}}\bigl(2D_{e_p}\cull{\gamma iim}\bigr)w^mw^p\Bigr)w^i \nonumber \\
  &\phantom{=\ } {} + \biggl(\sum_{\stack m{m\neq i}}2\cull{\gamma iim}\Bigl(\frac{\partial}{\partial x}w^m\Bigr)\biggr)w^i + \Bigl(\sum_{\stack m{m\neq i}}2\cull{\gamma iim}w^m\Bigr)\Bigl(\frac{\partial}{\partial x}w^i\Bigr) \nonumber \\
  &\phantom{=\ } {} + \Bigl(\sum_{\makebox[0pt][c]{\tiny $\stack{l,m}{\stack{l\neq i,m\neq i}{l\neq m}}$}}\bigl(D_{e_i}\cull{\gamma ilm}\bigr)w^lw^m\Bigr)w^i + 2\sum_{\makebox[0pt][c]{\tiny $\stack{l,m}{\stack{l\neq i,m\neq i}{l\neq m}}$}}\bigl(D_{e_m}\cull{\gamma ilm}\bigr)w^l(w^m)^2 \nonumber \\
  &\phantom{=\ } {} + \sum_{\makebox[0pt][c]{\tiny $\stack{l,m,p}{\stack{l\neq i,m\neq i,p\neq i}{l\neq m,p\neq l,p\neq m}}$}}\bigl(D_{e_p}\cull{\gamma ilm}\bigr)w^lw^mw^p + 2\sum_{\makebox[0pt][c]{\tiny $\stack{l,m}{\stack{l\neq i,m\neq i}{l\neq m}}$}}\cull{\gamma ilm}w^l\Bigl(\frac{\partial}{\partial x} w^m\Bigr) , \nonumber \\ 
\intertext{i.e.,}
\label{eq:dsidxwi}
  \frac{\partial}{\partial s_i}\Bigl(\frac{\partial}{\partial x}w^i\Bigr) %
  &= -3\cull{ciii}w^i\Bigl(\frac{\partial}{\partial x}w^i\Bigr) - \Bigl(\sum_{\stack m{m\neq i}}\cull{ciim}w^m\Bigr)\Bigl(\frac{\partial}{\partial x}w^i\Bigr) + \Bigl(\sum_{\stack m{m\neq i}}2\cull{\gamma iim}w^m\Bigr)\Bigl(\frac{\partial}{\partial x}w^i\Bigr) \\
  &\phantom{=\ } {} + \biggl(\sum_{\stack m{m\neq i}}2\cull{\gamma iim}\Bigl(\frac{\partial}{\partial x}w^m\Bigr)\biggr)w^i + 2\sum_{\makebox[0pt][c]{\tiny $\stack{l,m}{\stack{l\neq i,m\neq i}{l\neq m}}$}}\cull{\gamma ilm}w^l\Bigl(\frac{\partial}{\partial x} w^m\Bigr) \nonumber \\
  &\phantom{=\ } {} - \bigl(D_{e_i}\cull{ciii}\bigr)(w^i)^3 - \Bigl(\sum_{\stack p{p\neq i}}\bigl(D_{e_p}\cull{ciii}\bigr)w^p\Bigr)(w^i)^2 + \Bigl( \sum_{\stack m{m\neq i}}\bigl(2D_{e_i}\cull{\gamma iim}\bigr)w^m\Bigr)(w^i)^2 \nonumber \\
  &\phantom{=\ } {} + \Bigl( \sum_{\makebox[0pt][c]{\tiny $\stack{m,p}{m\neq i,p\neq i}$}}\bigl(2D_{e_p}\cull{\gamma iim}\bigr)w^mw^p\Bigr)w^i + \Bigl(\sum_{\makebox[0pt][c]{\tiny $\stack{l,m}{\stack{l\neq i,m\neq i}{l\neq m}}$}}\bigl(D_{e_i}\cull{\gamma ilm}\bigr)w^lw^m\Bigr)w^i \nonumber \\
  &\phantom{=\ } {} + \sum_{\makebox[0pt][c]{\tiny $\stack{l,m,p}{\stack{l\neq i,m\neq i,p\neq i}{l\neq m,p\neq l,p\neq m}}$}}\bigl(D_{e_p}\cull{\gamma ilm}\bigr)w^lw^mw^p + 2\sum_{\makebox[0pt][c]{\tiny $\stack{l,m}{\stack{l\neq i,m\neq i}{l\neq m}}$}}\bigl(D_{e_m}\cull{\gamma ilm}\bigr)w^l(w^m)^2 . \nonumber
\end{align}
We will make use of this equation to show the boundedness of $\frac{\partial}{\partial x} w^i$ ($i\in\{1,\ldots,N\}$) in the region $\bigl(\rz\times[0,t_\star)\bigr)\setminus\bigcup_{k}\mathcal R_k$, i.e., outside all the characteristic strips, using \eqref{eq:rhoi_lwr_bd_Cj_cap_Ri}. As will become apparent in the next subsection, we should also control $\frac{\partial}{\partial x}\rho_i$ ($i\in\{1,\ldots,N\}$) outside all the characteristic strips. But this will follow from the bounds on $\frac{\partial}{\partial x}w^i$ and the following evolution equation of $\frac{\partial}{\partial x}\rho_i$ along an $i^{\text{th}}$ characteristic.
\begin{align}
\nonumber
  \frac{\partial}{\partial s_i}\Bigl(\frac{\partial}{\partial x}\rho_i\Bigr) &\ \ \stackrel{\makebox[0pt][c]{\scriptsize \eqref{eq:commutator_dsidx}}}{=}\ \  -\Bigl(\frac{\partial}{\partial x}\lambda_i\Bigr)\Bigl(\frac{\partial}{\partial x}\rho_i\Bigr) + \frac{\partial}{\partial x}\Bigl(\frac{\partial}{\partial s_i}\rho_i\Bigr) \\
\nonumber
  &\ \ \stackrel{\makebox[0pt][c]{\scriptsize \eqref{eq:rhoi},\eqref{eq:Xi_ivp}}}{=}\ \  -\Bigl(\frac{\partial}{\partial x}\rho_i\Bigr)\Bigl(\frac{\partial}{\partial x}\lambda_i\Bigr) + \frac{\partial}{\partial x}\Bigl(\frac{\partial}{\partial z_i}\lambda_i\Bigr) \stackrel{\eqref{eqs:dccoords}}{=} \frac{\partial}{\partial z_i}\Bigl(\frac{\partial}{\partial x}\lambda_i\Bigr) = \rho_i\frac{\partial^2}{\partial x^2}\lambda_i \\
\nonumber
  &\ \ \stackrel{\makebox[0pt][c]{\scriptsize \eqref{eq:dxlambdaj}}}{=}\ \ \rho_i\frac{\partial}{\partial x}\Bigl(\sum_m\cull{ciim}w^m\Bigr) \stackrel{\eqref{eq:dxu}}{=} \rho_i\sum_{l,m}\bigl(D_{e_l}\cull{ciim}\bigr)w^m + \rho_i\sum_m\cull{ciim}\Bigl(\frac{\partial}{\partial x}w^m\Bigr) \\
\label{eq:dsidxrhoi}
  &\ \ =\ \ \cull{ciii}\Bigl(\frac{\partial}{\partial x}w^i\Bigr)\rho_i + \biggl(\sum_{\stack m{m\neq i}}\cull{ciim}\Bigl(\frac{\partial}{\partial x}w^m\Bigr)\biggr)\rho_i \\
\nonumber
  &\ \ \phantom{=\ }\ \ {} + \bigl(D_{e_i}\cull{ciii}\bigr)(w^i)^2\rho_i + \Bigr(\sum_{\stack m{m\neq i}}\bigl(D_{e_m}\cull{ciii}+D_{e_i}\cull{ciim}\bigr)w^m\Bigr)w^i\rho_i \\
\nonumber
  &\ \ \phantom{=\ }\ \ {} + \Bigl(\sum_{\stack m{m\neq i}}\bigl(D_{e_m}\cull{ciim}\bigr)(w^m)^2\Bigr)\rho_i + \Bigl(\sum_{\stack{l,m}{\stack{l\neq i,m\neq i}{l\neq m}}}\bigl(D_{e_l}\cull{ciim}\bigr)w^lw^m\Bigr)\rho_i .
\end{align}

Before computing the evolution equations for $\mu_i$ and $\nu^i$, we want to point out the following useful consequence of eqs.~\eqref{eqs:dbccoords_partial}:
\begin{equation} \label{eq:partialzj_in_terms_of_partialyj_and_partialsj}
  \frac{\partial}{\partial z_j} = \frac{\partial}{\partial y_j} - \frac{\rho_j}{\lambda_i-\lambda_j}\frac{\partial}{\partial s_j} , \qquad j\in\{1,\ldots,N\}\setminus\{i\} .
\end{equation}
We can use this to express $\mu_j$ and $\nu^j$ in terms of quantities we will be able to control. We have along an $i^{\text{th}}$ characteristic and for all $j\neq i$, 
\begin{subequations} \label{eqs:mujnuj_in_terms_of_tauijomegaij}
\begin{align}
\label{eq:muj_in_terms_of_tauij}
  \mu_j = \frac{\partial}{\partial z_j}\rho_j = \frac{\partial}{\partial y_j}\rho_j + \frac{\rho_j}{\lambda_j-\lambda_i}\Bigl(\frac{\partial}{\partial s_j}\rho_j\Bigr) = \presup{\tau}{(i)}^j + \frac{\rho_j}{\lambda_j-\lambda_i}\Bigl(\frac{\partial}{\partial s_j}\rho_j\Bigr) , \\
\label{eq:nuj_in_terms_of_omegaij}
  \nu^j = \frac{\partial}{\partial z_j}v^j = \frac{\partial}{\partial y_j}v^j + \frac{\rho_j}{\lambda_j-\lambda_i}\Bigl(\frac{\partial}{\partial s_j}v^j\Bigr) = \presup{\omega}{(i)}^j + \frac{\rho_j}{\lambda_j-\lambda_i}\Bigl(\frac{\partial}{\partial s_j}v^j\Bigr) .
\end{align}
\end{subequations}
Also, note that we can write for every $j\in\{1,\ldots,N\}\setminus\{i\}$, 
\begin{align}
\nonumber
  \frac{\partial}{\partial z_i}w^j &= \rho_i\frac{\partial}{\partial x}w^j = \frac{\rho_i}{\rho_j}\Bigl(\frac{\partial}{\partial z_j}\bigl(\rho_j^{-1}v^j\bigr)\Bigr) = -\frac{\rho_i}{\rho_j^3}\Bigl(\frac{\partial}{\partial z_j}\rho_j\Bigr)v^j + \frac{\rho_i}{\rho_j^2}\Bigl(\frac{\partial}{\partial z_j}v^j\Bigr) \\
\label{eq:dziwj}
  &= \frac{\rho_i}{\rho_j^2}\bigl( -w^j\mu_j + \nu^j \bigr) .
\end{align}
Using then \eqref{eqs:mujnuj_in_terms_of_tauijomegaij}, and taking into account \eqref{eq:vi}, \eqref{eq:dsiwi}, \eqref{eq:gammaill_is_zero}, \eqref{eq:gamma_sym}, we obtain 
\begin{align}
\nonumber
  \frac{\partial}{\partial z_i}w^j &= -\frac{w^j}{\rho_j^2}\rho_i\presup{\tau}{(i)}_j + \frac1{\rho_j^2}\rho_i\presup{\omega}{(i)}^j + \frac{\rho_i}{\rho_j(\lambda_j-\lambda_i)}\biggl( -w^j\Bigl(\frac{\partial}{\partial s_j}\rho_j\Bigr) + \Bigl(\frac{\partial}{\partial s_j}(\rho_j w^j)\Bigr) \biggr) \displaybreak[0] \\
\nonumber
  &= -\frac{w^j}{\rho_j^2}\rho_i\presup{\tau}{(i)}_j + \frac1{\rho_j^2}\rho_i\presup{\omega}{(i)}^j + \frac{\rho_i}{\lambda_j-\lambda_i}\Bigl(\frac{\partial}{\partial s_j}w^j\Bigr) \displaybreak[0] \\
\nonumber
  &= -\frac{w^j}{\rho_j^2}\rho_i\presup{\tau}{(i)}_j + \frac1{\rho_j^2}\rho_i\presup{\omega}{(i)}^j + \frac{\rho_i}{\lambda_j-\lambda_i}\sum_{p,q}\cull{\gamma jpq}w^pw^q \displaybreak[0] \\
\nonumber
  &= -\frac{w^j}{\rho_j^2}\rho_i\presup{\tau}{(i)}_j + \frac1{\rho_j^2}\rho_i\presup{\omega}{(i)}^j + \frac{\rho_i}{\lambda_j-\lambda_i}\biggl( \underbrace{\cull{\gamma jii}}_{\phantom0=0}(w^i)^2 + \Bigl(\sum_{\stack p{p\neq i}}2\cull{\gamma jip}w^p\Bigr)w^i + \sum_{\makebox[0pt][c]{\tiny $\stack{p,q}{p\neq i,q\neq i}$}}\cull{\gamma jpq}w^pw^q \biggr) \displaybreak[0] \\
\label{eq:dziwj_fully_expanded}
  &= -\frac{w^j}{\rho_j^2}\rho_i\presup{\tau}{(i)}_j + \frac1{\rho_j^2}\rho_i\presup{\omega}{(i)}^j + \frac1{\lambda_j-\lambda_i}\Bigl(\sum_{\stack p{p\neq i}}2\cull{\gamma jip}w^p\Bigr)v^i + \frac1{\lambda_j-\lambda_i}\Bigl(\sum_{\makebox[0pt][c]{\tiny $\stack{p,q}{p\neq i,q\neq i}$}}\cull{\gamma jpq}w^pw^q\Bigr)\rho_i .
\end{align}

Now, regarding the evolution equation of $\mu_i$, we get from \eqref{eq:dsirhoi}, taking into account \eqref{eq:dziu}, 
\begin{align*}
  \frac{\partial}{\partial s_i}\mu_i &= \frac{\partial}{\partial s_i}\Bigl(\frac{\partial}{\partial z_i}\rho_i\Bigr) = \frac{\partial}{\partial z_i}\Bigl(\frac{\partial}{\partial s_i}\rho_i\Bigr) \displaybreak[0] \\
  &= \frac{\partial}{\partial z_i}\biggl( \cull{ciii}v^i + \Bigl( \sum_{\stack m{m\neq i}}\cull{ciim}w^m \Bigr)\rho_i \biggr) \displaybreak[0] \\
  &= \bigl(D_{e_i}\cull{ciii}\bigr)(v^i)^2 + \Bigl(\sum_{\stack p{p\neq i}}\bigl(D_{e_p}\cull{ciii}\bigr)w^p\Bigr)\rho_iv^i + \cull{ciii}\Bigl(\frac{\partial}{\partial z_i}v^i\Bigr) \\
  &\phantom{=\ } {} + \Bigl(\sum_{\stack m{m\neq i}}\bigl(D_{e_i}\cull{ciim}\bigr)w^m\Bigr)\rho_iv^i + \Bigl(\sum_{\stack{m,p}{m\neq i,p\neq i}}\bigl(D_{e_p}\cull{ciim}\bigr)w^mw^p\Bigr)\rho_i^2 \\
  &\phantom{=\ } {} + \biggl(\sum_{\stack m{m\neq i}}\cull{ciim}\Bigl(\frac{\partial}{\partial z_i}w^m\Bigr)\biggr)\rho_i + \Bigl(\sum_{\stack m{m\neq i}}\cull{ciim}w^m\Bigr)\Bigl(\frac{\partial}{\partial z_i}\rho_i\Bigr) \displaybreak[0] \\
  &= \Bigl(\sum_{\stack m{m\neq i}}\cull{ciim}w^m\Bigr)\mu_i + \cull{ciii}\nu^i + \bigl(D_{e_i}\cull{ciii}\bigr)(v^i)^2 + \Bigl(\sum_{\stack{m,p}{m\neq i,p\neq i}}\bigl(D_{e_p}\cull{ciim}\bigr)w^mw^p\Bigr)\rho_i^2 \\
  &\phantom{=\ } {} + \biggl(\sum_{\stack m{m\neq i}}\Bigl(\bigl(D_{e_m}\cull{ciii}\bigr)+\bigl(D_{e_i}\cull{ciim}\bigr)\Bigr)w^m\biggr)\rho_iv^i + \biggl(\sum_{\stack m{m\neq i}}\cull{ciim}\Bigl(\frac{\partial}{\partial z_i}w^m\Bigr)\biggr)\rho_i .
\end{align*}
Using \eqref{eq:dziwj_fully_expanded} to re--express the last term, we obtain 
\begin{align}
\label{eq:dsimui}
  \frac{\partial}{\partial s_i}\mu_i &= \Bigl(\sum_{\makebox[0pt][c]{\tiny $\stack m{m\neq i}$}}\cull{ciim}w^m\Bigr)\mu_i + \cull{ciii}\nu^i - \Bigl(\sum_{\makebox[0pt][c]{\tiny $\stack m{m\neq i}$}}\frac{\cull{ciim}w^m}{\rho_m^2}\presup{\tau}{(i)}_m\Bigr)\rho_i^2 + \Bigl(\sum_{\makebox[0pt][c]{\tiny $\stack m{m\neq i}$}}\frac{\cull{ciim}}{\rho_m^2}\presup{\omega}{(i)}^m\Bigr)\rho_i^2 \\
\nonumber
  &\phantom{=\ } {} + \bigl(D_{e_i}\cull{ciii}\bigr)(v^i)^2 + \Bigl(\sum_{\makebox[0pt][c]{\tiny $\stack{m,p}{m\neq i,p\neq i}$}}\bigl(D_{e_p}\cull{ciim}\bigr)w^mw^p\Bigr)\rho_i^2 + \biggl(\sum_{\makebox[0pt][c]{\tiny $\stack m{m\neq i}$}}\frac{\cull{ciim}}{\lambda_m-\lambda_i}\Bigl(\sum_{\makebox[0pt][c]{\tiny $\stack{p,q}{p\neq i,q\neq i}$}}\cull{\gamma mpq}w^pw^q\Bigr)\biggr)\rho_i^2 \\
\nonumber
  &\phantom{=\ } {} + \biggl(\sum_{\makebox[0pt][c]{\tiny $\stack m{m\neq i}$}}\Bigl(\bigl(D_{e_m}\cull{ciii}\bigr)+\bigl(D_{e_i}\cull{ciim}\bigr)\Bigr)w^m\biggr)\rho_iv^i + \biggl(\sum_{\makebox[0pt][c]{\tiny $\stack m{m\neq i}$}}\frac{\cull{ciim}}{\lambda_m-\lambda_i}\Bigl(\sum_{\makebox[0pt][c]{\tiny $\stack p{p\neq i}$}}2\cull{\gamma mip}w^p\Bigr)\biggr)\rho_iv^i .
\end{align}

Similarly, we get from \eqref{eq:dsivi}, taking into account \eqref{eq:dziu}, \eqref{eq:gammaill_is_zero}, 
\begin{align*}
  \frac{\partial}{\partial s_i}\nu^i &= \frac{\partial}{\partial s_i}\Bigl(\frac{\partial}{\partial z_i}v^i\Bigr) = \frac{\partial}{\partial z_i}\Bigl(\frac{\partial}{\partial s_i}v^i\Bigr) \displaybreak[0] \\
  &= \frac{\partial}{\partial z_i}\biggl( \Bigl(\sum_{\stack m{m\neq i}}\bigl(2\cull{\gamma iim}+\cull{ciim}\bigr)w^m\Bigr)v^i + \Bigl(\sum_{\stack{l,m}{\stack{l\neq i,m\neq i}{l\neq m}}} \cull{\gamma ilm}w^lw^m \Bigr) \rho_i \biggr) \displaybreak[0] \\
  &= \Bigl(\sum_{\stack m{m\neq i}}\bigl(2D_{e_i}\cull{\gamma iim}+D_{e_i}\cull{ciim}\bigr)w^m\Bigr)(v^i)^2 + \Bigl(\sum_{\makebox[0pt][c]{\tiny $\stack{m,p}{m\neq i,p\neq i}$}}\bigl(2D_{e_p}\cull{\gamma iim}+D_{e_p}\cull{ciim}\bigr)w^mw^p\Bigr)\rho_iv^i \\
  &\phantom{=\ } {} + \biggl(\sum_{\stack m{m\neq i}}\bigl(2\cull{\gamma iim}+\cull{ciim}\bigr)\Bigl(\frac{\partial}{\partial z_i}w^m\Bigr)\biggr)v^i + \Bigl(\sum_{\stack m{m\neq i}}\bigl(2\cull{\gamma iim}+\cull{ciim}\bigr)w^m\Bigr)\Bigl(\frac{\partial}{\partial z_i}v^i\Bigr) \\
  &\phantom{=\ } {} + \Bigl(\sum_{\makebox[0pt][c]{\tiny $\stack{l,m}{\stack{l\neq i,m\neq i}{l\neq m}}$}} \bigl(D_{e_i}\cull{\gamma ilm}\bigr)w^lw^m \Bigr) \rho_iv^i + \Bigl(\sum_{\makebox[0pt][c]{\tiny $\stack{l,m,p}{\stack{l\neq i,m\neq i,p\neq i}{l\neq m}}$}} \bigl(D_{e_p}\cull{\gamma ilm}\bigr)w^lw^mw^p \Bigr) \rho_i^2 \\
  &\phantom{=\ } {} + 2\biggl(\sum_{\makebox[0pt][c]{\tiny $\stack{l,m}{\stack{l\neq i,m\neq i}{l\neq m}}$}} \cull{\gamma ilm}w^l\Bigl(\frac{\partial}{\partial z_i}w^m\Bigr) \biggr) \rho_i + \Bigl(\sum_{\makebox[0pt][c]{\tiny $\stack{l,m}{\stack{l\neq i,m\neq i}{l\neq m}}$}} \cull{\gamma ilm}w^lw^m \Bigr) \Bigl(\frac{\partial}{\partial z_i}\rho_i\Bigr) \displaybreak[0] \\
  &= \Bigl(\sum_{\makebox[0pt][c]{\tiny $\stack{l,m}{\stack{l\neq i,m\neq i}{l\neq m}}$}} \cull{\gamma ilm}w^lw^m \Bigr) \mu_i + \Bigl(\sum_{\stack m{m\neq i}}\bigl(2\cull{\gamma iim}+\cull{ciim}\bigr)w^m\Bigr) \nu^i \\
  &\phantom{=\ } {} + \Bigl(\sum_{\stack m{m\neq i}}\bigl(2D_{e_i}\cull{\gamma iim}+D_{e_i}\cull{ciim}\bigr)w^m\Bigr)(v^i)^2 + \Bigl(\sum_{\makebox[0pt][c]{\tiny $\stack{l,m,p}{\stack{l\neq i,m\neq i,p\neq i}{l\neq m}}$}} \bigl(D_{e_p}\cull{\gamma ilm}\bigr)w^lw^mw^p \Bigr) \rho_i^2 \\
  &\phantom{=\ } {} + \Bigl(\sum_{\makebox[0pt][c]{\tiny $\stack{m,p}{m\neq i,p\neq i}$}}\bigl(D_{e_i}\cull{\gamma imp}+2D_{e_p}\cull{\gamma iim}+D_{e_p}\cull{ciim}\bigr)w^mw^p\Bigr)\rho_iv^i \\
  &\phantom{=\ } {} + \biggl(\sum_{\makebox[0pt][c]{\tiny $\stack{l,m}{\stack{l\neq i,m\neq i}{l\neq m}}$}} 2\cull{\gamma ilm}w^l\Bigl(\frac{\partial}{\partial z_i}w^m\Bigr) \biggr) \rho_i + \biggl(\sum_{\stack m{m\neq i}}\bigl(2\cull{\gamma iim}+\cull{ciim}\bigr)\Bigl(\frac{\partial}{\partial z_i}w^m\Bigr)\biggr)v^i .
\end{align*}
Using \eqref{eq:dziwj_fully_expanded} to re--express the last two terms, we obtain 
\begin{align}
\label{eq:dsinui}
  \frac{\partial}{\partial s_i}\nu^i &= \Bigl(\sum_{\makebox[0pt][c]{\tiny $\stack{l,m}{\stack{l\neq i,m\neq i}{l\neq m}}$}} \cull{\gamma ilm}w^lw^m \Bigr) \mu_i + \Bigl(\sum_{\stack m{m\neq i}}\bigl(2\cull{\gamma iim}+\cull{ciim}\bigr)w^m\Bigr) \nu^i \\
\nonumber
  &\phantom{=\ } {} - \Bigl(\sum_{\makebox[0pt][c]{\tiny $\stack{l,m}{\stack{l\neq i,m\neq i}{l\neq m}}$}} \frac{2\cull{\gamma ilm}w^lw^m}{\rho_m^2}\presup{\tau}{(i)}_m\Bigr) \rho_i^2 + \Bigl(\sum_{\makebox[0pt][c]{\tiny $\stack{l,m}{\stack{l\neq i,m\neq i}{l\neq m}}$}} \frac{2\cull{\gamma ilm}w^l}{\rho_m^2}\presup{\omega}{(i)}^m\Bigr) \rho_i^2 \\
\nonumber
  &\phantom{=\ } {} - \Bigl(\sum_{\stack m{m\neq i}}\frac{\bigl(2\cull{\gamma iim}+\cull{ciim}\bigr)w^m}{\rho_m^2}\presup{\tau}{(i)}_m\Bigr)\rho_iv^i + \Bigl(\sum_{\stack m{m\neq i}}\frac{2\cull{\gamma iim}+\cull{ciim}}{\rho_m^2}\presup{\omega}{(i)}^m\Bigr)\rho_iv^i \\
\nonumber
  &\phantom{=\ } {} + \Bigl(\sum_{\stack m{m\neq i}}\bigl(2D_{e_i}\cull{\gamma iim}+D_{e_i}\cull{ciim}\bigr)w^m\Bigr)(v^i)^2 + \biggl(\sum_{\makebox[0pt][c]{\tiny $\stack m{m\neq i}$}}\frac{2\cull{\gamma iim}+\cull{ciim}}{\lambda_m-\lambda_i}\Bigl(\sum_{\makebox[0pt][c]{\tiny $\stack{p}{p\neq i}$}}2\cull{\gamma mip}w^p\Bigr)\biggr)(v^i)^2 \\
\nonumber
  &\phantom{=\ } {} + \Bigl(\sum_{\makebox[0pt][c]{\tiny $\stack{l,m,p}{\stack{l\neq i,m\neq i,p\neq i}{l\neq m}}$}} \bigl(D_{e_p}\cull{\gamma ilm}\bigr)w^lw^mw^p \Bigr) \rho_i^2 + \biggl(\sum_{\makebox[0pt][c]{\tiny $\stack{l,m}{\stack{l\neq i,m\neq i}{l\neq m}}$}} \frac{2\cull{\gamma ilm}w^l}{\lambda_m-\lambda_i}\Bigl(\sum_{\makebox[0pt][c]{\tiny $\stack{p,q}{p\neq i,q\neq i}$}}\cull{\gamma mpq}w^pw^q\Bigr) \biggr) \rho_i^2 \\
\nonumber
  &\phantom{=\ } {} + \Bigl(\sum_{\makebox[0pt][c]{\tiny $\stack{m,p}{m\neq i,p\neq i}$}}\bigl(D_{e_i}\cull{\gamma imp}+2D_{e_p}\cull{\gamma iim}+D_{e_p}\cull{ciim}\bigr)w^mw^p\Bigr)\rho_iv^i \\
\nonumber
  &\phantom{=\ } {} + \biggl(\sum_{\makebox[0pt][c]{\tiny $\stack{l,m}{\stack{l\neq i,m\neq i}{l\neq m}}$}} \frac{2\cull{\gamma ilm}w^l}{\lambda_m-\lambda_i}\Bigl(\sum_{\makebox[0pt][c]{\tiny $\stack{p}{p\neq i}$}}2\cull{\gamma mip}w^p\Bigr) \biggr) \rho_iv^i + \biggl(\sum_{\makebox[0pt][c]{\tiny $\stack m{m\neq i}$}}\frac{2\cull{\gamma iim}+\cull{ciim}}{\lambda_m-\lambda_i}\Bigl(\sum_{\makebox[0pt][c]{\tiny $\stack{p,q}{p\neq i,q\neq i}$}}\cull{\gamma mpq}w^pw^q\Bigr)\biggr)\rho_iv^i .
\end{align}

We have thus obtained a inhomogeneous system for $(\mu_i,\nu_i)$ whose coefficients depend only on first order quantities, except for some terms in the inhomogeneous part, which also depend on the second order quantities $(\presup{\tau}{(i)}_j,\presup{\omega}{(i)}^j)$ ($j\in\{1,\ldots,N\}\setminus\{i\}$). The first order quantities being already dealt with in the last section, we need to bound those second order quantities in order to obtain bounds on $(\mu_i,\nu^i)$. For this, we compute their evolution equations along $j^{\text{th}}$ characteristics inside $\mathcal R_i$ ($j\neq i$) using bi--characteristic coordinates (recall that $y_i$ is a parameter along $\mathcal C_j(z_j)$).

First, notice that from \eqref{eqs:dbccoords_partial}, we have for every $j,k\in\{1,\ldots,N\}\setminus\{i\}$, 
\begin{equation} \label{eq:partialyj_in_terms_of_partialyk}
  \frac{\partial}{\partial y_j} = \frac{\rho_j}{\lambda_i-\lambda_j}\frac{\partial}{\partial s_i} = \frac{\lambda_i-\lambda_k}{\lambda_i-\lambda_j}\frac{\rho_j}{\rho_k}\frac{\rho_k}{\lambda_i-\lambda_k}\frac{\partial}{\partial s_i} = \frac{\lambda_k-\lambda_i}{\lambda_j-\lambda_i}\frac{\rho_j}{\rho_k}\frac{\partial}{\partial y_k} .
\end{equation}
We will use this identity to pass from $(y_i,y_j)$--coordinates to $(y_i,y_k)$--coordinates. Also, \eqref{eqs:dbccoords_partial} and \eqref{eq:dsiu} imply 
\begin{equation}
\label{eq:dyju}
  \frac{\partial}{\partial y_j}u = \frac{\rho_j}{\lambda_i-\lambda_j}\Bigl(\frac{\partial}{\partial s_i}u\Bigr) = \frac{\rho_j}{\lambda_j-\lambda_i}\sum_{\stack k{k\neq i}}(\lambda_k-\lambda_i)w^ke_k ,
\end{equation}
which we will use to compute the partial derivatives with respect to $y_j$ of functions of $u$. Finally, for any $j\neq i$, we compute using \eqref{eqs:dbccoords_partial}, \eqref{eq:dsirhoi}, \eqref{eq:vi}, \eqref{eq:dsiwi}, 
\begin{align}
\label{eq:dyirhoj}
  \frac{\partial}{\partial y_i}\rho_j &= \frac{\rho_i}{\lambda_j-\lambda_i}\frac{\partial}{\partial s_j}\rho_j = \frac{\rho_i\rho_j}{\lambda_j-\lambda_i}\Bigl(\sum_p\cull{cjjp}w^p\Bigr) , \\
\nonumber
  \frac{\partial}{\partial y_i}{v^j} &= \frac{\rho_i}{\lambda_j-\lambda_i}\frac{\partial}{\partial s_j}v^j = \frac{\rho_i}{\lambda_j-\lambda_i}\frac{\partial}{\partial s_j}(\rho_jw^j) \\
\label{eq:dyivj}
  &= \frac{\rho_i\rho_j}{\lambda_j-\lambda_i}\biggl( \Bigl(\sum_p\cull{cjjp}w^p\Bigr)w^j + \sum_{p,q}\cull{\gamma jpq}w^pw^q \biggr) .
\end{align}
We are now ready to compute the evolution equations for $\presup{\tau}{(i)}_j$ and $\presup{\omega}{(i)}^j$ along a $j^{\text{th}}$ characteristic inside $\mathcal R_i$ ($j\neq i$).

From \eqref{eqs:mt_tauijomegaij}, \eqref{eq:dyirhoj}, \eqref{eq:vi}, we have 
\begin{align*}
  \frac{\partial}{\partial y_i}\presup{\tau}{(i)}_j &= \frac{\partial}{\partial y_i}\Bigl(\frac{\partial}{\partial y_j}\rho_j\Bigr) = \frac{\partial}{\partial y_j}\Bigl(\frac{\partial}{\partial y_i}\rho_j\Bigr) \displaybreak[0] \\
  &= \frac{\partial}{\partial y_j}\biggl( \frac{\rho_i\rho_j}{\lambda_j-\lambda_i}\Bigl(\sum_p\cull{cjjp}w^p\Bigr) \biggr) = -\frac{\frac{\partial}{\partial y_j}(\lambda_j-\lambda_i)}{(\lambda_j-\lambda_i)^2}\Bigl(\rho_i\rho_j\sum_p\cull{cjjp}w^p\Bigr) \\
  &\phantom{=\ } {} + \frac{\rho_j}{\lambda_j-\lambda_i}\biggl( \rho_i\sum_m\Bigl(\frac{\partial}{\partial y_j}\cull{cjjp}\Bigr)w^p + \Bigl(\sum_{\stack m{m\neq i}}\cull{cjjp}w^p\Bigr)\Bigl(\frac{\partial}{\partial y_j}\rho_i\Bigr) + \cull{cjji}\Bigl(\frac{\partial}{\partial y_j}v^i\Bigr) \biggr) \\
  &\phantom{=\ } {} + \frac{\rho_i}{\lambda_j-\lambda_i}\biggl( \Bigl(\sum_{\stack p{p\neq j}}\cull{cjjp}w^p\Bigr)\Bigl(\frac{\partial}{\partial y_j}\rho_j\Bigr) + \cull{cjjj}\Bigl(\frac{\partial}{\partial y_j}v^j\Bigr) \biggr) \\
  &\phantom{=\ } {} + \frac{\rho_i\rho_j}{\lambda_j-\lambda_i}\Biggl( \sum_{\stack p{p\neq i,p\neq j}}\frac{\cull{cjjp}}{\rho_p}\biggl(-w^p\Bigl(\frac{\partial}{\partial y_j}\rho_p\Bigr)+\Bigl(\frac{\partial}{\partial y_j}v^p\Bigr)\biggr) \Biggr)
\end{align*}
Using then \eqref{eqs:mt_tauijomegaij}, \eqref{eq:partialyj_in_terms_of_partialyk}, \eqref{eq:vi}, we can write this as 
\begin{align}
\label{eq:dyitauij}
  \frac{\partial}{\partial y_i}\presup{\tau}{(i)}_j &= \frac1{\lambda_j-\lambda_i}\biggl( \Bigl(\sum_{\makebox[0pt][c]{\tiny $\stack p{p\neq i,p\neq j}$}}\cull{cjjp}w^p\Bigr)\presup{\tau}{(i)}_j\rho_i + \cull{cjji}\presup{\tau}{(i)}_jv^i + \cull{cjjj}\presup{\omega}{(i)}^j\rho_i \biggr) \\
\nonumber
  &\phantom{=\ } {} + \frac{\rho_i\rho_j^2}{(\lambda_j-\lambda_i)^2}\Bigl( \sum_{\makebox[0pt][c]{\tiny $\stack p{p\neq i,p\neq j}$}}\frac{\cull{cjjp}}{\rho_p^2}(\lambda_p-\lambda_i)\bigl( - \presup{\tau}{(i)}_pw^p + \presup{\omega}{(i)}^p \bigr) \Bigr) \\
\nonumber
  &\phantom{=\ } {} + \frac{\rho_j^2}{(\lambda_j-\lambda_i)^2}\presup{M}{(i)}_j ,
\end{align}
where the inhomogeneous term $\frac{\rho_j^2}{(\lambda_j-\lambda_i)^2}\presup{M}{(i)}_j$, with 
\begin{align*}
  \presup{M}{(i)}_j &= -\frac{\frac{\partial}{\partial y_j}(\lambda_j-\lambda_i)}{\rho_j}\Bigl(\rho_i\sum_p\cull{cjjp}w^p\Bigr) \\
  &\phantom{=\ } {} + \frac{\lambda_j-\lambda_i}{\rho_j}\biggl( \rho_i\sum_p\Bigl(\frac{\partial}{\partial y_j}\cull{cjjp}\Bigr)w^p + \Bigl(\sum_{\stack p{p\neq i}}\cull{cjjp}w^p\Bigr)\Bigl(\frac{\partial}{\partial y_j}\rho_i\Bigr) + \cull{cjji}\Bigl(\frac{\partial}{\partial y_j}v^i\Bigr) \biggr) ,
\end{align*}
depends only on first order quantities which are bounded (cf.~\eqref{eqs:tV_J_V_U_bootstrap_results}, \eqref{eqs:other_bootstrap_results}) inside $\mathcal R_i$. This is immediate from \eqref{eq:dyju}, \eqref{eq:Dellambdak}, \eqref{eq:dyirhoj}, \eqref{eq:dyivj}, for 
\begin{align}
\label{eq:Mij}
  \presup{M}{(i)}_j &= -\Bigl( \sum_{\stack q{q\neq i}}\frac{\lambda_q-\lambda_i}{\lambda_j-\lambda_i}w^q\bigl(\cull{cjjq}-\cull{ciiq}\bigr) \Bigr)\biggl(\Bigl(\sum_{\stack p{p\neq i}}\cull{cjjp}w^p\Bigr)\rho_i + \cull{cjji}v^i \biggr) \\
\nonumber
  &\phantom{=\ } {} + \Bigl(\sum_{\makebox[0pt][c]{\tiny $\stack{p,q}{p\neq i,q\neq i}$}}\bigl(D_{e_q}\cull{cjjp}\bigr)(\lambda_q-\lambda_i)w^pw^q\Bigr)\rho_i + \Bigl(\sum_{\stack q{q\neq i}}\bigl(D_{e_q}\cull{cjji}\bigr)(\lambda_q-\lambda_i)w^q\Bigr)v^i \\
\nonumber
  &\phantom{=\ } {} - \Bigl(\sum_{\stack p{p\neq i}}\cull{cjjp}w^p\Bigr)\biggl( \Bigl(\sum_{\stack q{q\neq i}}\cull{ciiq}w^q\Bigr)\rho_i + \cull{ciii}v^i \biggr) \\
\nonumber
  &\phantom{=\ } {} - \cull{cjji}\biggl( \Bigl(\sum_{\makebox[0pt][c]{\tiny $\stack{p,q}{\stack{p\neq i,q\neq i}{p\neq q}}$}}\cull{\gamma ipq}w^pw^q\Bigr)\rho_i + \Bigl(\sum_{\stack p{p\neq i}}(2\cull{\gamma iip}+\cull{ciip})w^p\Bigr)v^i \biggr) ,
\end{align}
where we have used \eqref{eq:gammaiii}, \eqref{eq:gammaill_is_zero}, \eqref{eq:gamma_sym}.

Similarly, from \eqref{eqs:mt_tauijomegaij}, \eqref{eq:dyivj}, \eqref{eq:vi}, and using \eqref{eq:gammaiii}, \eqref{eq:gammaill_is_zero}, \eqref{eq:gamma_sym}, we have 
\begin{align*}
  \frac{\partial}{\partial y_i}\presup{\omega}{(i)}^j &= \frac{\partial}{\partial y_i}\Bigl(\frac{\partial}{\partial y_j}v^j\Bigr) = \frac{\partial}{\partial y_j}\Bigl(\frac{\partial}{\partial y_i}v^j\Bigr) \displaybreak[0] \\
  &= \frac{\partial}{\partial y_j}\Biggl( \frac{\rho_i\rho_j}{\lambda_j-\lambda_i}\biggl(\Bigl(\sum_{\stack p{p\neq j}}\cull{cjjp}w^p\Bigr)w^j + \sum_{\stack{p,q}{p\neq q}}\cull{\gamma jpq}w^pw^q \biggr) \Biggr) \displaybreak[0] \\
  &= -\frac{\frac{\partial}{\partial y_j}(\lambda_j-\lambda_i)}{(\lambda_j-\lambda_i)^2}\Biggl(\rho_i\rho_j\biggl(\Bigl(\sum_{\stack p{p\neq j}}\cull{cjjp}w^p\Bigr)w^j+\sum_{\stack{p,q}{p\neq q}}\cull{\gamma jpq}w^pw^q\Bigr)\biggr)\Biggr) \\
  &\phantom{=\ } {} + \frac{\rho_i\rho_j}{\lambda_j-\lambda_i} \Biggl( \biggl(\sum_{\stack p{p\neq j}}\Bigl(\frac{\partial}{\partial y_j}\cull{cjjp}\Bigr)w^p\biggr)w^j + \biggl(\sum_{\stack{p,q}{p\neq q}}\Bigl(\frac{\partial}{\partial y_j}\cull{\gamma jpq}\Bigr)w^pw^q \biggr) \Biggr) \\
  &\phantom{=\ } {} + \frac{\rho_j}{\lambda_j-\lambda_i} \Biggl( \biggl(\Bigl(\sum_{\stack p{p\neq i,p\neq j}}\cull{cjjp}w^p\Bigr)w^j + \sum_{\stack{p,q}{\stack{p\neq i,q\neq i}{p\neq q}}}\cull{\gamma jpq}w^pw^q \biggr)\Bigl(\frac{\partial}{\partial y_j}\rho_i\Bigr) \Biggr) \\
  &\phantom{=\ } {} + \frac{\rho_j}{\lambda_j-\lambda_i} \biggl( \Bigl( \cull{cjji}w^j + \sum_{\stack{p}{p\neq i}}2\cull{\gamma jip}w^p \Bigr)\Bigl(\frac{\partial}{\partial y_j}v^i\Bigr) \biggr) \\
  &\phantom{=\ } {} + \frac{\rho_i}{\lambda_j-\lambda_i} \biggl( \Bigl(\sum_{\stack{p,q}{\stack{p\neq j,q\neq j}{p\neq q}}}\cull{\gamma jpq}w^pw^q \Bigr)\Bigl(\frac{\partial}{\partial y_j}\rho_j\Bigr) \biggr) \\
  &\phantom{=\ } {} + \frac{\rho_i}{\lambda_j-\lambda_i} \biggl( \Bigl(\sum_{\stack p{p\neq j}}\cull{cjjp}w^p + \sum_{\stack p{p\neq j}}2\cull{\gamma jjp}w^p \Bigr)\Bigl(\frac{\partial}{\partial y_j}v^j\Bigr) \biggr) \\
  &\phantom{=\ } {} + \frac{\rho_i\rho_j}{\lambda_j-\lambda_i}\bigggl(\Biggl(\sum_{\stack p{p\neq i,p\neq j}}\frac{\cull{cjjp}}{\rho_p}\biggl(-w^p\Bigl(\frac{\partial}{\partial y_j}\rho_p\Bigr)+\Bigl(\frac{\partial}{\partial y_j}v^p\Bigr)\biggr)\Biggr)w^j \bigggr) \\
  &\phantom{=\ } {} + \frac{\rho_i\rho_j}{\lambda_j-\lambda_i}\Biggl( \sum_{\stack{\stack{p,q}{p\neq q}}{p\neq i,p\neq j}}\frac{2\cull{\gamma jpq}}{\rho_p}w^q\biggl(-w^p\Bigl(\frac{\partial}{\partial y_j}\rho_p\Bigr)+\Bigl(\frac{\partial}{\partial y_j}v^p\Bigr)\biggr) \Biggr)
\end{align*}
With \eqref{eqs:mt_tauijomegaij}, \eqref{eq:partialyj_in_terms_of_partialyk}, \eqref{eq:vi}, we write this as 
\begin{align}
\label{eq:dyiomegaij}
  \frac{\partial}{\partial y_i}\presup{\omega}{(i)}^j &= \frac1{\lambda_j-\lambda_i} \biggl( \Bigl(\sum_{\makebox[0pt][c]{\tiny $\stack{\stack{p,q}{p\neq q}}{\stack{p\neq i,q\neq i}{p\neq j,q\neq j}}$}}\cull{\gamma jpq}w^pw^q \Bigr)\presup{\tau}{(i)}_j\rho_i + \Bigl(\sum_{\makebox[0pt][c]{\tiny $\stack p{p\neq i,p\neq j}$}}2\cull{\gamma jip}w^p\Bigr)\presup{\tau}{(i)}_jv^i \biggr) \\
\nonumber
  &\phantom{=\ } {} + \frac1{\lambda_j-\lambda_i} \biggl( \Bigl(\sum_{\makebox[0pt][c]{\tiny $\stack p{p\neq i,p\neq j}$}}\bigl(\cull{cjjp}+2\cull{\gamma jjp}\bigr)w^p \Bigr)\presup{\omega}{(i)}^j\rho_i + \bigl(\cull{cjji}+2\cull{\gamma jji}\bigr)\presup{\omega}{(i)}^jv^i \biggr) \\
\nonumber
  &\phantom{=\ } {} + \frac{\rho_i\rho_j^2}{(\lambda_j-\lambda_i)^2}\biggl(\Bigl(\sum_{\makebox[0pt][c]{\tiny $\stack p{p\neq i,p\neq j}$}}\frac{\cull{cjjp}+2\cull{\gamma jjp}}{\rho_p^2}(\lambda_p-\lambda_i)\bigl( - \presup{\tau}{(i)}_pw^p + \presup{\omega}{(i)}^p \bigr)\Bigr)w^j \biggr) \\
\nonumber
  &\phantom{=\ } {} + \frac{\rho_i\rho_j^2}{(\lambda_j-\lambda_i)^2}\Bigl( \sum_{\makebox[0pt][c]{\tiny $\stack{\stack{p,q}{p\neq q}}{\stack{p\neq i,q\neq i}{p\neq j,q\neq j}}$}}\frac{2\cull{\gamma jpq}}{\rho_p^2}w^q(\lambda_p-\lambda_i)\bigl( - \presup{\tau}{(i)}_pw^p + \presup{\omega}{(i)}^p \bigr) \Bigr) \\
\nonumber
  &\phantom{=\ } {} + \frac{\rho_j^2}{(\lambda_j-\lambda_i)^2}\biggl(\Bigl( \sum_{\makebox[0pt][c]{\tiny $\stack p{p\neq i,p\neq j}$}}\frac{2\cull{\gamma jip}}{\rho_p^2}(\lambda_p-\lambda_i)\bigl( - \presup{\tau}{(i)}_pw^p + \presup{\omega}{(i)}^p \bigr) \Bigr)v^i\biggr) \\
\nonumber
  &\phantom{=\ } {} + \frac{\rho_j^2}{(\lambda_j-\lambda_i)^2}\presup{N}{(i)}^j ,
\end{align}
where 
\begin{align*}
  \presup{N}{(i)}^j &= -\frac{\frac{\partial}{\partial y_j}(\lambda_j-\lambda_i)}{\rho_j}\Biggl(\rho_i\biggl(\Bigl(\sum_{\stack p{p\neq j}}\cull{cjjp}w^p\Bigr)w^j+\sum_{\stack{p,q}{p\neq q}}\cull{\gamma jpq}w^pw^q\biggr)\Biggr) \\
  &\phantom{=\ } {} + \frac{\lambda_j-\lambda_i}{\rho_j} \Biggl( \rho_i\biggl(\sum_{\stack p{p\neq j}}\Bigl(\frac{\partial}{\partial y_j}\cull{cjjp}\Bigr)w^p\biggr)w^j + \rho_i\biggl(\sum_{\stack{p,q}{p\neq q}}\Bigl(\frac{\partial}{\partial y_j}\cull{\gamma jpq}\Bigr)w^pw^q \biggr) \Biggr) \\
  &\phantom{=\ } {} + \frac{\lambda_j-\lambda_i}{\rho_j} \Biggl( \biggl(\Bigl(\sum_{\stack p{p\neq i,p\neq j}}\cull{cjjp}w^p\Bigr)w^j + \sum_{\stack{p,q}{\stack{p\neq i,q\neq i}{p\neq q}}}\cull{\gamma jpq}w^pw^q \biggr)\Bigl(\frac{\partial}{\partial y_j}\rho_i\Bigr) \Biggr) \\
  &\phantom{=\ } {} + \frac{\lambda_j-\lambda_i}{\rho_j} \biggl( \Bigl( \cull{cjji}w^j + \sum_{\stack{p}{p\neq i}}2\cull{\gamma jip}w^p \Bigr)\Bigl(\frac{\partial}{\partial y_j}v^i\Bigr) \biggr) .
\end{align*}
Again, the inhomogeneous term $\frac{\rho_j^2}{(\lambda_j-\lambda_i)^2}\presup{N}{(i)}^j$ in \eqref{eq:dyiomegaij} depends only on first order quantities which are bounded in $\mathcal R_i$, since \eqref{eq:dyju}, \eqref{eq:Dellambdak}, \eqref{eq:dyirhoj}, \eqref{eq:dyivj}, \eqref{eq:gammaiii}, \eqref{eq:gammaill_is_zero}, \eqref{eq:gamma_sym} imply 
\begin{align}
\label{eq:Nij}
  \presup{N}{(i)}^j &= -\Bigl( \sum_{\stack r{r\neq i}}\frac{\lambda_r-\lambda_i}{\lambda_j-\lambda_i}w^r\bigl(\cull{cjjr}-\cull{ciir}\bigr) \Bigr)\Biggl( \biggl( \Bigl(\sum_{\makebox[0pt][c]{\tiny $\stack p{p\neq i,p\neq j}$}}\cull{cjjp}w^p\Bigr)w^j+\sum_{\makebox[0pt][c]{\tiny $\stack{p,q}{\stack{p\neq i,q\neq i}{p\neq q}}$}}\cull{\gamma jpq}w^pw^q\biggr)\rho_i\Biggr) \\
\nonumber
  &\phantom{=\ } {} -\Bigl( \sum_{\stack r{r\neq i}}\frac{\lambda_r-\lambda_i}{\lambda_j-\lambda_i}w^r\bigl(\cull{cjjr}-\cull{ciir}\bigr) \Bigr)\biggl( \Bigl(\cull{cjji}w^j+\sum_{\makebox[0pt][c]{\tiny $\stack p{p\neq i}$}}2\cull{\gamma jip}w^p\Bigr)v^i \biggr) \\
\nonumber
  &\phantom{=\ } {} + \biggl( \Bigl(\sum_{\makebox[0pt][c]{\tiny $\stack{p,r}{\stack{p\neq i,r\neq i}{p\neq j}}$}}\bigl(D_{e_r}\cull{cjjp}\bigr)(\lambda_r-\lambda_i)w^pw^r\Bigr)w^j + \sum_{\makebox[0pt][c]{\tiny $\stack{p,q,r}{\stack{p\neq i,q\neq i}{r\neq i,p\neq q}}$}}\bigl(D_{e_r}\cull{\gamma jpq}\bigr)(\lambda_r-\lambda_i)w^pw^qw^r \biggr)\rho_i \\
\nonumber
  &\phantom{=\ } {} + \biggl( \Bigl(\sum_{\makebox[0pt][c]{\tiny $\stack r{r\neq i}$}}\bigl(D_{e_r}\cull{cjji}\bigr)(\lambda_r-\lambda_i)w^r\Bigr)w^j + \sum_{\makebox[0pt][c]{\tiny $\stack{p,r}{p\neq i,r\neq i}$}}\bigl(2D_{e_r}\cull{\gamma jip}\bigr)(\lambda_r-\lambda_i)w^pw^r \biggr)v^i \\
\nonumber
  &\phantom{=\ } {} - \biggl(\Bigl(\sum_{\makebox[0pt][c]{\tiny $\stack p{p\neq i,p\neq j}$}}\cull{cjjp}w^p\Bigr)w^j + \sum_{\makebox[0pt][c]{\tiny $\stack{p,q}{\stack{p\neq i,q\neq i}{p\neq q}}$}}\cull{\gamma jpq}w^pw^q \biggr)\biggl( \Bigl(\sum_{\makebox[0pt][c]{\tiny $\stack r{r\neq i}$}}\cull{ciir}w^r\Bigr)\rho_i + \cull{ciii}v^i \biggr) \\
\nonumber
  &\phantom{=\ } {} - \Bigl( \cull{cjji}w^j + \sum_{\makebox[0pt][c]{\tiny $\stack r{r\neq i}$}}2\cull{\gamma jir}w^r \Bigr)\biggl( \Bigl(\sum_{\makebox[0pt][c]{\tiny $\stack{p,q}{\stack{p\neq i,q\neq i}{p\neq q}}$}}\cull{\gamma ipq}w^pw^q\Bigr)\rho_i + \Bigl(\sum_{\stack p{p\neq i}}(2\cull{\gamma iip}+\cull{ciip})w^p\Bigr)v^i \biggr) .
\end{align}

We are now ready to show statement \eqref{thm:main_thm_secondorder} of Theorem~\ref{thm:main_thm}.

\subsection{The bounds on the second order quantities} \label{subsec:2ndorder_bounds}

We first establish the boundedness of $\frac{\partial}{\partial x}w^i$ in $\rz\times[0,t_\star)\setminus\bigcup_k\mathcal R_k$. This, together with a bound on $\frac{\partial}{\partial x}\rho_i$ in the same region, will yield bounded initial conditions at the boundary of $\mathcal R_i$ for $\presup{\tau}{(i)}_j$ and $\presup{\omega}{(i)}^j$, which are necessary to show the boundedness of the latter. This will then finally yield bounded inhomogeneous terms in the evolution equations along $i^{\text{th}}$ characteristics for $\mu_i$ and $\nu_i$ inside $\mathcal R_i$, making it possible to show bounds on those quantities, too.

Our starting point is the evolution equation \eqref{eq:dsidxwi} for $\frac{\partial}{\partial x}w^i$ along an $i^{\text{th}}$ characteristic. For $i\in\{1,\ldots,N\}$, let $\mathcal S_i$ denote the region in $\mathcal R_i$ where every $j^{\text{th}}$ characteristic ($j\neq i$) stays inside $\mathcal R_i$ for increasing times, i.e., 
\begin{multline} \label{eq:mathcalSi}
  \mathcal S_i = \Bigl\{ (x,t)\in \mathcal R_i \Bigm| \forall j\in\{1,\ldots,N\}\setminus\{i\} \text{ we have:} \\
  \text{if } z_j \text{ is such that } X_j(z_j,t)=x , \text{ then } \bigl(X_j(z_j,s),s\bigr)\in \mathcal R_i, \forall s\in[t,t_\star) \Bigr\} .
\end{multline}
Now define 
\begin{equation} \label{eq:Rchecki}
  \widecheck{\mathcal R}_i = \mathcal R_i \setminus \mathcal S_i .
\end{equation}
Then, for every point in $\widecheck{\mathcal R}_i$, there is a $j\in\{1,\ldots,N\}\setminus\{i\}$ such that the $j^{\text{th}}$ characteristic that passes through that point has to leave $\mathcal R_i$ before $t_\star$. As a consequence, we obtain from \eqref{eq:rhoi_lwr_bd_Cj_cap_Ri}, that 
\begin{equation} \label{eq:rhoi_bdd_in_Rchecki}
  \frac1{\rho_i} = \mathcal O\Bigl(\frac1\vartheta\Bigr) \qquad\text{ in } \widecheck{\mathcal R}_i .
\end{equation}
By \eqref{eq:vi}, \eqref{eqs:tV_J_V_U_bootstrap_results} and \eqref{eq:W0_is_Otheta}, this implies that 
\begin{equation} \label{eq:wi_bdd_in_Rchecki}
  w^i = \mathcal O\Bigl( \frac{W_0}{\vartheta} \Bigr) = \mathcal O(1) \qquad\text{ in } \widecheck{\mathcal R}_i .
\end{equation}
Moreover, by \eqref{eqs:tV_J_V_U_bootstrap_results}, recall that $w^i = \mathcal O(\vartheta^2)$ outside $\mathcal R_i$. We conclude that, for all $i\in\{1,\ldots,N\}$, 
\begin{equation} \label{eq:wi_bdd_outside_mathcalSi}
  w^i = \mathcal O(1) \qquad\text{ in } \rz\times[0,t_\star)\setminus \mathcal S_i .
\end{equation}
Now fix $(x,t)\in\rz\times[0,t_\star)\setminus \bigcup_k\mathcal S_k=\bigl(\rz\times[0,t_\star)\setminus\bigcup_k \mathcal R_k\bigr)\bigcup\bigl(\bigcup_l\widecheck{\mathcal R}_l\bigr)$, and assume $i$ and $z_i$ are such that $X_i(z_i,t)=x$ (i.e., $\mathcal C_i(z_i)$ passes through $(x,t)$). Then 
\begin{equation*}
  \Bigl\{ \bigl(X_i(z_i,s),s\bigr) \Bigm| s\in[0,t] \Bigr\} \bigcap \Bigl( \bigcup_k \mathcal S_k \Bigr) = \emptyset .
\end{equation*}
This means that we can integrate \eqref{eq:dsidxwi} along $\mathcal C_i(z_i)$ enjoying bounds on all the coefficients.

Indeed, similarly to \eqref{eq:Vt}, consider 
\begin{equation} \label{eq:Vprimet}
  V'(t) = \max_i \sup_{\stack{(x',t')\not\in\bigcup_k\mathcal S_k}{0\leq t'\leq t}} \Bigl|\frac{\partial}{\partial x}w^i(x',t')\Bigr| .
\end{equation}
Notice that $V'$ is non--decreasing (and, actually, even continuous) in $t$. By \eqref{eq:dsidxwi} and the preceding argumentation, we have for all $i\in\{1,\ldots,N\}$, as well as for every $x$ and almost every $t$ such that $(x,t)\in\rz\times[0,t_\star)\setminus\bigcup_k\mathcal S_k$, 
\begin{equation*}
  \frac{\partial}{\partial s_i}\Bigabs{\frac{\partial}{\partial x}w^i(x,t)} \leq \biggabs{\frac{\partial}{\partial s_i}\Bigl(\frac{\partial}{\partial x}w^i(x,t)\Bigr)} \leq C V'(t) + D ,
\end{equation*}
where $C$ and $D$ are constants, depending only on $a$, $\delta$ and the $C^{K,1}$--norm of $f_0$. We can thus argue entirely analogously to how we obtained \eqref{eq:wiabs_initialestimate} to conclude that 
\begin{equation} \label{eq:Vprimet_bd}
  V'(t) \leq e^{Ct}\Bigl(V'(0)+\frac DC\Bigr) - \frac DC ,
\end{equation}
where $V'(0)$ is bounded by the fact that $f_0$ is $C^K$ and has compact support in $I_0=[-1,1]$. Since $t$ is finite by \eqref{eq:tstar_range}, \eqref{eq:T0}, we thus established that, indeed, $\frac{\partial}{\partial x}w^i$ is bounded for all $i$ in the region $\rz\times[0,t_\star)\setminus\bigcup_k\mathcal S_k$. In particular, $\frac{\partial}{\partial x}w^i$ remains also bounded outside all the strips, as claimed in Theorem~\ref{thm:main_thm}~\eqref{thm:main_thm_secondorder}.

Notice that, since $\rho_i=\mathcal O(1)$ by \eqref{eqs:other_bootstrap_results} and \eqref{eq:rhoi_range_outside_Ri}, the foregoing observations and \eqref{eq:dsidxrhoi} immediately entail the boundedness of $\frac{\partial}{\partial x}\rho_i$ in $\rz\times[0,t_\star)\setminus\bigcup_k\mathcal S_k$ and thus in $\rz\times[0,t_\star)\setminus\bigcup_k\mathcal R_k$ ($\forall i\in\{1,\ldots,N\}$). Note, however, that the bounds on $\frac{\partial}{\partial x}w^i$ and $\frac{\partial}{\partial x}\rho_i$ get worse, the smaller we pick $\vartheta$. Indeed, they blow up exponentially as $\vartheta$ goes to zero (see \eqref{eq:Vprimet_bd}, \eqref{eq:T0} and \eqref{eq:W0p_pos}).

We now focus our attention on $\bigl\{\presup{\tau}{(i)}_j\bigr\}_{j\neq i}$ and $\bigl\{\presup{\omega}{(i)}^j\bigr\}_{j\neq i}$ inside $\mathcal R_i$. So fix $i\in\{1,\ldots,N\}$. If we want to integrate \eqref{eq:dyitauij} and \eqref{eq:dyiomegaij} along a $j^{\text{th}}$ characteristic, we will need the initial values of $\presup{\tau}{(i)}_j$ and $\presup{\tau}{(i)}^j$ at the boundary of $\mathcal R_i$ (recall that $y_i$ serves as a parameter for $\mathcal C_j$), i.e., along $\mathcal C_i(-1)$ and $\mathcal C_i(1)$. But, by \eqref{eqs:dbccoords_partial}, \eqref{eqs:dccoords}
\begin{equation*}
  \frac{\partial}{\partial y_j} = \rho_j\frac{\partial}{\partial x} + \frac{\rho_j}{\lambda_i-\lambda_j}\frac{\partial}{\partial s_j} .
\end{equation*}
As a consequence, the bounds on $\frac{\partial}{\partial x}\rho_j$ and $\frac{\partial}{\partial x}w^j$ just established, as well as \eqref{eq:dsirhoi}, \eqref{eq:dsiwi}, \eqref{eq:rhoi_range_outside_Ri}, \eqref{eq:wi_bdd_in_Rchecki} and \eqref{eq:sigma_pos} imply by continuity that $\presup{\tau}{(i)}_j$ and $\presup{\omega}{(i)}^j$ are bounded on $\mathcal C_i(-1)\bigcup \mathcal C_i(1)$, for all $j\in\{1,\ldots,N\}\setminus\{i\}$. For later reference, let us write 
\begin{equation} \label{eqs:Gbari0Hbari0} 
  \begin{split}
    \presup{\overline G}{(i)}_0 &= \max_{j\neq i} \sup_{(x',t')\in \mathcal C_i(-1)\bigcup \mathcal C_i(1)} \bigabs{\presup{\tau}{(i)}_j(x',t')} , \\
    \presup{\overline H}{(i)}_0 &= \max_{j\neq i} \sup_{(x',t')\in \mathcal C_i(-1)\bigcup \mathcal C_i(1)} \bigabs{\presup{\omega}{(i)}^j(x',t')} ,
  \end{split}
\end{equation}
which are thus both finite. Similarly, define 
\begin{equation} \label{eqs:GbariHbari} 
  \begin{split}
    \presup{\overline G}{(i)} &= \max_{j\neq i} \sup_{(x',t')\in \mathcal R_i} \bigabs{\presup{\tau}{(i)}_j(x',t')} , \\
    \presup{\overline H}{(i)} &= \max_{j\neq i} \sup_{(x',t')\in \mathcal R_i} \bigabs{\presup{\omega}{(i)}^j(x',t')} .
  \end{split}
\end{equation}
Our aim is to show that both these quantities are finite. Fixing $j\in\{1,\ldots,N\}\setminus\{i\}$, we have from \eqref{eq:dyitauij} and \eqref{eq:dyiomegaij}, respectively, that, along any $j^{\text{th}}$ characteristic inside $\mathcal R_i$, 
\begin{equation} \label{eqs:tauijomegaij_system}
  \begin{split}
    \frac{\partial}{\partial y_i}\presup{\tau}{(i)}_j &= \sum_{\stack p{p\neq i}} \presup{A}{(i)}_j^p\presup{\tau}{(i)}_p + \sum_{\stack p{p\neq i}} \presup{B}{(i)}_p^j\presup{\omega}{(i)}^p + \presup{M}{(i)}_j , \\
    \frac{\partial}{\partial y_i}\presup{\omega}{(i)}^j &= \sum_{\stack p{p\neq i}} \presup{E}{(i)}_j^p\presup{\tau}{(i)}_p + \sum_{\stack p{p\neq i}} \presup{F}{(i)}_p^j\presup{\omega}{(i)}^p + \presup{N}{(i)}_j ,
  \end{split}
\end{equation}
where the coefficients are, by \eqref{eqs:tV_J_V_U_bootstrap_results}, \eqref{eqs:other_bootstrap_results} and \eqref{eq:rhoi_range_outside_Ri}, at least $\mathcal O(\vartheta^2)$, except for $\presup{B}{(i)}_p^j$ and $\presup{F}{(i)}_p^j$ ($p\in\{1,\ldots,N\}\setminus\{i\}$), which are $\mathcal O(1)$ and $\mathcal O(\vartheta)$, respectively. Set 
\begin{equation} \label{eqs:Abari_etc}
  \begin{split}
    \presup{\overline A}{(i)} &= 2\max_{\stack{j,p}{j\neq i,p\neq i}} \sup_{(x',t')\in \mathcal R_i} \bigabs{ \presup{A}{(i)}_j^p(x',t') } , \\
    \presup{\overline B}{(i)} &= 2\max_{\stack{j,p}{j\neq i,p\neq i}} \sup_{(x',t')\in \mathcal R_i} \bigabs{ \presup{B}{(i)}_p^j(x',t') } , \\
    \presup{\overline M}{(i)} &= 2\max_{\stack{j}{j\neq i}} \sup_{(x',t')\in \mathcal R_i} \bigabs{ \presup{M}{(i)}_j(x',t') } , \\
    \presup{\overline E}{(i)} &= 2\max_{\stack{j,p}{j\neq i,p\neq i}} \sup_{(x',t')\in \mathcal R_i} \bigabs{ \presup{E}{(i)}_j^p(x',t') } , \\
    \presup{\overline F}{(i)} &= 2\max_{\stack{j,p}{j\neq i,p\neq i}} \sup_{(x',t')\in \mathcal R_i} \bigabs{ \presup{F}{(i)}_p^j(x',t') } , \\
    \presup{\overline N}{(i)} &= 2\max_{\stack{j}{j\neq i}} \sup_{(x',t')\in \mathcal R_i} \bigabs{ \presup{N}{(i)}^j(x',t') } ,
  \end{split}
\end{equation}
which are all $\mathcal O(\vartheta^2)$, except for $\presup{\overline B}{(i)}$ and $\presup{\overline F}{(i)}$, which are $\mathcal O(1)$ and $\mathcal O(\vartheta)$, respectively. Then integrating along $\mathcal C_j$ yields 
\begin{equation*}
  \begin{split}
    \presup{\tau}{(i)}_j &\leq \presup{\overline G}{(i)}_0 + \presup{\overline A}{(i)}\presup{\overline G}{(i)} + \presup{\overline B}{(i)}\presup{\overline H}{(i)} + \presup{\overline M}{(i)} , \\
    \presup{\omega}{(i)}^j &\leq \presup{\overline H}{(i)}_0 + \presup{\overline E}{(i)}\presup{\overline G}{(i)} + \presup{\overline F}{(i)}\presup{\overline H}{(i)} + \presup{\overline N}{(i)} , 
  \end{split}
\end{equation*}
where we have used that, along $\mathcal C_j\bigcap \mathcal R_i$, we have $y_i\in[-1,1]$. Taking supremums over $\mathcal R_i$, we obtain 
\begin{equation*}
  \begin{split}
    \presup{\overline G}{(i)} &\leq \presup{\overline G}{(i)}_0 + \presup{\overline A}{(i)}\presup{\overline G}{(i)} + \presup{\overline B}{(i)}\presup{\overline H}{(i)} + \presup{\overline M}{(i)} , \\
     \presup{\overline H}{(i)} &\leq \presup{\overline H}{(i)}_0 + \presup{\overline E}{(i)}\presup{\overline G}{(i)} + \presup{\overline F}{(i)}\presup{\overline H}{(i)} + \presup{\overline N}{(i)} .
  \end{split}
\end{equation*}
But since $\presup{\overline A}{(i)},\presup{\overline M}{(i)},\presup{\overline E}{(i)},\presup{\overline N}{(i)}=\mathcal O(\vartheta^2)$, while $\presup{\overline B}{(i)}=\mathcal O(1)$ and $\presup{\overline F}{(i)}=\mathcal O(\vartheta)$, and since $\presup{\overline G}{(i)}_0$, $\presup{\overline H}{(i)}_0$ are bounded, we can re--write this as
\begin{equation*}
  \begin{split}
    \presup{\overline G}{(i)} &= \mathcal O\bigl( \presup{\overline G}{(i)}_0 + \vartheta^2\presup{\overline G}{(i)} + \presup{\overline H}{(i)} \bigr) , \\
    \presup{\overline H}{(i)} &= \mathcal O\bigl( \presup{\overline H}{(i)}_0 + \vartheta^2\presup{\overline G}{(i)} + \vartheta\presup{\overline H}{(i)} \bigr) ,
  \end{split}
\end{equation*}
which, for sufficiently small $\vartheta$, reads 
\begin{equation*}
  \begin{split}
  \presup{\overline G}{(i)} &= \mathcal O\bigl( \presup{\overline G}{(i)}_0 + \presup{\overline H}{(i)} \bigr) , \\
  \presup{\overline H}{(i)} &= \mathcal O\bigl( \presup{\overline H}{(i)}_0 + \vartheta^2\presup{\overline G}{(i)} \bigr) .
  \end{split}
\end{equation*}
Inserting the first line into the second yields 
\begin{equation*}
  \begin{split}
  \presup{\overline G}{(i)} &= \mathcal O\bigl( \presup{\overline G}{(i)}_0 + \presup{\overline H}{(i)} \bigr) , \\
  \presup{\overline H}{(i)} &= \mathcal O\bigl( \presup{\overline H}{(i)}_0 + \vartheta^2\presup{\overline G}{(i)}_0 + \vartheta^2\presup{\overline H}{(i)} \bigr) = \mathcal O\bigl( \presup{\overline G}{(i)}_0 + \presup{\overline H}{(i)}_0 \bigr) ,
  \end{split}
\end{equation*}
for sufficiently small $\vartheta$. As a result, we obtain 
\begin{equation} \label{eqs:GbarHbar_bdd}
  \presup{\overline G}{(i)} = \mathcal O\bigl( \presup{\overline G}{(i)}_0 + \presup{\overline H}{(i)}_0 \bigr) \qquad\text{ and }\qquad \presup{\overline H}{(i)} = \mathcal O\bigl( \presup{\overline G}{(i)}_0 + \presup{\overline H}{(i)}_0 \bigr) ,
\end{equation}
so that $\presup{\tau}{(i)}_j$ and $\presup{\omega}{(i)}^j$ are bounded, as stated in Theorem~\ref{thm:main_thm}~\eqref{thm:main_thm_secondorder}.

At last, we are ready to show the boundedness of $\mu_i$, $\nu^i$ inside $\mathcal R_i$. In order to do this, we wish to integrate \eqref{eq:dsimui} and \eqref{eq:dsinui} along an $i^{\text{th}}$ characteristic inside $\mathcal R_i$. We can re--write the evolution equations as 
\begin{equation} \label{eqs:muinui_system}
  \begin{split}
    \frac{\partial}{\partial s_i}\mu_i &= A_{\mu_i} \mu_i + B_{\mu_i} \nu^i + C_{\mu_i} , \\
    \frac{\partial}{\partial s_i}\nu^i &= A_{\nu^i} \mu_i + B_{\nu^i} \nu^i + C_{\nu^i} , 
  \end{split}
\end{equation}
with bounded (see \eqref{eqs:tV_J_V_U_bootstrap_results}, \eqref{eqs:other_bootstrap_results}, \eqref{eq:rhoi_range_outside_Ri}, \eqref{eqs:GbarHbar_bdd}) coefficients.
Setting 
\begin{equation} \label{eqs:MitNit}
    M_i(t) = \sup_{\stack{(z_i',s_i')}{\stack{z_i'\in I_0}{0\leq s_i'\leq t}}} \bigabs{\mu_i(z_i',s_i')} \qquad\text{ and }\qquad N^i(t) = \sup_{\stack{(z_i',s_i')}{\stack{z_i'\in I_0}{0\leq s_i'\leq t}}} \bigabs{\nu^i(z_i',s_i')} ,
\end{equation}
which are non--decreasing in $t$, as well as
\begin{equation} \label{eqs:Abarmui_etc}
  \begin{split}
    \overline A_{\mu_i} &= \sup_{(x',t')\in \mathcal R_i} \bigabs{A_{\mu_i}(x',t')} , \\
    \overline A_{\nu^i} &= \sup_{(x',t')\in \mathcal R_i} \bigabs{A_{\nu^i}(x',t')} , \\
    \overline B_{\mu_i} &= \sup_{(x',t')\in \mathcal R_i} \bigabs{B_{\mu_i}(x',t')} , \\
    \overline B_{\nu^i} &= \sup_{(x',t')\in \mathcal R_i} \bigabs{B_{\nu^i}(x',t')} , \\
    \overline C_{\mu_i} &= \sup_{(x',t')\in \mathcal R_i} \bigabs{C_{\mu_i}(x',t')} , \\
    \overline C_{\nu^i} &= \sup_{(x',t')\in \mathcal R_i} \bigabs{C_{\nu^i}(x',t')} ,
  \end{split}
\end{equation}
we obtain, for almost every $t$, 
\begin{equation*}
  \begin{split}
    \frac{\partial}{\partial s_i}\bigabs{\mu_i(z_i,t)} &\leq \Bigabs{\frac{\partial}{\partial s_i}\mu_i(z_i,t)} \leq \overline A_{\mu_i} M_i(t) + \overline B_{\mu_i} N^i(t) + \overline C_{\mu_i} , \\
    \frac{\partial}{\partial s_i}\bigabs{\nu^i(z_i,t)} &\leq \Bigabs{\frac{\partial}{\partial s_i}\nu^i(z_i,t)} \leq \overline A_{\nu^i} M_i(t) + \overline B_{\nu^i} N^i(t) + \overline C_{\nu^i} .
  \end{split}
\end{equation*}
Integrating along $\mathcal C_i$ and subsequently varying $z_i\in I_0$, we arrive at 
\begin{equation} \label{eqs:MitNit_system}
  \begin{split}
    M_i(t) &\leq \overline A_{\mu_i} \int_0^t M_i(s) ds + \overline B_{\mu_i} \int_0^t N^i(s) ds + \overline C_{\mu_i}t , \\
    N^i(t) &\leq \overline A_{\nu^i} \int_0^t M_i(s) ds + \overline B_{\nu^i} \int_0^t N^i(s) ds + \overline C_{\nu^i}t + N^i(0) ,
  \end{split}
\end{equation}
where we have used that $\frac{\partial}{\partial z_i}\rho_i(z_i,0) = \frac{\partial}{\partial x}\rho_i(z_i,0) = 0$, since $\rho_i(z_i,0)=1$. Note also that $N^i(0)\leq V'(0)$ (cf.~\eqref{eq:Vprimet}). To show that $M_i$, $N^i$ are bounded on $[0,t_\star)$, we make use of the second statement in Lemma~\ref{lem:sG}, for $t<t_\star\leq\overline T_0$ is finite (cf.~\eqref{eq:tstar_range},\eqref{eq:T0}).

This, therefore, establishes the rest of Theorem~\ref{thm:main_thm}~\eqref{thm:main_thm_secondorder}.

\section{Application to electromagnetic plane waves in non--linear crystals} \label{sec:application}

This section is devoted to the application of the theory established so far to the case of plane waves in a non--linear crystal. We assume the material to be birefringent, non--dispersive, non--magnetic and exhibiting electric third order non--linearities in the energy density. We begin by recalling quickly the relation between the Lagrangian and the energy density in the electromagnetic theory, from which one can recover the Maxwell equations. After restricting to the aforementioned type of material and plane waves, we proceed to exposing under what assumptions our theory is applicable. Subsequently, we show the existence of Riemann invariants in a special case, showing thereafter their non--existence for generic third order non--linearities. As an experimental setup would require to deal with a material boundary, we end the section with a quick description of that situation. We then solve in full (albeit only implicitly) the special case above, and comment briefly on the problem under generic conditions.

\subsection{The relevant equations} \label{subsec:equations}

Inside matter, the Maxwell equations are 
\begin{equation} \label{eqs:Maxwell_full}
  \begin{aligned}
    \div B &= 0 , & \div D &= \rho , \\
    \frac{\partial}{\partial t} B + \curl E &= 0 , & \frac{\partial}{\partial t}D - \curl H &= -J ,
  \end{aligned}
\end{equation}
where $E$, $B$, $D$, $H$, $\rho$ and $J$ denote the electric and the magnetic field, the electric and the magnetic displacement, the density of free charges and the density of free currents, respectively. These equations can be obtained from a variational formalism as follows (see, e.g., \cite{christodoulou00}). Let $A=(A_0,A_1,A_2,A_3)$ denote the electromagnetic vector potential such that, for $i\in\{1,2,3\}$, $E^i=\frac{\partial}{\partial x^i}A_0-\frac{\partial}{\partial t}A_i$ and $B^i=\sum_{k,l=1}^3\frac{\partial}{\partial x^k}A_l\epsilon^{ikl}$, where $\epsilon^{ikl}$ denotes the fully antisymmetric symbol (here, and in the following, indices are raised and lowered using the Minkowski metric $\diag(-1,1,1,1)$). The field tensor $F=dA$, a two--form, is given by $F=\sum_{\mu,\nu=0}^3\frac12F_{\mu\nu}dx^\mu\wedge dx^\nu$, with $F_{\mu\nu}=\frac{\partial}{\partial x^\mu}A_\nu - \frac{\partial}{\partial x^\nu}A_\mu$, $\mu,\nu\in\{0,1,2,3\}$, $x^0=t$. We have 
\begin{equation*}
  (F_{\mu\nu}) = 
  \begin{pmatrix}
      0  & -E_1 & -E_2 & -E_3 \\
     E_1 &   0  &  B_3 & -B_2 \\
     E_2 & -B_3 &   0  &  B_1 \\
     E_3 &  B_2 & -B_1 &   0
  \end{pmatrix}
  ,
\end{equation*}
and we obtain the two equations on the left in \eqref{eqs:Maxwell_full} from $dF=d^2A=0$. Given a (dual) Lagrangian density $\postsup{\mathcal L}{\star}=\postsup{\mathcal L}{\star}(x,\alpha,\Phi)$, where $\alpha$ and $\Phi$ denote one-- and two--forms, respectively, we define the dual canonical momentum $\postsup{\Psi}{\star}$ and the dual electromagnetic displacement $\postsup{G}{\star}$, both bivectors, by $\postsup{\Psi}{\star} = \sum_{\mu,\nu=0}^3 \frac12\postsup{\Psi}{\star}^{\mu\nu}\frac{\partial}{\partial x^\mu}\wedge\frac{\partial}{\partial x^\nu}$ and $\postsup{G}{\star} = \sum_{\mu,\nu=0}^3 \frac12\postsup{G}{\star}^{\mu\nu}\frac{\partial}{\partial x^\mu}\wedge\frac{\partial}{\partial x^\nu}$, respectively, where 
\begin{equation*}
  \postsup{\Psi}{\star}^{\mu\nu} = \frac{\partial}{\partial \Phi_{\mu\nu}}\postsup{\mathcal L}{\star} , \quad \postsup{G}{\star}^{\mu\nu}(x) = \postsup{\Psi}{\star}^{\mu\nu}\bigl(x,A(x),F(x)\bigr) , \qquad \mu,\nu\in\{0,1,2,3\},\ \mu\neq\nu .
\end{equation*}
Setting 
\begin{equation*}
  D^i = -\frac{\partial}{\partial E_i}\bigl(\postsup{\mathcal L}{\star}(\argplaceholder,A,F)\bigr) \quad\text{ and }\quad H^i = \frac{\partial}{\partial B_i}\bigl(\postsup{\mathcal L}{\star}(\argplaceholder,A,F)\bigr) ,  \qquad i\in\{1,2,3\} ,
\end{equation*}
we get 
\begin{equation*}
  (\postsup{G}{\star}^{\mu\nu}) = 
  \begin{pmatrix}
      0  &  D^1 &  D^2 &  D^3 \\
    -D^1 &   0  &  H^3 & -H^2 \\
    -D^2 & -H^3 &   0  &  H^1 \\
    -D^3 &  H^2 & -H^1 &   0
  \end{pmatrix}
  .
\end{equation*}
The canonical momentum $\Psi$ and the electromagnetic displacement $G$ are then given by $\Psi=\sum_{\mu,\nu=0}^3\frac12\Psi_{\mu\nu}dx^\mu\wedge dx^\nu$ and $G=\sum_{\mu,\nu=0}^3\frac12 G_{\mu\nu}dx^\mu\wedge dx^\nu$, respectively, with $\Psi_{\mu\nu}=\sum_{\kappa,\lambda=0}^3 \frac12\postsup{\Psi}{\star}^{\kappa\lambda}\epsilon_{\kappa\lambda\mu\nu}$ and $G_{\mu\nu}=\sum_{\kappa,\lambda=0}^3 \frac12\postsup{G}{\star}^{\kappa\lambda}\epsilon_{\kappa\lambda\mu\nu}$, where $\epsilon_{\kappa\lambda\mu\nu}$ is the totally antisymmetric symbol. We obtain 
\begin{equation*}
  (G_{\mu\nu}) = 
  \begin{pmatrix}
      0  &  H^1 &  H^2 &  H^3 \\
    -H^1 &   0  &  D^3 & -D^2 \\
    -H^2 & -D^3 &   0  &  D^1 \\
    -H^3 &  D^2 & -D^1 &   0
  \end{pmatrix}
  .
\end{equation*}
Setting 
\begin{align*}
  \rho(x) &= \Bigl(-\frac{\partial}{\partial \alpha_0}\postsup{\mathcal L}{\star}\Bigr)\bigl(x,A(x),F(x)\bigr) , \\
   J^i(x) &= \Bigl(-\frac{\partial}{\partial \alpha_i}\postsup{\mathcal L}{\star}\Bigr)\bigl(x,A(x),F(x)\bigr) , \qquad i\in\{1,2,3\} ,
\end{align*}
and defining the dual four--current $\postsup{I}{\star}$ by $\postsup{I}{\star}^0=\rho$ and $\postsup{I}{\star}^i=J^i$, $i\in\{1,2,3\}$, the current three--form $I$ is given by $I=\sum_{\lambda,\mu,\nu=0}^3\frac16 I_{\lambda\mu\nu}dx^\lambda\wedge dx^\mu\wedge dx^\nu$, where $I_{\lambda\mu\nu}=\sum_{\kappa=0}^3\postsup{I}{\star}^\kappa\epsilon_{\kappa\lambda\mu\nu}$. The equations on the right in \eqref{eqs:Maxwell_full} are then obtained from the Euler--Lagrange equations 
\begin{equation*}
  dG = I \qquad\Leftrightarrow\qquad \postsup{(dG)}{\star} = \postsup{I}{\star} ,
\end{equation*}
which arise from varying $A$ in the electromagnetic action $\mathcal S[A;\mathcal D] = \int_{\mathcal D}\mathcal L(\argplaceholder,A,F)$ of $A$ in the space--time domain $\mathcal D$ (cf.~\cite[eq.~(6.1.32)]{christodoulou00}). The current $I$ being the canonical force $\iota$ along a solution, $I=\iota(\argplaceholder,A,F)$, we write 
\begin{equation*}
  \iota = \sum_{\lambda,\mu,\nu=0}^3\frac16\iota_{\lambda\mu\nu}dx^\lambda\wedge dx^\mu\wedge dx^\nu ,\ \iota_{\lambda\mu\nu} = \sum_{\kappa=0}^3\postsup{\iota}{\star}^\kappa\epsilon_{\kappa\lambda\mu\nu} ,\ \postsup{\iota}{\star}^\kappa = -\frac{\partial}{\partial \alpha_\kappa}\postsup{\mathcal L}{\star} .
\end{equation*}
From the requirement that the action $S[A;\mathcal D]$ be invariant under gauge transformations $A\mapsto A+df$ compactly supported in the space--time domain $\mathcal D$, one shows (cf.~\cite[p.265]{christodoulou00}) that $\iota$ has to be a \emph{conserved null current}, in the sense that $dJ=0$ and that $\frac{\partial}{\partial \alpha_\nu}\postsup{\iota}{\star}^\mu+\frac{\partial}{\partial \alpha_\mu}\postsup{\iota}{\star}^\nu=0$ (see \cite[\S6.3\&Ch.3]{christodoulou00}). Theorem~6.1 in \cite{christodoulou00} then implies that $\iota$ is, in general, of the form 
\begin{equation*}
  \iota = \iota(x,\Phi) = \stackrel{(0)}{C}(x) + \Phi(x)\wedge\stackrel{(1)}{C}(x) ,
\end{equation*}
where $\stackrel{(i)}{C}$, $i\in\{0,1\}$, are closed exterior differential forms on space--time of degree $3-2i$. One concludes that the Lagrangian is of the form:
\begin{equation*}
  \mathcal L(x,\alpha,\Phi) = \stackrel{(0)}{\mathcal L}(x,\Phi) + \alpha(x)\wedge\iota(x,\Phi) .
\end{equation*}
The canonical electromagnetic stress tensor $\postsup{T}{\star}$ is defined through its components by (cf.~\cite[\S6.3]{christodoulou00}) 
\begin{equation*}
  \ul{{\postsup{T}{\star}}\mu\nu} = \sum_{\lambda=0}^3\postsup{G}{\star}^{\mu\lambda}F_{\nu\lambda} - \stackrel{{\mkern-10.0mu(0)}}{\postsup{\mathcal L}{\star}}(\argplaceholder,F)\ul{\delta\mu\nu} , \qquad \mu,\nu\in\{0,1,2,3\} .
\end{equation*}
Corresponding to a time translation field $X$, we have the energy--momentum density vectorfield 
\begin{equation*}
  \postsup{P}{\star}^\mu = -\sum_{\nu=0}^3\ul{{\postsup{T}{\star}}\mu\nu}X^\nu .
\end{equation*}
Taking $X=\frac{\partial}{\partial x^0}$, we get 
\begin{equation*}
  \postsup{P}{\star}^\mu = -\ul{{\postsup{T}{\star}}\mu0} .
\end{equation*}
Then $\postsup{P}{\star}^0=\mathcal E$ is the field energy density and $\postsup{P}{\star}^i=\mathcal P^i$ is the field energy flux (Poynting flux). From the above, we find in terms of $E$, $B$, $D$, $H$ and slightly abusing notation that, along a solution, 
\begin{equation*}
  \mathcal E(\argplaceholder,D,B) = E\cdot D + \stackrel{\mkern-10mu(0)}{\postsup{\mathcal L}{\star}}(\argplaceholder,E,B) \quad\text{ and }\quad \mathcal P(\argplaceholder,E,H) = E\times H ,
\end{equation*}
where, now, 
\begin{equation*}
  E = \frac{\partial}{\partial D}\mathcal E , \qquad H = \frac{\partial}{\partial B}\mathcal E .
\end{equation*}
If we assume $\stackrel{{\mkern-10.0mu(0)}}{\postsup{\mathcal L}{\star}}$ not to depend explicitly on $x^0=t$, the Maxwell equations imply 
\begin{equation*}
  \sum_{\mu=0}^3 \nabla_\mu{\postsup{P}{\star}}^\mu = -\sum_{\mu=0}^3F_{\mu0}\postsup{I}{\star}^\mu ,
\end{equation*}
or,
\begin{equation*}
  \frac{\partial}{\partial t}\mathcal E + \nabla\cdot\mathcal P = -E\cdot J ,
\end{equation*}
which expresses the differential energy law.

We now focus our attention on electromagnetic plane waves inside a non--dispersive homogeneous dielectric material with no magnetic properties. Dealing with a perfect insulator, there is no free current, i.e., $J\equiv0$. For simplicity, we shall also assume the absence of extraneous charges, i.e., $\rho\equiv0$, so that we have $I\equiv0$. Note also that we assume to be well below the \emph{electrical breakdown} of the material, i.e., we assume the fields to be weak enough so that the material actually behaves as an insulator. This is in accordance with our mathematical treatment, where we also showed that the smallness assumption on the initial data is preserved in evolution, even though infinite field gradients develop in finite time. %
The Lagrangian density $\postsup{\mathcal L}{\star}$ for a homogeneous material with only time--independent properties (hence our requirement of no dispersion) is invariant under both time and space translations. We can thus write, for some function $L$,
\begin{equation*}
  \postsup{\mathcal L}{\star}(t,x,A,E,B) = L(E,B) ,
\end{equation*}
and $D$, $H$ are defined by 
\begin{equation*}
  D = -\frac{\partial}{\partial E}L , \quad H = \frac{\partial}{\partial B}L ,
\end{equation*}
respectively. If the material has no magnetic properties, there is a function $\mathfrak F$ such that 
\begin{equation*}
  L = \mathfrak F(E) + \frac12\abs{B}^2 ,
\end{equation*}
whence $H=B$ and $D=-\frac{\partial\mathfrak F}{\partial E}$, i.e., $D$ depends on $E$ alone. Since the Maxwell equations \eqref{eqs:Maxwell_full} are left unchanged by addition of a constant to $L$, we may assume that $\mathfrak F(0)=0$. In addition, we shall assume that $-d\mathfrak F$ is invertible around zero. In terms of the energy density (viewed as a function of $(D,B)$) 
\begin{equation*}
  \mathcal E(D,B) = E\cdot D + L(E,B) ,
\end{equation*}
we obtain, for $i\in\{1,2,3\}$, 
\begin{align*}
  \frac{\partial}{\partial D^i}\mathcal E &= \sum_j\Bigl(\frac{\partial}{\partial D^i}E^j\Bigr)D^j + E^i + \sum_j \Bigl(\underbrace{\frac{\partial}{\partial E^j}L}_{\phantom{-D^j}=-D^j}\Bigr)\frac{\partial}{\partial D^i}E^j = E^i , \\
\intertext{and}
  \frac{\partial}{\partial B^i}\mathcal E &= \sum_j\Bigl(\frac{\partial}{\partial B^i}E^j\Bigr)D^j + \sum_j \Bigl(\underbrace{\frac{\partial}{\partial E^j}L}_{\phantom{-D^j}=-D^j}\Bigr)\frac{\partial}{\partial B^i}E^j + \underbrace{\frac{\partial}{\partial B^i}L}_{\phantom{H^i}=H^i} = H^i ,
\end{align*}
as expected. Since $D$ depends only on $E$, we can write 
\begin{equation*}
  \mathcal E(D,B) = \mathfrak G(D) + \frac12\abs{B}^2 ,
\end{equation*}
where 
\begin{equation*}
  \mathfrak G(D) = E\cdot D + \mathfrak F(E) .
\end{equation*}
Note that $\mathfrak F(0)=0$ entails $\mathfrak G(0)=0$.

The aim of what follows is to study plane wave solutions of the Maxwell equations for $\mathfrak G$ given by its third order Taylor expansion around zero. For emphasis, we shall henceforth write $x=x^1$, $y=x^2$, $z=x^3$, and, correspondingly,
\begin{equation*}
  \begin{split}
  D_x = D^1,\ D_y = D^2,\ D_z = D^3,\ E_x = E^1,\ E_y = E^2,\ E_z = E^3, \qquad\qquad \\
  \qquad B_x = B^1,\ B_y = B^2,\ B_z = B^3 .
  \end{split}
\end{equation*}
An electromagnetic plane wave travelling in the $x$--direction is characterised by the fact that the field variables depend only on $x$ and $t$. From the constraint equations $\div B=0$ and $\div D=0$, we obtain that $D_x$ and $B_x$ are constant in space. Their constancy in time is a consequence of the two other equations in \eqref{eqs:Maxwell_full}. We may therefore assume that $D_x$ and $B_x$ vanish throughout. The remaining equations governing the evolution of $(D_y,D_z,B_y,B_z)$ reduce to 
\begin{equation} \label{eqs:Maxwell}
  \left\{ \begin{aligned}
    \frac{\partial}{\partial t} D_y + \frac{\partial}{\partial x} B_z &=0 , & \frac{\partial}{\partial t} D_z - \frac{\partial}{\partial x} B_y &=0 , \\
    \frac{\partial}{\partial t} B_y - \frac{\partial}{\partial x} E_z &=0 , & \frac{\partial}{\partial t} B_z + \frac{\partial}{\partial x} E_y &=0 ,
  \end{aligned} \right.
\end{equation}
where $E=\frac{\partial\mathcal E}{\partial D}=\frac{\partial\mathfrak G}{\partial D}$. In the type of material under consideration, modulo the correct alignment of the rectangular coordinate axes, we can assume that, up to terms of fourth order, 
\begin{equation*}
  \mathfrak G(D) = \frac12 K_1 D_y^2 + \frac12 K_2 D_z^2 + C_{111} D_y^3 + C_{112} D_y^2D_z + C_{122} D_yD_z^2 + C_{222} D_z^3 ,
\end{equation*}
where $K_1, K_2, C_{111}, C_{112}, C_{122}, C_{222} \in\rz$ are constants (note that there is no need to consider terms linear in the components of $D$, as they would not alter the Maxwell equations \eqref{eqs:Maxwell}). In order to obtain a strictly hyperbolic system, we restrict ourselves to birefringent materials, i.e., media where the refractive index depends on the polarisation of the light. This situation occurs in certain crystals and translates into the condition 
\begin{equation} \label{eq:K1K2}
  0 < K_2 < K_1 < 1 .
\end{equation}
The energy density is 
\begin{equation} \label{eq:energy_third_order}
  \mathcal E = \frac12 B_y^2 + \frac12 B_z^2 + \frac12 K_1 D_y^2 + \frac12 K_2 D_z^2 + C_{111} D_y^3 + C_{112} D_y^2D_z + C_{122} D_yD_z^2 + C_{222} D_z^3 .
\end{equation}
We then obtain from $E=\frac{\partial\mathcal E}{\partial D}=\frac{\partial\mathfrak G}{\partial D}$, 
\begin{equation} \label{eqs:E_in_terms_of_D}
  \left\{ \begin{aligned}
    E_y &= K_1 D_y + 3 C_{111} D_y^2 + 2 C_{112} D_yD_z +   C_{122} D_z^2 , \\
    E_z &= K_2 D_z +   C_{112} D_y^2 + 2 C_{122} D_yD_z + 3 C_{222} D_z^2 .
  \end{aligned} \right.
\end{equation}
We now show that \eqref{eqs:Maxwell}, taking into account \eqref{eqs:E_in_terms_of_D}, falls in the framework of our theory. More precisely, we prove 
\begin{lemma} \label{lem:theory_applies}
  Let $u=\bigl(D_y,D_z,B_y,B_z\bigr)$. Then there is a matrix function $a=a(u)$ such that \eqref{eqs:Maxwell} for $E$ given by \eqref{eqs:E_in_terms_of_D} is equivalent to 
  \begin{equation} \label{eq:theory_applies_main_pde}
    \frac{\partial}{\partial t}u + a(u)\frac{\partial}{\partial x}u = 0 .
  \end{equation}
  Moreover, under the assumption \eqref{eq:K1K2}, there is a $\delta>0$ such that the system \eqref{eq:theory_applies_main_pde} is uniformly strictly hyperbolic on $\nball[4]_{\delta}(0)$. Finally, \eqref{eq:theory_applies_main_pde} is genuinely non--linear, whenever $C_{111}$ and $C_{222}$ are non--zero.
\end{lemma}

\subsection{Proof of Lemma~\ref{lem:theory_applies}} \label{subsec:theory_applies}

Let $u = \bigl( D_y, D_z, B_y, B_z \bigr)^T$. Then the first row of \eqref{eqs:Maxwell} is 
\begin{equation*}
  \frac{\partial}{\partial t} u^1 + \frac{\partial}{\partial x} u^4 = 0 , \qquad \frac{\partial}{\partial t} u^2 - \frac{\partial}{\partial x} u^3 = 0 .
\end{equation*}
Using \eqref{eqs:E_in_terms_of_D}, the second row of \eqref{eqs:Maxwell} translates to
\begin{align*}
  \frac{\partial}{\partial t} u^3 - K_2\frac{\partial}{\partial x} u^2 - 2C_{112}u^1\frac{\partial}{\partial x} u^1 - 2C_{122}u^1\frac{\partial}{\partial x} u^2 - 2C_{122}u^2\frac{\partial}{\partial x} u^1 - 6C_{222}u^2\frac{\partial}{\partial x} u^2 &= 0 , \\
\intertext{and}
  \frac{\partial}{\partial t} u^4 + K_1\frac{\partial}{\partial x} u^1 + 6C_{111}u^1\frac{\partial}{\partial x} u^1 + 2C_{112}u^1\frac{\partial}{\partial x} u^2 + 2C_{112}u^2\frac{\partial}{\partial x} u^1 + 2C_{122}u^2\frac{\partial}{\partial x} u^2 &= 0 .
\end{align*}
Thus, setting 
\begin{subequations} \label{eqs:d1d2c}
\begin{align}
\label{eq:d1}
  d_1 = d_1(u) &= K_1 + 6C_{111}u^1 + 2C_{112}u^2 , \\
\label{eq:d2}
  d_2 = d_2(u) &= K_2 + 2C_{122}u^1 + 6C_{222}u^2 , \\
\intertext{and}
\label{eq:c}
  \phantom{d_2}\makebox[0pt][r]{c} = \phantom{d_2}\makebox[0pt][r]{c}(u) &= 2C_{112}u^1 + 2C_{122}u^2 ,
\end{align}
\end{subequations}
we see that the Maxwell equations \eqref{eqs:Maxwell} are equivalent to 
\begin{equation*}
  \frac{\partial}{\partial t} u + a \frac{\partial}{\partial x} u = 0 ,
\end{equation*}
with 
\begin{equation} \label{eq:a}
  a = a(u) = \begin{pmatrix}
         0   &    0    &  0 &  1 \\
         0   &    0    & -1 &  0 \\
      -c(u)  & -d_2(u) &  0 &  0 \\
      d_1(u) &   c(u)  &  0 &  0 
    \end{pmatrix} .
\end{equation}
We have thus proved the first statement in Lemma~\ref{lem:theory_applies}.

Now set 
\begin{equation} \label{eqs:mr}
  m = m(u) = \frac{d_1+d_2}2 \qquad\text{ and }\qquad r = r(u) = \frac{d_1-d_2}2 ,
\end{equation}
so that $m+r=d_1$ and $m-r=d_2$. Notice that 
\begin{equation} \label{eqs:d1d2c_at_zero}
  d_1(0) = K_1 , \qquad d_2(0) = K_2 \qquad\text{ and }\qquad c(0)=0 ,
\end{equation}
whence, setting $m_0=m(0)$, $r_0=r(0)$, 
\begin{equation} \label{eqs:m0r0}
  m_0 = \frac{K_1+K_2}2 \qquad\text{ and }\qquad r_0 = \frac{K_1-K_2}2 .
\end{equation}
Pick then any $h>0$ such that 
\begin{equation} \label{eq:h}
  h < \frac{\min\Bigl\{r_0\,,\,\overbrace{m_0-r_0}^{K_2}\,,\,1-\bigl(\overbrace{m_0+r_0}^{K_1}\bigr)\Bigr\}}{2m_0} ,
\end{equation}
so that 
\begin{align*}
  (1-h) &> \max\Biggl\{ 1 - \frac{r_0}{2m_0} \;,\; \frac{K_1}{2m_0} \Biggr\} , \\
  (1+h) &< \min\Biggl\{ 1 + \frac{r_0}{2m_0} \;,\; \frac{1+K_2}{2m_0} \Bigl\} .
\end{align*}
Then it is easy to check that, by \eqref{eq:K1K2}, 
\begin{equation*}
  0 < (1-h)K_2 < (1+h)K_2 < (1-h)K_1 < (1+h)K_1 < 1 ,
\end{equation*}
for 
\begin{align*}
  (1-h)K_2 &> \frac{K_1K_2}{K_1+K_2} > 0 , \\
  (1+h)K_1 &< \frac{K_1(1+K_2)}{K_1+K_2} = 1 - \frac{K_2(1-K_1)}{K_1+K_2} < 1 , \\
  (1-h)K_1-(1+h)K_2 &> K_1-K_2 - \frac{K_1r_0+K_2r_0}{2m_0} = 2r_0 - r_0 = r_0 > 0 .
\end{align*}
Let now $\delta>0$ be so small that 
\begin{equation} \label{eq:d1d2c_range}
  \begin{split}
    \bigabs{K_1-d_1(v)} < hK_1 , \quad&\quad \bigabs{K_2-d_2(v)} < hK_2 , \\
    &\makebox[0pt][c]{\text{ and }} \\
    c(v)^2 < \min\Bigl\{ K_1\bigl(K_2-h(K_1+K_2)\bigr) \,&,\, (1-K_2)\bigl(1-K_1-h(K_1+K_2)\bigr) \Bigl\} ,
  \end{split}
\end{equation}
for all $v\in\nball[4]_{2\delta}(0)=\bigl\{v'\in\rz[4]\bigm|\lrabs{v'}<2\delta\bigr\}$, and let $u\in\nball[4]_{2\delta}(0)$. Then, by \eqref{eqs:mr}, \eqref{eqs:m0r0}, 
\begin{align*}
  m &= \frac{d_1+d_2}2 = \frac{K_1+K_2}2 + \frac{d_1-K_1}2 + \frac{d_2-K_2}2 \\
    &\in \Bigl( \frac{K_1+K_2}2 - \frac{hK_1}2 - \frac{hK_2}2, \frac{K_1+K_2}2 + \frac{hK_1}2 + \frac{hK_2}2 \Bigr) = \bigl( (1-h)m_0, (1+h)m_0 \bigr) ,
\end{align*}
so that 
\begin{equation*}
  m > (1-h)m_0 \quad\text{ and }\quad 1-m > 1-(1+h)m_0 .
\end{equation*}
Since \eqref{eq:h} implies $h<\frac{r_0}{2m_0}$, we have by \eqref{eq:K1K2}, 
\begin{align*}
  (1-h)m_0 &> m_0 - \frac{r_0}2 = \frac{K_1+K_2}2 - \frac{K_1-K_2}4 = \frac14K_1+\frac34K_2 > 0 \\
\intertext{and}
  (1+h)m_0 &< m_0 + \frac{r_0}2 = \frac{K_1+K_2}2 + \frac{K_1-K_2}4 = \frac34K_1+\frac14K_2 < 1 ,
\end{align*}
i.e.,
\begin{equation*}
  (1-h)m_0 > 0 \quad\text{ and }\quad 1-(1+h)m_0 > 0 .
\end{equation*}
As a result, 
\begin{equation*}
  0 < m < 1 .
\end{equation*}
Define 
\begin{equation} \label{eq:R}
  R = R(u) = \sqrt{r^2+c^2} ,
\end{equation}
and observe that, by \eqref{eqs:mr}, \eqref{eq:d1d2c_range}, \eqref{eqs:m0r0},
\begin{align*}
  R^2 = r^2 + c^2 &< \Bigl( \frac{K_1-K_2}2 + h\frac{K_1+K_2}2 \Bigr)^2 + K_1\bigl(K_2-h(K_1+K_2)\bigr) \\
    &= r_0^2 + 2hr_0m_0 + h^2m_0^2 + (m_0+r_0)(m_0-r_0-2hm_0) \\
    &= \bigl((1-h)m_0\bigr)^2 < m^2 .
\end{align*}
Similarly, \eqref{eqs:mr}, \eqref{eq:d1d2c_range}, \eqref{eqs:m0r0} yield 
\begin{align*}
  R^2 = r^2 + c^2 &< \Bigl( \frac{K_1-K_2}2 + h\frac{K_1+K_2}2 \Bigr)^2 + (1-K_2)\bigl(1-K_1-h(K_1+K_2)\bigr) \\
    &= r_0^2 + 2hr_0m_0 + h^2m_0^2 + (1-m_0+r_0)(1-m_0-r_0-2hm_0) \\
    &= \bigl(1-(1+h)m_0\bigr)^2 < (1-m)^2 .
\end{align*}
Finally, \eqref{eqs:mr}, \eqref{eq:d1d2c_range}, \eqref{eqs:m0r0} tell us 
\begin{equation*}
  R = \sqrt{r^2+c^2} > r > \frac{K_1-K_2}2 - h\frac{K_1+K_2}2 = r_0 - h m_0 , 
\end{equation*}
whence, by \eqref{eq:h}, 
\begin{align*}
  m - R &< (1+h)\frac{K_1+K_2}2 - R < m_0-r_0 + 2hm_0 < m_0  \\
\intertext{and}
  m + R &> (1-h)\frac{K_1+K_2}2 + R > m_0+r_0 - 2hm_0 > m_0 .
\end{align*}
We have thus showed that 
\begin{equation} \label{eq:mpmR_range}
  0 < m - R < m_0 < m + R < 1 .
\end{equation}
We are then ready to compute the eigenvalues of $a$. Writing (cf.~\eqref{eq:a}, \eqref{eqs:mr}) 
\begin{equation*}
  a = \begin{pmatrix}
    0&0&0&1 \\ 0&0&-1&0 \\ -c&-m+r&0&0 \\ m+r&c&0&0 
  \end{pmatrix} ,
\end{equation*}
we calculate 
\begin{align*}
  \det(\lambda\id-a) %
    &= \lambda^4 - \lambda^2(m-r) - \lambda^2(m+r) + (m^2-r^2-c^2) \\
    &= (\lambda^2)^2 - 2m\lambda^2 + (m^2-R^2) \\
    &= (\lambda^2)^2 - \bigl((m+R)+(m-R)\bigr)\lambda^2 + (m+R)(m-R) \\
    &= \bigl( \lambda^2-(m+R) \bigr)\bigl( \lambda^2-(m-R) \bigr) ,
\end{align*}
and obtain the eigenvalues $\lambda$ of $a$: 
\begin{equation} \label{eq:lambda}
  \lambda = \lambda(u) = \pm\sqrt{m\pm R} .
\end{equation}
Notice that these are, by \eqref{eq:mpmR_range}, real and of modulus strictly smaller than one. Moreover, the four values are distinct, uniformly on $\nball[4]_\delta(0)$. Indeed, by \eqref{eq:mpmR_range}, we have on $\nball[4]_{2\delta}(0)$, 
\begin{equation} \label{eqs:lambda_range}
  -1 < -\sqrt{m+R} < -\sqrt{m_0} < -\sqrt{m-R} < 0 < \sqrt{m-R} < \sqrt{m_0} < \sqrt{m+R} < 1 ,
\end{equation}
which, therefore, holds uniformly on $\nball[4]_\delta(0)$. As a result, the second statement in Lemma~\ref{lem:theory_applies} is proved. For the rest of this subsection, we shall assume that $u\in\nball[4]_\delta(0)$. Notice that, for $u=0$, we obtain by \eqref{eqs:mr}, \eqref{eq:R}, \eqref{eqs:d1d2c_at_zero},
\begin{equation} \label{eq:lambda_at_zero}
  \lambda(0) = \left\{ \begin{aligned}
    &\pm\sqrt{K_1} \\
    &\pm\sqrt{K_2}
  \end{aligned} \right. \qquad.
\end{equation}
Also, it is easy to see that (see \eqref{eq:a} and \eqref{eqs:d1d2c_at_zero}) 
\begin{equation} \label{eq:ei_at_zero}
  \postsub{e}{0}_1 = \begin{pmatrix} 1 \\ 0 \\ 0 \\ \sqrt{K_1} \end{pmatrix} ,\  \postsub{e}{0}_2 = \begin{pmatrix} 0 \\ 1 \\ -\sqrt{K_2} \\ 0 \end{pmatrix} ,\  \postsub{e}{0}_3 = \begin{pmatrix} 0 \\ 1 \\ \sqrt{K_2} \\ 0 \end{pmatrix} \quad\text{ and }\quad \postsub{e}{0}_4 = \begin{pmatrix} 1 \\ 0 \\ 0 \\ -\sqrt{K_1} \end{pmatrix} 
\end{equation}
are eigenvectors of $a(0)$ corresponding to the eigenvalues $\lambda_1(0)=\sqrt{K_1}$, $\lambda_2(0)=\sqrt{K_2}$, $\lambda_3(0)=-\sqrt{K_2}$ and $\lambda_4(0)=-\sqrt{K_1}$, respectively. Now note that the basis $\{\postsub{e}{0}_i\}_{i\in\{1,\ldots,4\}}$ is orthonormal with respect to the scalar product 
\begin{equation} \label{eq:at_zero_scprod}
  \scprod{v,w}_0 = \sum_{i,j} g_{ij}v^iw^j , \qquad v,w\in\rz[4] ,
\end{equation}
where $g_{ij}$, $i,j\in\{1,\ldots,4\}$, are the components of the matrix 
\begin{equation} \label{eq:at_zero_scprod_matrix}
  g = \diag\Bigl( \frac12, \frac12, \frac1{2K_2}, \frac1{2K_1} \Bigr) .
\end{equation}
We will use the norms $\norm{\argplaceholder}_0$ and $\norm{\argplaceholder}_{0^\star}$ induced by this scalar product on $\rz[4]$ and $\postsup{\rz[4]}{\star}$, respectively, to normalise the eigenvectors of $a(u)$ and measure their duals (cf.~Subsection~\ref{subsec:cauchyproblem}).

For $c\neq0$, let 
\begin{equation} \label{eqs:mumuhat}
  \mu = \frac{\lambda^2-d_2}{c} = \frac1c\bigl(r\pm R\bigr) \qquad\text{ and }\qquad \widehat\mu = -\frac{\lambda^2-d_1}{c} = \frac1c\bigl(r\mp R\bigr) .
\end{equation}
We have 
\begin{equation} \label{eqs:mumuhat_prod_sum}
  \mu\widehat\mu = \frac1{c^2}( r^2-R^2 ) = -1 \qquad\text{ and }\qquad \mu+\widehat\mu = \frac{2r}c .
\end{equation}
One easily checks that 
\begin{subequations} \label{eqs:estaretilde}
  \begin{equation} \label{eq:etilde}
    \widetilde e_\lambda = \sigma_\mu\begin{pmatrix} \mu \\ 1 \\ -\lambda \\ \lambda\mu \end{pmatrix} ,
  \end{equation}
  where 
  \begin{equation} \label{eq:sigmamu}
    \sigma_\mu = \sign(\mu)\delta_{\lambda,\pm\sqrt{m+R}} + \delta_{\lambda,\pm\sqrt{m-R}} \qquad \in\{-1,1\} , 
  \end{equation}
  is an eigenvector of $a$ for the eigenvalue $\lambda$ with dual
  \begin{equation} \label{eq:estartilde}
    \postsup{\widetilde e}{\star}^\lambda = \frac{\sigma_\mu}{2\lambda(1+\mu^2)}\bigl( \lambda\mu, \lambda, -1, \mu \bigr) .
  \end{equation}
\end{subequations}
Having 
\begin{align*}
  \norm{\widetilde e_\lambda}_0^2 = \scprod{\widetilde e_\lambda,\widetilde e_\lambda}_0 &= \frac12\Bigl( 1 + \mu^2 + \frac{\lambda^2}{K_2} + \frac{\lambda^2\mu^2}{K_1} \Bigr) \\
\intertext{and}
  \norm{\postsup{\widetilde e}{\star}^\lambda}_{0^\star}^2 = \scprod{\postsup{\widetilde e}{\star}^\lambda,\postsup{\widetilde e}{\star}^\lambda}_{0^\star} &= \frac1{8\lambda^2(1+\mu^2)^2}\Bigl( (2\lambda\mu)^2 + (2\lambda)^2 + \frac{(-2K_2)^2}{K_2} + \frac{(2K_1\mu)^2}{K_1} \Bigr) \\
    &= \frac1{2\lambda^2(1+\mu^2)^2}\bigl( \lambda^2 + \lambda^2\mu^2 + K_2 + K_1\mu^2 \bigr) ,
\end{align*}
we normalise 
\begin{equation} \label{eqs:estare}
  \begin{split}
    e_\lambda &= \frac{\sigma_\mu}{\sqrt{\frac{1+\mu^2}2 + \frac{\lambda^2}2\bigl(\frac1{K_2}+\frac{\mu^2}{K_1}\bigr)}}\begin{pmatrix} \mu\\1\\-\lambda\\\lambda\mu \end{pmatrix} , \\
    \postsup{e}{\star}^\lambda &= \frac{\sigma_\mu\sqrt{\frac{1+\mu^2}2 + \frac{\lambda^2}2\bigl(\frac1{K_2}+\frac{\mu^2}{K_1}\bigr)}}{2\lambda(1+\mu^2)}\bigl(\lambda\mu, \lambda, -1, \mu \bigr) ,
  \end{split}
\end{equation}
so that 
\begin{multline} \label{eqs:estare_normalising_conditions}
  \norm{e_\lambda}_0 = 1 , \qquad \postsup{e}{\star}^\lambda e_{\lambda'} = \ul{\delta\lambda{\lambda'}} \\
    \text{ and }\qquad \norm{\postsup{e}{\star}^\lambda}_{0^\star} = \frac{\sqrt{\Bigl(1 + \mu^2 + \lambda^2\bigl(\frac1{K_2} + \frac{\mu^2}{K_1}\bigr)\Bigr)\bigl(\lambda^2 (1+\mu^2) + K_2 + K_1\mu^2\bigr)}}{2\abs{\lambda}(1+\mu^2)} ,
\end{multline}
for 
\begin{equation*}
  \postsup{\widetilde e}{\star}^\lambda\widetilde e_{\lambda'} = \frac{\sigma_\mu\sigma_{\mu'}}{2\lambda(1+\mu^2)}\bigl(\lambda\mu\mu'+\lambda+\lambda'+\lambda'\mu\mu'\bigr) = \frac{\sigma_\mu\sigma_{\mu'}(\lambda+\lambda')(1+\mu\mu')}{2\lambda(1+\mu^2)} 
\end{equation*}
and, in case $\lambda\neq\lambda'$, we have either $\mu=\mu'$, whence $\lambda=-\lambda'$, or $\mu=\widehat{\mu'}$, whence $\mu\mu'=-1$ by \eqref{eqs:mumuhat_prod_sum}.

On the other hand, for $c=0$ (which occurs, e.g., when $u=0$, or when $C_{112}=C_{122}=0$; cf.~\eqref{eq:c}), $a$ becomes 
\begin{equation*}
  a = \begin{pmatrix}
    0&0&0&1 \\ 0&0&-1&0 \\ 0&-d_2&0&0 \\ d_1&0&0&0 
  \end{pmatrix} ,
\end{equation*}
and, since $R=r$ in this case, we obtain 
\begin{equation*}
  \lambda = \left\{ \begin{aligned} &\pm\sqrt{d_1} \\ &\pm\sqrt{d_2} \end{aligned} \right. \qquad.
\end{equation*}
Now the problem decouples into a pair of two--dimensional equations for $(u^1,u^4)$ and $(u^2,u^3)$, respectively (i.e., for $(D_y,B_z)$ and $(D_z,B_y)$, respectively). The (normalised) eigenvectors of $a$ are easily seen to be given by
\begin{align*}
  e_{\pm\sqrt{d_1}} &= \frac1{\sqrt{\frac12+\frac{d_1}{2K_1}}}\begin{pmatrix} 1\\0\\0\\\pm\sqrt{d_1} \end{pmatrix} , & \postsup{e}{\star}^{\pm\sqrt{d_1}} &= \frac{\sqrt{\frac12+\frac{d_1}{2K_1}}}{\pm2\sqrt{d_1}}\bigl(\pm\sqrt{d_1},0,0,1\bigr) , \\
  e_{\pm\sqrt{d_2}} &= \frac1{\sqrt{\frac12+\frac{d_2}{2K_2}}}\begin{pmatrix} 0\\1\\\mp\sqrt{d_2}\\0 \end{pmatrix} , & \postsup{e}{\star}^{\pm\sqrt{d_2}} &= \frac{\sqrt{\frac12+\frac{d_2}{2K_2}}}{\pm2\sqrt{d_2}}\bigl(0,\pm\sqrt{d_2},-1,0\bigr) .
\end{align*}

We show that our eigenvectors \eqref{eqs:estare} for $c\neq0$ can be extended by continuity to $c=0$ and coincide there with the appropriate eigenvector from above. Let 
\begin{equation} \label{eq:mupm}
  \mu_{\pm} = \frac1c\bigl(r\pm R\bigr) = \frac1c\Bigl(r\pm\sqrt{r^2+c^2}\Bigr) .
\end{equation}
For $0<\lrabs{c}\ll r$, we have 
\begin{equation*}
  R = r + \frac{c^2}{2r} + \mathcal O\Bigl( \frac{\lrabs{c}^4}{r^3} \Bigr) .
\end{equation*}
Thus 
\begin{equation*}
  \mu_\pm = \frac{r\pm r}c \pm \frac{c}{2r} + \mathcal O\biggl(\Bigl(\frac{\lrabs{c}}r\Bigr)^3\biggr) ,
\end{equation*}
so that 
\begin{equation} \label{eqs:mupm_leadingorder}
  \mu_+ = \frac{2r}{c} + \mathcal O\Bigl(\frac{\lrabs{c}}r\Bigr) \qquad\text{ and }\qquad \mu_- = -\frac{c}{2r} + \mathcal O\biggl(\Bigl(\frac{\lrabs{c}}r\Bigr)^3\biggr) ,
\end{equation}
and $\sign(\mu_\pm)=\sign(c)$ for $\abs{c}$ small enough. A short computation then gives 
\begin{align*}
  \frac{\mu_+}{\sqrt{1+\mu_+^2}} &= \sign(\mu_+) + \mathcal O\Bigl(\frac{\lrabs{c}}r\Bigr) , & \frac1{\sqrt{1+\mu_+^2}} &= \mathcal O\Bigl(\frac{\lrabs{c}}r\Bigr) , \\
  \frac{\mu_-}{\sqrt{1+\mu_-^2}} &= \mathcal O\Bigl(\frac{\lrabs{c}}r\Bigr) , & \frac1{\sqrt{1+\mu_-^2}} &= 1 + \mathcal O \Bigl( \frac{\lrabs{c}^2}{r^2} \Bigr) .
\end{align*}
As a consequence,
\begin{align*}
  e_{\pm\sqrt{m+R}} &\stackrel{c\to0}\longrightarrow e_{\pm\sqrt{d_1}} , & \postsup{e}{\star}^{\pm\sqrt{m+R}} &\stackrel{c\to0}\longrightarrow \postsup{e}{\star}^{\pm\sqrt{d_1}} , \\
  e_{\pm\sqrt{m-R}} &\stackrel{c\to0}\longrightarrow e_{\pm\sqrt{d_2}} , & \postsup{e}{\star}^{\pm\sqrt{m-R}} &\stackrel{c\to0}\longrightarrow \postsup{e}{\star}^{\pm\sqrt{d_2}} ,
\end{align*}
which proves the claim.

Let us now compute the ``gradient'' of the matrix $a$. Let $\dot u = \bigl( \dot u^1, \dot u^2, \dot u^3, \dot u^4 \bigr)^T$ be a variation of $u$. Since the coefficients of $a$ are affine functions of the components of $u$, we obtain simply 
\begin{equation} \label{eq:a_variation}
  \frac1s \bigl( a(u+s\dot u)-a(u) \bigr) = \begin{pmatrix}
         0    &     0     &  0 &  0 \\
         0    &     0     &  0 &  0 \\
     -\dot c  & -\dot d_2 &  0 &  0 \\
     \dot d_1 &   \dot c  &  0 &  0 
    \end{pmatrix} , \qquad \forall s\in\rz ,
\end{equation}
where 
\begin{subequations} \label{eqs:d1d2c_dotted}
\begin{align}
\label{eq:d1dot}
  \dot d_1 &= 6 C_{111} \dot u^1 + 2 C_{112} \dot u^2 = 2 \bigl( 3 C_{111} \dot u^1 +   C_{112} \dot u^2 \bigr) , \\
\label{eq:d2dot}
  \dot d_2 &= 2 C_{122} \dot u^1 + 6 C_{222} \dot u^2 = 2 \bigl(   C_{122} \dot u^1 + 3 C_{222} \dot u^2 \bigr) , \\
\label{eq:cdot}
    \dot c &= 2 C_{112} \dot u^1 + 2 C_{122} \dot u^2 = 2 \bigl(   C_{112} \dot u^1 +   C_{122} \dot u^2 \bigr) .
\end{align}
\end{subequations}
It follows 
\begin{align*}
  \postsup{e}{\star}^\lambda\bigl(D_{\dot u} a\bigr)e_{\lambda'} %
    &= \frac{\sigma_\mu\sigma_{\mu'}}{2\lambda(1+\mu^2)}\sqrt{\frac{1+\mu^2 + \lambda^2\bigl(\frac1{K_2}+\frac{\mu^2}{K_1}\bigr)}{1+\mu'^2 + \lambda'^2\bigl(\frac1{K_2}+\frac{\mu'^2}{K_1}\bigr)}}\bigl( \lambda\mu,\lambda,-1,\mu\bigr)\begin{pmatrix} 0\\0\\-\mu'\dot c-\dot d_2\\\mu'\dot d_1+\dot c \end{pmatrix} \\
    &= \sigma_\mu\sigma_{\mu'}\sqrt{\frac{1+\mu^2 + \lambda^2\bigl(\frac1{K_2}+\frac{\mu^2}{K_1}\bigr)}{1+\mu'^2 + \lambda'^2\bigl(\frac1{K_2}+\frac{\mu'^2}{K_1}\bigr)}}\frac{\mu\mu'\dot d_1+\dot d_2 +(\mu+\mu')\dot c}{2\lambda(1+\mu^2)} .
\end{align*}
As a result, we obtain for $\dot u = e_{\lambda''}$, 
\begin{align}
\nonumber
  \cull{c\lambda{\lambda'}{\lambda''}} &= \cull{c\lambda{\lambda'}{\lambda''}}(u)
    = \postsup{e}{\star}^\lambda(u)\bigl(D_{e_{\lambda''}(u)}a(u)\bigr)e_{\lambda'}(u) \\
\nonumber
    &= \sigma_\mu\sigma_{\mu'}\sigma_{\mu''}\sqrt{2\frac{1+\mu^2 + \lambda^2\bigl(\frac1{K_2}+\frac{\mu^2}{K_1}\bigr)}{\Bigl(1+\mu'^2 + \lambda'^2\bigl(\frac1{K_2}+\frac{\mu'^2}{K_1}\bigr)\Bigr)\Bigl(1+\mu''^2 + \lambda''^2\bigl(\frac1{K_2}+\frac{\mu''^2}{K_1}\bigr)\Bigr)}} \\
\nonumber
    &\phantom{=\ } {} \cdot \frac{\mu\mu'(3C_{111}\mu''+C_{112})+(C_{122}\mu''+3C_{222})+(\mu+\mu')(C_{112}\mu''+C_{122})}{\lambda(1+\mu^2)} \\
\label{eq:cllplpp}
    &= \sigma_\mu\sigma_{\mu'}\sigma_{\mu''}\sqrt{2\frac{1+\mu^2 + \lambda^2\bigl(\frac1{K_2}+\frac{\mu^2}{K_1}\bigr)}{\Bigl(1+\mu'^2 + \lambda'^2\bigl(\frac1{K_2}+\frac{\mu'^2}{K_1}\bigr)\Bigr)\Bigl(1+\mu''^2 + \lambda''^2\bigl(\frac1{K_2}+\frac{\mu''^2}{K_1}\bigr)\Bigr)}} \\
\nonumber
    &\phantom{=\ } {} \cdot \frac{ 3C_{111}\mu\mu'\mu'' + C_{112}(\mu\mu'+\mu\mu''+\mu'\mu'') + C_{122}(\mu+\mu'+\mu'') + 3C_{222} }{\lambda(1+\mu^2)} .
\end{align}
Notice that $\cull{c\lambda{\lambda''}{\lambda'}}=\cull{c\lambda{\lambda'}{\lambda''}}$. Notice also that at least two of the $\mu,\mu',\mu''$ are identical. In particular, we have 
\begin{equation} \label{eq:clll}
  \cull{c\lambda\lambda\lambda} = \frac{3\sigma_\mu^3}{\lambda(1+\mu^2)\sqrt{\frac{1+\mu^2}2 + \frac{\lambda^2}2\bigl(\frac1{K_2}+\frac{\mu^2}{K_1}\bigr)}}\bigl( C_{111}\mu^3 + C_{112}\mu^2 + C_{122}\mu + C_{222} \bigr) .
\end{equation}
As a consequence, we have for 
\begin{align*}
  \lambda&=\pm\sqrt{m+R} \quad: & \cull{c\lambda\lambda\lambda} \stackrel{c\to0}\longrightarrow \frac3{\pm\sqrt{d_1}\sqrt{\frac12+\frac{d_1}{2K_1}}}C_{111} , \\
\intertext{and for}
  \lambda&=\pm\sqrt{m-R} \quad: & \cull{c\lambda\lambda\lambda} \stackrel{c\to0}\longrightarrow \frac3{\pm\sqrt{d_2}\sqrt{\frac12+\frac{d_2}{2K_2}}}C_{222} ,
\end{align*}
which shows that the system $\frac{\partial}{\partial t}u+a\frac{\partial}{\partial x}u=0$ is genuinely non--linear around $u=0$, as soon as $C_{111},C_{222}\neq 0$. This establishes the last statement in Lemma~\ref{lem:theory_applies}, completing the proof.

\begin{remark*}
  If we want to ensure that $\cull{c\lambda\lambda\lambda}<0$ for all sufficiently small $\abs{u}$, we replace $\postsup{e}{\star}^\lambda$ and $e_\lambda$ by $\sigma_\lambda\postsup{e}{\star}^\lambda$ and $\sigma_\lambda e_\lambda$, respectively, where 
  \begin{equation} \label{eq:sigmal}
    \sigma_\lambda = - \sign(C_{111})\sign(\lambda)\delta_{\lambda,\pm\sqrt{m+R}} - \sign(C_{222})\sign(\lambda)\delta_{\lambda,\pm\sqrt{m-R}} \qquad\in\{-1,1\} .
  \end{equation}
  (Note that, by definition, $\sigma_{-\lambda}=-\sigma_\lambda$.) We then have 
  \begin{align}
  \label{eq:cllplpp_signed}
    \cull{c\lambda{\lambda'}{\lambda''}} &= \frac{\sigma_\mu\sigma_{\mu'}\sigma_{\mu''}\sigma_\lambda\sigma_{\lambda'}\sigma_{\lambda''}}{\lambda(1+\mu^2)}\sqrt{2\frac{1+\mu^2 + \lambda^2\bigl(\frac1{K_2}+\frac{\mu^2}{K_1}\bigr)}{\Bigl(1+\mu'^2 + \lambda'^2\bigl(\frac1{K_2}+\frac{\mu'^2}{K_1}\bigr)\Bigr)\Bigl(1+\mu''^2 + \lambda''^2\bigl(\frac1{K_2}+\frac{\mu''^2}{K_1}\bigr)\Bigr)}} \\
  \nonumber
      &\phantom{=\ } {} \cdot \bigl( 3C_{111}\mu\mu'\mu'' + C_{112}(\mu\mu'+\mu\mu''+\mu'\mu'') + C_{122}(\mu+\mu'+\mu'') + 3C_{222} \bigr)
  \end{align}
  and
  \begin{equation} \label{eq:clll_signed}
    \cull{c\lambda\lambda\lambda} = \frac{3\sigma_\mu^3\sigma_\lambda^3}{\lambda(1+\mu^2)\sqrt{\frac{1+\mu^2}2 + \frac{\lambda^2}2\bigl(\frac1{K_2}+\frac{\mu^2}{K_1}\bigr)}}\bigl( C_{111}\mu^3 + C_{112}\mu^2 + C_{122}\mu + C_{222} \bigr) .
  \end{equation}
\end{remark*}

\begin{remark*}
  Observe that, for $u=0$, we have 
  \begin{equation} \label{eqs:absclll_at_zero_minmax}
    \begin{split}
      \min_{\lambda} \lrabs{\cull{c\lambda\lambda\lambda}(0)} &= \min \biggl\{ \frac3{\sqrt{K_1}}\lrabs{C_{111}} \;,\; \frac3{\sqrt{K_2}}\lrabs{C_{222}} \biggr\} , \\
      \max_{\lambda} \lrabs{\cull{c\lambda\lambda\lambda}(0)} &= \max \biggl\{ \frac3{\sqrt{K_1}}\lrabs{C_{111}} \;,\; \frac3{\sqrt{K_2}}\lrabs{C_{222}} \biggr\} .
    \end{split}
  \end{equation}
  In particular, we have $\min_\lambda \lrabs{\cull{c\lambda\lambda\lambda}(0)}=\max_\lambda \lrabs{\cull{c\lambda\lambda\lambda}(0)}$ if and only if 
  \begin{equation*}
    \frac{\lrabs{C_{111}}}{\lrabs{C_{222}}} = \sqrt{\frac{K_1}{K_2}} ,
  \end{equation*}
  which, by \eqref{eq:K1K2}, is strictly larger than one.
\end{remark*}

This finishes the discussion that our theory is applicable to plane electromagnetic waves in a birefringent crystal with no magnetic properties and cubic non--linearity in the energy density, assuming no dispersion. Except for a smallness condition on the field we note that we only need to require $C_{111},C_{222}\neq0$. We proceed by observing in the next two subsections that there exist Riemann invariants in the decoupled case $c\equiv0$, thus simplifying the problem considerably, whereas no such simplification occurs in the case $c\not\equiv0$.

\subsection{\texorpdfstring{On the existence of Riemann invariants in the case $c\equiv0$}{On the existence of Riemann invariants in the case c=0}} \label{subsec:czerocase}

Let us assume that $c\equiv0$ (i.e., we assume that $C_{112}=C_{122}=0$; cf.~\eqref{eq:c}). We already observed that this implies the decoupling of the two polarisations. Now we show that this actually implies the existence of Riemann invariants.

First, notice that $C_{112}=C_{122}=0$ implies that $d_1(u)=d_1(u^1)$ and $d_2(u)=d_2(u^2)$ (cf.~\eqref{eq:d1} and \eqref{eq:d2}). So, indeed, the two decoupled systems do not depend on one another. We can then define 
\begin{equation} \label{eqs:m}
  \left\{ \begin{aligned}
    m^1 = m^1(u^1,u^4) &= u^4 + \int_0^{u^1}\sqrt{d_1(s)}ds , & m^2 = m^2(u^2,u^3) &= -u^3 + \int_0^{u^2}\sqrt{d_2(s)}ds , \\
    m^3 = m^3(u^2,u^3) &= -u^3 - \int_0^{u^2}\sqrt{d_2(s)}ds , & m^4 = m^4(u^1,u^4) &= u^4 - \int_0^{u^1}\sqrt{d_1(s)}ds .
  \end{aligned} \right.
\end{equation}
It is easy to see that $a$ becomes diagonal in these coordinates. Indeed, we have
\begin{equation*}
  dm(u) = \begin{pmatrix}
     \sqrt{d_1(u^1)} &        0         &  0 & 1 \\
           0         &  \sqrt{d_2(u^2)} & -1 & 0 \\
           0         & -\sqrt{d_2(u^2)} & -1 & 0 \\
    -\sqrt{d_1(u^1)} &        0         &  0 & 1 
    \end{pmatrix} ,
\end{equation*}
so that 
\begin{equation*}
  \bigl(dm(u)\bigr)^{-1} = \frac12 \begin{pmatrix}
    \frac1{\sqrt{d_1(u^1)}} &            0            &            0             & -\frac1{\sqrt{d_1(u^1)}} \\
              0             & \frac1{\sqrt{d_2(u^2)}} & -\frac1{\sqrt{d_2(u^2)}} &             0            \\
              0             &           -1            &           -1             &             0            \\
              1             &            0            &            0 		 &             1            
    \end{pmatrix} ,
\end{equation*}
and therefore 
\begin{equation*}
  dm(u)a(u)\bigl(dm(u)\bigr)^{-1} = \begin{pmatrix}
    \sqrt{d_1(u^1)} &       0         &        0         &        0         \\
          0         & \sqrt{d_2(u^2)} &        0         &        0         \\
          0         &       0         & -\sqrt{d_2(u^2)} &        0         \\
          0         &       0         &        0         & -\sqrt{d_1(u^1)} 
    \end{pmatrix} .
\end{equation*}
We thus obtain that $\frac{\partial}{\partial t}u+a(u)\frac{\partial}{\partial x}u=0$ is equivalent to the system
\begin{equation*}
  \left\{ \begin{aligned}
    \frac{\partial}{\partial t} m^1 + \sqrt{d_1(u^1)}\frac{\partial}{\partial x} m^1 &= 0 , \\
    \frac{\partial}{\partial t} m^2 + \sqrt{d_2(u^2)}\frac{\partial}{\partial x} m^2 &= 0 , \\
    \frac{\partial}{\partial t} m^3 - \sqrt{d_2(u^2)}\frac{\partial}{\partial x} m^3 &= 0 , \\
    \frac{\partial}{\partial t} m^4 - \sqrt{d_1(u^1)}\frac{\partial}{\partial x} m^4 &= 0 ,
  \end{aligned} \right.
\end{equation*}
which we write more compactly as 
\begin{equation} \label{eq:m_pde_udependent}
  \frac{\partial}{\partial t} m^{i(1-j)+(5-i)j} + (-1)^j\sqrt{d_i(u^i)}\frac{\partial}{\partial x} m^{i(1-j)+(5-i)j} = 0 , \qquad i\in\{1,2\},j\in\{0,1\} .
\end{equation}

We still need to express $u^1$ and $u^2$ in terms of $m^1$, $m^2$, $m^3$ and $m^4$. From \eqref{eq:d1}, \eqref{eq:d2} we have 
\begin{equation*}
  \sqrt{d_i(u^i)} = \sqrt{K_i+6C_{iii}u^i} , \qquad i\in\{1,2\} ,
\end{equation*}
whence 
\begin{equation} \label{eq:int_sqrt_di}
  \int_0^{u^i}\sqrt{d_i(s)}ds = \left\{ \begin{aligned} 
    &\frac{(K_i+6C_{iii}u^i)^{\frac32}-(K_i)^{\frac32}}{9C_{iii}} , & \text{if $C_{iii}\neq0$,} \\
    &\sqrt{K_i}u^i , & \text{if $C_{iii}=0$,}
  \end{aligned} \right. \qquad i\in\{1,2\} .
\end{equation}
But from \eqref{eqs:m}, we see that this also equals $\frac12(m^i-m^{5-i})$. As a result, we obtain for $i\in\{1,2\}$ 
\begin{equation} \label{eq:ui_in_terms_of_m}
  u^i = u^i(m^i,m^{5-i}) = \left\{ \begin{aligned}
    & \frac{\Bigl((K_i)^{\frac32} + \frac{9C_{iii}}2(m^i-m^{5-i})\Bigr)^{\frac23}-K_i}{6C_{iii}} , & \text{if $C_{iii}\neq0$,} \\
    & \frac1{2\sqrt{K_i}}(m^i-m^{5-i}) , & \text{if $C_{iii}=0$.}
  \end{aligned} \right.
\end{equation}
Inserting this back into \eqref{eq:d1}, \eqref{eq:d2} and then into \eqref{eq:m_pde_udependent}, we obtain for $i\in\{1,2\},j\in\{0,1\}$
\begin{equation*}
  \left\{ \begin{aligned}
    &\frac{\partial}{\partial t} m^{i-2ij+5j} + (-1)^j \biggl( (K_i)^{\frac32} + \frac{9C_{iii}}2(m^i-m^{5-i}) \biggr)^{\frac13} \frac{\partial}{\partial x} m^{i-2ij+5j} = 0 , & \text{if $C_{iii}\neq 0$,} \\ 
    &\frac{\partial}{\partial t} m^{i-2ij+5j} + (-1)^j \sqrt{K_i} \frac{\partial}{\partial x} m^{i-2ij+5j} = 0 , & \text{if $C_{iii}=0$.} 
  \end{aligned} \right.
\end{equation*}
The first equation being compatible with the second, we have for every $C_{iii}\in\rz$, 
\begin{multline} \label{eq:m_pde}
  \frac{\partial}{\partial t} m^{i-2ij+5j}  + (-1)^j \biggl( (K_i)^{\frac32} + \frac{9C_{iii}}2(m^i-m^{5-i}) \biggr)^{\frac13} \frac{\partial}{\partial x} m^{i-2ij+5j} = 0 , \\
  i\in\{1,2\}, j\in\{0,1\} .
\end{multline}
We conclude that $m^{i-2ij+5j}$ is constant along the $(i-2ij+5j)^\text{th}$ characteristic with speed 
\begin{equation} \label{eq:m_char_speeds}
  \lambda_{i-2ij+5j} = (-1)^j \biggl( (K_i)^{\frac32} + \frac{9C_{iii}}2(m^i-m^{5-i}) \biggr)^{\frac13} .
\end{equation}
As a result, the characteristics must be straight lines as soon as we are outside a region where characteristic strips overlap.

Next, we want to investigate whether or not a similar argument can be performed in the case $c\not\equiv0$. It turns out that the answer is negative.

\subsection{\texorpdfstring{On the non--existence of Riemann invariants in the case $c\not\equiv0$}{On the non--existence of Riemann invariants in the case c!=0}} \label{subsec:cnonzerocase}

Notice that, in general, the existence of Riemann invariants for an $N$--dimensional ($N\in\nz\setminus\{0\}$) quasi--linear system 
\begin{equation*}
  \frac{\partial}{\partial t} u + a(u)\frac{\partial}{\partial x} u = 0
\end{equation*}
depends on the existence of coordinates $m=m(u)$ in $\rz[N]$ such that the rows of $dm$ are (each separately) proportional to the elements of a basis $\{\postsup{e}{\star}^i\}_{i=1}^N$ of eigen--$1$--forms of $a$ (cf.~Remark~\ref{rem:geometric_interpretation} on page~\pageref{rem:geometric_interpretation}), i.e., there must exist functions $\mu_i$ such that 
\begin{equation*}
  \mu_i dm^i = \postsup{e}{\star}^i , \qquad i\in\{1,\ldots,N\} .
\end{equation*}
We have the necessary integrability condition 
\begin{equation*}
  d\postsup{e}{\star}^i\wedge \postsup{e}{\star}^i = (d\mu_i\wedge dm^i + \mu_i d^2m^i)\wedge(\mu_i dm^i) = \mu_i d\mu_i\wedge(dm^i\wedge dm^i) = 0 .
\end{equation*}

In order to check this for the case at hand (i.e., $N=4$, $a$ given by \eqref{eq:a} and $c\not\equiv0$), we compute
\begin{align*}
  d\postsup{\widehat e}{\star}^\lambda 
    &=   \bigl( \partial_{u^1}\lambda - \partial_{u^2}(\lambda\mu) \bigr) du^1\wedge du^2 
       + \bigl( -\partial_{u^3}(\lambda\mu) \bigr) du^1\wedge du^3 
       + \bigl( \partial_{u^1}\mu - \partial_{u^4}(\lambda\mu) \bigr) du^1\wedge du^4 \\
    &\phantom{=\ } {} + ( -\partial_{u^3}\lambda ) du^2\wedge du^3 
       + ( \partial_{u^2}\mu - \partial_{u^4}\lambda ) du^2\wedge du^4 
       + ( \partial_{u^3}\mu ) du^3\wedge du^4 ,
\end{align*}
where we take 
\begin{equation*}
\postsup{\widehat e}{\star}^\lambda = \lambda\mu du^1 + \lambda du^2 - du^3 + \mu du^4 
\end{equation*}
for the dual eigenvectors of $a$ (compare with \eqref{eq:estartilde}).
Since, by \eqref{eqs:mumuhat}, \eqref{eq:lambda}, \eqref{eq:R}, \eqref{eqs:mr}, \eqref{eq:d1}, \eqref{eq:d2}, \eqref{eq:c}, the quantities $\lambda$ and $\mu$ are independent of $u^3$ and $u^4$, this reduces to 
\begin{equation*}
  d\postsup{\widehat e}{\star}^\lambda 
    = \bigl( \partial_{u^1}\lambda - \partial_{u^2}(\lambda\mu) \bigr) du^1\wedge du^2 
      + (\partial_{u^1}\mu) du^1\wedge du^4
      + (\partial_{u^2}\mu) du^2\wedge du^4 .
\end{equation*}
Therefore, 
\begin{align*}
  d\postsup{\widehat e}{\star}^\lambda\wedge\postsup{\widehat e}{\star}^\lambda 
    &= \bigl( - \partial_{u^1}\lambda + \partial_{u^2}(\lambda\mu) \bigr) du^1\wedge du^2\wedge du^3 \\
    &\phantom{=\ } {} + \bigl( \mu\partial_{u^1}\lambda - \mu\partial_{u^2}(\lambda\mu) - \lambda\partial_{u^1}\mu + \lambda\mu\partial_{u^2}\mu \bigr) du^1\wedge du^2\wedge du^4 \\
    &\phantom{=\ } {} + (\partial_{u^1}\mu) du^1\wedge du^3\wedge du^4
                  + (\partial_{u^2}\mu) du^2\wedge du^3\wedge du^4 ,
\end{align*}
which has to vanish. Comparing coefficients, we obtain that $\mu$ needs to be constant and that $\lambda$ has to fulfil 
\begin{equation}
  \partial_{u^1}\lambda = \mu\partial_{u^2}\lambda .
\end{equation}
Since, by the first equation in \eqref{eqs:mumuhat_prod_sum}, $\mu$ being independent of $u$ implies that $\widehat\mu$ is independent of $u$, we obtain from the second equation in \eqref{eqs:mumuhat_prod_sum} that $\frac rc$ has to be a constant as well. On the other hand, considering points closer and closer to the trivial state, \eqref{eqs:mupm_leadingorder} implies that $\mu_-\equiv0$ and $\frac1{\mu_+}\equiv0$. Hence, again by \eqref{eqs:mupm_leadingorder}, $\frac cr\equiv0$, implying that $c\equiv0$. But this is a contradiction to our assumption that $c\not\equiv0$. As a result, no Riemann invariants exist in the case $c\not\equiv0$.

In the remaining part of this section, we want to comment on a more realistic setting one might encounter in experiments.

\subsection{A sketch of an actual experimental setup} \label{subsec:experimental_setup}

If one were to set up an experiment for comparison with the theoretical study conducted thus far, one would need to take into account at least a material boundary. In what follows, we shall assume that our crystal extends infinitely in the half--space $\bigl\{(x,y,z)\in\rz[3]\bigm|x\geq0\bigr\}$ and that a suitable plane wave source is located in the plane $\{x=x_0\}$ for some $x_0<0$. For simplicity, we shall assume that the half--space $\bigl\{(x,y,z)\in\rz[3]\bigm|x<0\bigr\}$ is vacuum, although it should be easy to adapt what follows to the case of, say, air. The initial conditions for the electromagnetic field are obtained from a smooth light profile emanating from $x_0$ and compactly supported in time. We will later choose the time origin so that the light first hits the crystal at $(x,t)=(0,0)$.

In order to deal with the material boundary, we obviously have to set up jump conditions at $x=0$. Those are, as usual, obtained from integrating the Maxwell equations over suitable regions (infinitesimally flat pill boxes, respectively loops, based on arbitrary subdomains, respectively curves, on the boundary) and using Gauss' and Stokes' theorems. In general, one obtains for insulators (at rest) without extraneous charges and currents that the normal components of $B$ and $D$, as well as the tangential components of $E$ and $H$, have to be continuous across the boundary. We denote this as 
\begin{equation} \label{eqs:general_jump_conditions}
  \begin{aligned}[]
    [B_\perp] &= 0 , & [D_\perp] &= 0 , \\
    [E_\parallel] &= 0 , & [H_\parallel] &= 0 .
  \end{aligned}
\end{equation}
For plane waves travelling in the $x$--direction, the boundary being $\{x=0\}$ here, the top line has no effect other than ensuring that if we assumed $B_x=0$, $D_x=0$ initially, the same holds on the other side of the boundary. The bottom line, on the other hand, translates into 
\begin{equation} \label{eqs:jump_conditions_E_H}
  [E_y] = 0 , \quad [E_z] = 0 , \quad [H_y] = 0 , \quad [H_z] = 0 .
\end{equation}
These conditions, of course, have to be converted into jump conditions for $D$ and $B$, using $H=B$ on both sides, as well as $E=D$ if $x<0$ and \eqref{eqs:E_in_terms_of_D} if $x\geq 0$. We observe that 
\begin{equation*}
  \frac{\partial (E_y,E_z)}{\partial (D_y,D_z)} = \begin{pmatrix} d_1 & c \\ c & d_2 \end{pmatrix} ,
\end{equation*}
so that $(D_y,D_z)$ can indeed be expressed in terms of $(E_y,E_z)$ in a neighbourhood of the solution $(D_y,D_z)\bigr|_{(E_y,E_z)=(0,0)}=(0,0)$.

We now give the complete solution of this situation in the decoupled case $c\equiv0$, i.e., when $C_{112}=C_{122}=0$.

\subsection{\texorpdfstring{The decoupled case $c\equiv0$}{The decoupled case c=0}} \label{subsec:decoupled_case}

Both inside and outside the crystal, we take \eqref{eq:energy_third_order} as an ansatz for the energy density. In the region outside the crystal, i.e., when $x<0$, the constants $C_{111}$, $C_{112}$, $C_{122}$ and $C_{222}$ all vanish, and $K_1=K_2$. For simplicity, we shall assume that we are in vacuum, so that $K_1=K_2=1$. Inside the crystal, we assume \eqref{eq:K1K2}, $C_{112}=C_{122}=0$ and, for simplicity, we only consider the genuinely non--linear case $C_{111}\neq0$, $C_{222}\neq0$. Let $u=(D_y,D_z,B_y,B_z)$ be a solution of $\frac{\partial}{\partial t}u+a\frac{\partial}{\partial x}u=0$, where $a$ is given by \eqref{eq:a} with the aforementioned values of the constants in accordance as to whether $x$ is positive or negative, and such that the jump conditions \eqref{eqs:jump_conditions_E_H} are fulfilled. In what follows, we describe this solution in terms of Riemann invariants. We can use Subsection~\ref{subsec:czerocase} to do this in both the vacuum and the crystal regions. To avoid additional sub- or superscripts, we shall denote the Riemann invariants inside the crystal by $m$, but those outside the crystal by $n$. Let us start with the vacuum region ($x<0$). Here, since the equations are linear, we expect the solution $u$ to be the sum of an incident wave $\postsub{u}{(i)}$ travelling in the positive $x$--direction, and a reflected wave $\postsub{u}{(r)}$ travelling in the negative $x$--direction, so that $u=\postsub{u}{(i)}+\postsub{u}{(r)}$. Given an initial two--dimensional profile $\mathfrak f\in C^K_c\bigl(\rz;\nball[2]_{\frac\delta2}(0)\bigr)$ for $(D_y,D_z)=(u^1,u^2)$ with compact support in $(x_0-1,x_0)$, we can parametrise the incident wave $\postsub{u}{(i)}$ by 
\begin{equation} \label{eq:u_vacuum_incident}
  \postsub{u}{(i)}(x,t) = \begin{pmatrix} \mathfrak f^1(x_0+x-t) \\ \mathfrak f^2(x_0+x-t) \\ -\mathfrak f^2(x_0+x-t) \\ \mathfrak f^1(x_0+x-t) \end{pmatrix} , \qquad x\in(x_0,0),
\end{equation}
as can be easily checked using \eqref{eqs:Maxwell}. Here, $\delta$ is chosen small enough according to Subsection~\ref{subsec:theory_applies} for the crystal region. Notice that we chose the time origin in such a way that $t\mapsto\lim\limits_{x\nearrow0}\postsub{u}{(i)}(x,t)$ is compactly supported in $(0,1)$. The Riemann invariants $\postsub{n}{(i)}$ related to this wave are (cf.~\eqref{eqs:m})
\begin{equation} \label{eq:m_vacuum_incident}
  \postsub{n}{(i)}(x,t) = \begin{pmatrix} \postsub{u}{(i)}^1 + \postsub{u}{(i)}^4 \\ \postsub{u}{(i)}^2 - \postsub{u}{(i)}^3 \\ -\postsub{u}{(i)}^2 - \postsub{u}{(i)}^3 \\ -\postsub{u}{(i)}^1 + \postsub{u}{(i)}^4 \end{pmatrix} = \begin{pmatrix} 2\mathfrak f^1(x_0+x-t) \\ 2\mathfrak f^2(x_0+x-t) \\ 0 \\ 0 \end{pmatrix} , \qquad x\in(x_0,0),
\end{equation}
which is consistent with $\postsub{u}{(i)}$ travelling in the positive $x$--direction. Similarly, we can make the following ansatz for the Riemann invariants $\postsub{n}{(r)}$ representing the reflected wave $\postsub{u}{(r)}$, 
\begin{equation} \label{eq:m_vacuum_reflected}
  \postsub{n}{(r)}(x,t) = \begin{pmatrix} 0 \\ 0 \\ 2\mathfrak g^2(x,t) \\ 2\mathfrak g^1(x,t) \end{pmatrix} , \qquad x<0,
\end{equation}
for two unknown functions $\mathfrak g^1$ and $\mathfrak g^2$ (notice that, by linearity, $n=\postsub{n}{(i)}+\postsub{n}{(r)}$). Since the latter are constant along the characteristics with speed $-1$, they must depend solely on $x+t$. Slightly absuing notation, we then get for the reflected wave $\postsub{u}{(r)}$ using \eqref{eqs:m}, \eqref{eq:ui_in_terms_of_m}, 
\begin{equation} \label{eq:u_vacuum_reflected}
  \postsub{u}{(r)}(x,t) = \begin{pmatrix} -\mathfrak g^1(x+t) \\ -\mathfrak g^2(x+t) \\ -\mathfrak g^2(x+t) \\ \mathfrak g^1(x+t) \end{pmatrix} , \qquad x<0,
\end{equation}
which is consistent with the Maxwell equations \eqref{eqs:Maxwell}. We shall promptly see that the two unknown functions $\mathfrak g^1$ and $\mathfrak g^2$ are obtained from the jump conditions \eqref{eqs:jump_conditions_E_H}. We are going to need the following expressions for $E_y$, $E_z$, $H_y$, $H_z$ in terms of $n^1$, $n^2$, $n^3$ and $n^4$, which are readily obtained from \eqref{eqs:m}, \eqref{eq:ui_in_terms_of_m}, using that, in vacuum, $E=D$ and $H=B$. We have, for $x<0$, 
\begin{equation} \label{eqs:E_H_in_terms_of_m_vacuum}
  \begin{aligned}
    E_y &= D_y = u^1 = \frac12(n^1-n^4) , &  E_z &= D_z = u^2 = \frac12(n^2-n^3) , \\
    H_y &= B_y = u^3 = -\frac12(n^2+n^3) , &  H_z &= B_z = u^4 = \frac12(n^1+n^4) .
  \end{aligned}
\end{equation}

Regarding the region inside the crystal, we consider the Riemann invariants $m=m(x,t)$ ($x>0$) representing the transmitted wave. Since $m^3$ and $m^4$ are constant along characteristics with negative speed (see \eqref{eq:m_char_speeds}), and assuming that the incident wave from above is the only source, it is clear that they must vanish, as every characteristic $\mathcal C_3$ or $\mathcal C_4$ intersects points in the trivial region. We thus have 
\begin{equation*}
  m(x,t) = \begin{pmatrix} m^1(x,t) \\ m^2(x,t) \\ 0 \\ 0 \end{pmatrix} , \qquad x>0 ,
\end{equation*}
and observe that, by \eqref{eq:m_pde}, \eqref{eq:m_char_speeds}, the characteristics in the crystal region are all straight lines. We are going to use \eqref{eqs:jump_conditions_E_H} to obtain $m^1$ and $m^2$. Before doing so, however, we need to express $m^1$, $m^2$, $m^3$ and $m^4$ in terms of $E_y$, $E_z$, $H_y$ and $H_z$, inside the crystal region $x>0$. We have that 
\begin{equation*}
  u^1 = D_y = \frac{\sqrt{K_1^2+12C_{111}E_y}-K_1}{6C_{111}} \qquad\text{ and }\qquad u^2 = D_z = \frac{\sqrt{K_2^2+12C_{222}E_z}-K_2}{6C_{222}} 
\end{equation*}
are the unique solutions of \eqref{eqs:E_in_terms_of_D} that vanish when $E_y$, respectively $E_z$, vanish. We then obtain from \eqref{eqs:m}, \eqref{eq:int_sqrt_di}, and using $H=B$, that, for $x>0$, 
\begin{equation} \label{eqs:m_in_terms_of_E_H}
  \begin{aligned}
    m^1 &= u^4 + \frac{(K_1+6C_{111}u^1)^{\frac32}-K_1^{\frac32}}{9C_{111}} = H_z + \frac{(K_1^2+12C_{111}E_y)^{\frac34}-K_1^{\frac32}}{9C_{111}} , \\
    m^2 &= -u^3 + \frac{(K_2+6C_{222}u^2)^{\frac32}-K_2^{\frac32}}{9C_{222}} = -H_y + \frac{(K_2^2+12C_{222}E_z)^{\frac34}-K_2^{\frac32}}{9C_{222}} , \\
    m^3 &= -u^3 - \frac{(K_2+6C_{222}u^2)^{\frac32}-K_2^{\frac32}}{9C_{222}} = -H_y - \frac{(K_2^2+12C_{222}E_z)^{\frac34}-K_2^{\frac32}}{9C_{222}} , \\
    m^4 &= u^4 - \frac{(K_1+6C_{111}u^1)^{\frac32}-K_1^{\frac32}}{9C_{111}} = H_z - \frac{(K_1^2+12C_{111}E_y)^{\frac34}-K_1^{\frac32}}{9C_{111}} .
  \end{aligned}
\end{equation}
We are now ready to use \eqref{eqs:jump_conditions_E_H} to solve for $m^1$ and $m^2$, as well as for $n^3=\postsub{n}{(r)}^3=2\mathfrak g^2$ and $n^4=\postsub{n}{(r)}^4=2\mathfrak g^1$ from above.

For $i\in\{1,\ldots,4\}$, let 
\begin{equation*}
  n_0(t) = \lim_{x\nearrow0} n(x,t) = \begin{pmatrix} 2\mathfrak f^1(x_0-t) \\ 2\mathfrak f^2(x_0-t) \\ 2\mathfrak g^2(t) \\ 2\mathfrak g^1(t) \end{pmatrix} , \qquad m_0(t) = \lim_{x\searrow0} m(x,t) = \begin{pmatrix} \lim\limits_{x\searrow0}m^1(x,t) \\ \lim\limits_{x\searrow0}m^2(x,t) \\ 0 \\ 0 \end{pmatrix} .
\end{equation*}
Using \eqref{eqs:E_H_in_terms_of_m_vacuum}, \eqref{eqs:m_in_terms_of_E_H} and \eqref{eqs:jump_conditions_E_H}, we obtain 
\begin{equation*}
  \begin{aligned}
    m_0^1 &= \frac{n_0^1+n_0^4}2 + \frac{\bigl(K_1^2+12C_{111}\frac{n_0^1-n_0^4}2\bigr)^{\frac34}-K_1^{\frac32}}{9C_{111}} , \\
    m_0^2 &= \frac{n_0^2+n_0^3}2 + \frac{\bigl(K_2^2+12C_{222}\frac{n_0^2-n_0^3}2\bigr)^{\frac34}-K_2^{\frac32}}{9C_{222}} , \\
    m_0^3 &= \frac{n_0^2+n_0^3}2 - \frac{\bigl(K_2^2+12C_{222}\frac{n_0^2-n_0^3}2\bigr)^{\frac34}-K_2^{\frac32}}{9C_{222}} , \\
    m_0^4 &= \frac{n_0^1+n_0^4}2 - \frac{\bigl(K_1^2+12C_{111}\frac{n_0^1-n_0^4}2\bigr)^{\frac34}-K_1^{\frac32}}{9C_{111}} .
  \end{aligned}
\end{equation*}
Using \eqref{eq:m_vacuum_incident}, \eqref{eq:m_vacuum_reflected}, together with $m^3\equiv 0$ and $m^4\equiv 0$, this reads 
\begin{equation*}
  \begin{aligned}
    m_0^1(t) &= \widetilde{\mathfrak f}^1(t) + \mathfrak g^1(t) + \frac{\Bigl(K_1^2+12C_{111}\bigl( \widetilde{\mathfrak f}^1(t) - \mathfrak g^1(t) \bigr)\Bigr)^{\frac34}-K_1^{\frac32}}{9C_{111}} , \\
    m_0^2(t) &= \widetilde{\mathfrak f}^2(t) + \mathfrak g^2(t) + \frac{\Bigl(K_2^2+12C_{222}\bigl( \widetilde{\mathfrak f}^2(t) - \mathfrak g^2(t) \bigr)\Bigr)^{\frac34}-K_2^{\frac32}}{9C_{222}} , \\
    0        &= \widetilde{\mathfrak f}^2(t) + \mathfrak g^2(t) - \frac{\Bigl(K_2^2+12C_{222}\bigl( \widetilde{\mathfrak f}^2(t) - \mathfrak g^2(t) \bigr)\Bigr)^{\frac34}-K_2^{\frac32}}{9C_{222}} , \\
    0        &= \widetilde{\mathfrak f}^1(t) + \mathfrak g^1(t) - \frac{\Bigl(K_1^2+12C_{111}\bigl( \widetilde{\mathfrak f}^1(t) - \mathfrak g^1(t) \bigr)\Bigr)^{\frac34}-K_1^{\frac32}}{9C_{111}} ,
  \end{aligned}
\end{equation*}
where 
\begin{equation} \label{eq:widetildemathfrakf}
  \widetilde{\mathfrak f}(t)=\mathfrak f(x_0-t) .
\end{equation}
Inserting the last two lines into the first two, we obtain the following system for $m_0^1$, $m_0^2$, $\mathfrak g^1$ and $\mathfrak g^2$, 
\begin{equation} \label{eqs:decoupled_case_system}
  \begin{aligned}
    m_0^i &= 2\widetilde{\mathfrak f}^i + 2\mathfrak g^i , \\
    K_i^{\frac32}+9C_{iii}(\widetilde{\mathfrak f}^i+\mathfrak g^i) &= \bigl(K_i^2+12C_{iii}(\widetilde{\mathfrak f}^i-\mathfrak g^i)\bigr)^{\frac34} ,
  \end{aligned} \qquad i\in\{1,2\} .
\end{equation}
In order to solve, for each $i\in\{1,2\}$, the second equation, we set 
\begin{equation*}
  Z_i = \bigl(K_i^2+12C_{iii}(\widetilde{\mathfrak f}^i-\mathfrak g^i)\bigr)^{\frac14} , 
\end{equation*}
and notice that $Z_i$ solves 
\begin{align*}
  3Z_i^4+4Z_i^3 &= 3\bigl(K_i^2+12C_{iii}(\widetilde{\mathfrak f}^i-\mathfrak g^i)\bigr) + 4\bigl(K_i^{\frac32}+9C_{iii}(\widetilde{\mathfrak f}^i+\mathfrak g^i)\bigr) \\
    &= 3K_i^2+4K_i^{\frac32}+72C_{iii}\widetilde{\mathfrak f}^i ,
\end{align*}
which is independent of $\mathfrak g^i$. Denoting the right--hand side by $\mathfrak c_i$, i.e., 
\begin{equation} \label{eq:mathfrak_ci}
  \mathfrak c_i = 3K_i^2+4K_i^{\frac32}+72C_{iii}\widetilde{\mathfrak f}^i ,
\end{equation}
we see from the second equation in \eqref{eqs:decoupled_case_system} that 
\begin{equation} \label{eq:mathfrakgi_solution}
  \mathfrak g^i = \frac{ Z_{i,0}^3 - K_i^{\frac32} }{9C_{iii}} - \widetilde{\mathfrak f}^i ,
\end{equation}
where $Z_{i,0}$ is the unique positive solution of $3Z^4+4Z^3=\mathfrak c_i$, which exists for any given $\mathfrak c_i>0$ (i.e., for small enough $\abs{\widetilde{\mathfrak f}^i}$), since $3Z^4+4Z^3$ is an increasing function of $Z$ for positive $Z$. Notice that we have $Z_{i,0}=\sqrt{K_i}$ when $\widetilde{\mathfrak f}^i=0$. %
Using $u=\postsub{u}{(i)}+\postsub{u}{(r)}$ when $x<0$, where $\postsub{u}{(i)}$ and $\postsub{u}{(r)}$ are given by \eqref{eq:u_vacuum_incident} and \eqref{eq:u_vacuum_reflected}, respectively, we have thus obtained an explicit solution in the vacuum region in terms of the initial profile $\mathfrak f$. Regarding then the crystal region, we obtain $m^i(x,t)$, $i\in\{1,2\}$, $x>0$, from $m^i(x,t)=m_0^i(t_0)$, where $m_0^i$ is given by the first line of \eqref{eqs:decoupled_case_system}, inserting the solution \eqref{eq:mathfrakgi_solution}, and $t_0$ is given implicitly, owing to \eqref{eq:m_char_speeds}, by the equation 
\begin{equation} \label{eq:widehatt_implicit}
  x = \Bigl( K_i^{\frac32}+\frac{9C_{iii}}2m_0^i(t_0) \Bigr)^{\frac13}(t-t_0) .
\end{equation}
This concludes the discussion of the decoupled case.

\subsection{\texorpdfstring{A few words about the general case $c\not\equiv0$}{A few words about the general case c!=0}} \label{subsec:general_case}

We conclude this section with a few remarks on the generic case, where $c\neq0$ and no Riemann invariants can be found. Obviously, the vacuum region can be treated in a similar way as in the decoupled case, the incident wave being represented by \eqref{eq:u_vacuum_incident}, and the ansatz \eqref{eq:m_vacuum_reflected} for the reflected wave remaining valid, and leading to \eqref{eq:u_vacuum_reflected}. Also, \eqref{eqs:E_H_in_terms_of_m_vacuum} are to be used when applying the jump conditions \eqref{eqs:jump_conditions_E_H}. However, the crystal region cannot be dealt with as easily. Having no Riemann invariants at our disposal, we have to appeal to the evolution equations along the characteristics of the components $w^i=\postsup{e}{\star}^i\frac{\partial}{\partial x}u$, $i\in\{1,2,3,4\}$, of the spatial derivative of the solution $u$, as in the main body of the paper. Note from \eqref{eq:lambda} that 
\begin{equation*}
  \lambda_1 = \sqrt{m+R} ,\ \lambda_2 = \sqrt{m-R} ,\ \lambda_3 = -\sqrt{m-R} ,\ \lambda_4 = -\sqrt{m+R} ,
\end{equation*}
so $\lambda_1,\lambda_2$ are positive and bounded away from zero, while $\lambda_3,\lambda_4$ are negative and, again, bounded away from zero. The solution is trivial for $x>K_1t$, the line $x=K_1t$ being a $\mathcal C_1$ characteristic in the crystal. A $\mathcal C_1$ or a $\mathcal C_2$ characteristic originating at a point in the non--trivial region $x<K_1t$ in the crystal intersects in the past the material boundary $\{x=0\}$. Consequently, $w^1$ and $w^2$ must be given boundary conditions on $\{x=0\}$. On the other hand, a $\mathcal C_3$ or a $\mathcal C_4$ characteristic originating at a point in the non--trivial region intersects in the past the line $x=K_1t$ which constitutes the boundary of the trivial region in the crystal. As a consequence, $w^3$ and $w^4$ satisfy trivial initial conditions on $x=K_1t$. To deduce the boundary conditions for $w^1$ and $w^2$ on the material boundary, we impose the jump conditions \eqref{eqs:jump_conditions_E_H}, which, since the vectorfield $\frac{\partial}{\partial t}$ is tangential to $\{x=0\}$, imply similar conditions for the time derivatives: 
\begin{equation*}
  \Bigl[\frac{\partial}{\partial t}E_y\Bigr] = 0 , \quad \Bigl[\frac{\partial}{\partial t}E_z\Bigr] = 0 , \quad \Bigl[\frac{\partial}{\partial t}H_y\Bigr] = 0 , \quad \Bigl[\frac{\partial}{\partial t}H_z\Bigr] = 0 .
\end{equation*}
In applying the first two of these, not only must we solve the system \eqref{eqs:E_in_terms_of_D} around the solution 
\begin{equation*}
  (D_y,D_z)\bigr|_{(E_y,E_z)=(0,0)} = (0,0) ,
\end{equation*}
but also its derivative with respect to $t$, namely the linear system 
\begin{equation*}
  \left\{ \begin{aligned}
    \frac{\partial}{\partial t}E_y &= K_1\frac{\partial}{\partial t}D_y + 6C_{111}D_y\frac{\partial}{\partial t}D_y + 2C_{112}\Bigl( D_z\frac{\partial}{\partial t}D_y + D_y\frac{\partial}{\partial t}D_z \Bigr) + 2C_{122}D_z\frac{\partial}{\partial t}D_z , \\
    \frac{\partial}{\partial t}E_z &= K_2\frac{\partial}{\partial t}D_z + 6C_{222}D_z\frac{\partial}{\partial t}D_z + 2C_{122}\Bigl( D_y\frac{\partial}{\partial t}D_z + D_z\frac{\partial}{\partial t}D_y \Bigr) + 2C_{112}D_y\frac{\partial}{\partial t}D_y .
  \end{aligned} \right.
\end{equation*}
As in the decoupled case, the jump conditions not only provide the boundary conditions for $w^1$ and $w^2$, but also give the boundary conditions for the two components ($\mathfrak g^1$ and $\mathfrak g^2$) of the reflected wave. Here, we use the fact that, by virtue of the equation \eqref{eq:theory_applies_main_pde}, the components $\postsup{e}{\star}^i\frac{\partial}{\partial t}u$ of the time derivative of the solution $u$ are given by:
\begin{equation*}
  \postsup{e}{\star}^i\frac{\partial}{\partial t}u = - \lambda_i w^i.
\end{equation*}
A key difference from the decoupled situation, aside from the characteristics inside the crystal not being necessarily straight lines, is that we expect the reflected wave to extend to all $t>0$ at the crystal boundary (hence to all $t>-x$ in the vacuum region). This is due to the fact that, inside the crystal, $w^3$ and $w^4$ will be non--zero throughout the non--trivial region $x<K_1t$, generating non--zero boundary conditions at the material boundary for the reflected wave for all $t>0$. This is because (see equations \eqref{eq:dsiwi_fully_expanded} of Subsection~\ref{subsec:1storder_evoleqs}) the equations for $\frac{\partial}{\partial s_3}w^3$ and $\frac{\partial}{\partial s_4}w^4$ contain, respectively, the terms $2\cull{\gamma 312}w^1w^2$ and $2\cull{\gamma 412}w^1w^2$ which are inhomogeneous from the perspective of the variables $(w^3,w^4)$.

\bibliographystyle{amsalpha}
\bibliography{fosepwnlc}

\end{document}